\documentclass[a4paper,10pt]{article}
\usepackage{indentfirst}
\usepackage{a4}
\usepackage{amsmath, amsfonts, amssymb, amsthm, amscd, bezier}
\usepackage[utf8]{inputenc}
\usepackage[brazil,english]{babel}
\usepackage{amsfonts}
\usepackage{graphicx}
\usepackage{float}
\usepackage[all]{xy}
\usepackage{color}   
\usepackage{fullpage}
\usepackage{cite}
\advance\oddsidemargin by -1cm \advance\textwidth by 1cm
\usepackage{mathrsfs}

\usepackage{hyphenat}
\usepackage{multicol}
\newtheorem{theorem}{Theorem}[section]
\newtheorem{definition}{Definition}[section]
\newtheorem{proposition}{Proposition}[section]
\newtheorem{corollary}{Corollary}[section]
\newtheorem{lemma}{Lemma}[section]
\theoremstyle{definition}

\newtheorem{example}{Example}[section]
\usepackage[mathscr]{euscript}
\usepackage{graphicx}
\usepackage{subfig}

\def\href{}

\DeclareMathOperator{\Ker}{Ker}

\begin{document}

\begin{center}

{ \bf \LARGE   Homotopical  Cancellation Theory for  \\[10pt] Gutierrez-Sotomayor Singular Flows}



\vspace{0.6cm}
\large 
D.V.S. Lima\footnote{Supported by FAPESP under grants 2014/11943-6 and 2015/10930-0.}  
\hspace{0.3cm}
S. A. Raminelli\footnote{Partially supported by CNPq under grant 140712/2016-0 and by CAPES under grant 1185783.}
\hspace{0.3cm}
K. A. de Rezende\footnote{Partially supported by CNPq under grant 305649/2018-3}

\end{center}
\vspace{0.3cm}

\begin{abstract}

In this article, we present a dynamical homotopical cancellation theory for Gutierrez-Sotomayor singular flows  $\varphi$, GS-flows,  on singular surfaces $M$. This theory  generalizes  the classical theory of Morse complexes of smooth dynamical systems together with the corresponding cancellation theory for non-degenerate singularities. This is accomplished by defining 
a GS-chain complex for  $(M,\varphi)$ and computing its spectral sequence $(E^r,d^r)$.  As $r$ increases, algebraic cancellations occur, causing  modules in  $E^r$ to become trivial. The main theorems herein relate these algebraic cancellations within the spectral sequence to a family $\{M_r,\varphi_r\}$ of GS-flows $\varphi_r $ on singular surfaces $M_r$, all of which have the same homotopy type as $M$.
The surprising element in these results is that the dynamical homotopical cancellation of GS-singularities of the flows $\varphi_r$ are in consonance with the algebraic cancellation of the modules in $E^r$ of its associated spectral sequence. Also, the convergence of the spectral sequence corresponds to a GS-flow $\varphi_{\bar{r}}$ on $M_{\bar{r}}$, for some $\bar{r}$, with the property that $\varphi_{\bar{r}}$ admits no further dynamical homotopical cancellation of GS-singularities.

\end{abstract}
\vspace{0.6cm}



\section{Introduction}

In~\cite{GS}, Gutierrez and Sotomayor presented simple singularities, as well as their characterization and genericity theorems for $C^1$-structurally stable vector fields tangent to a 2-dimensional compact subset $M$ of $\mathbb{R}^k$. 

For the first time, in~\cite{dRM}, the flows associated to these vector fields, with no periodic orbits or limit cycles, were studied using Conley index theory and named {\it Gutierrez-Sotomayor flows}, GS-flows for short. The simple singularities presented in~\cite{GS}, namely,  regular  ($\mathcal{R}$), cone ($\mathcal{C}$), Whitney ($\mathcal{W}$), double crossing ($\mathcal{D}$) and triple crossing  ($\mathcal{T}$) singularities were called GS-singularities.  In~\cite{dRM}, the Conley index of each GS-singularity was computed. The existence of Lyapunov functions for GS-flows was established and a GS-handle theory was introduced in order to construct isolating blocks for each GS-singularity. 

In this work, we take this analysis a step further, by analyzing global GS-flows on singular closed surfaces. Our goal is to  investigate the global connections of flow lines of GS-flows under a spectral sequence analysis of a chain complex associated to it. This method was successful in~\cite{BLMdRS1, BLMdRS2, LMdRS} in order to obtaining cancellation theorems in smooth settings, gradient flows of Morse functions, as well as, for circle-valued Morse functions.

However, it is a great challenge to adapt the smooth theory to the singular setting, more specifically for GS-flows.  One wishes to maintain  the principles that undergird the former setting in the latter. In order for the theory to retain its basic structure and be a valid generalization, the definitions and postulates of the singular setting  must encompass the definitions  and postulates of the smooth setting. Hence, one must face the problem of defining intersection numbers in the absence of differentiability, as well as  defining a chain complex generated by  GS-singularities.

Furthermore, a generalized notion of cancellation must be presented for GS-singularities. This will be captured by defining a dynamical  homotopical cancellation which is a generalization of the classical notion of cancellation in the smooth case  as Figure \ref{fig:exemplo_morse_intro2} suggests. 
In a classical cancellation, the manifold before and after the cancellation are always  homeomorphic.  In a homotopical cancellation, the singular manifold before and after the cancellation are of the same homotopy type  and may not be homeomorphic. 
Roughly, the idea behind a  dynamical homotopical cancellation  is to consider  a set of three  singularities  $x$, $x'$ and $y$ and the flow lines $u,u'$ joining them in a neighborhood $U$ which, through a homotopy will be taken to a neighborhood $\bar{U}$ containing  a GS-singularity $\overline{x}'$.  The regions $U$ and $\bar{U}$ are of the same homotopy  type and this homotopy respects the number of singular regions (droplets and folds) that exist in $U$.

 \begin{figure}[h!]
 \centering
 \includegraphics[scale=0.75]{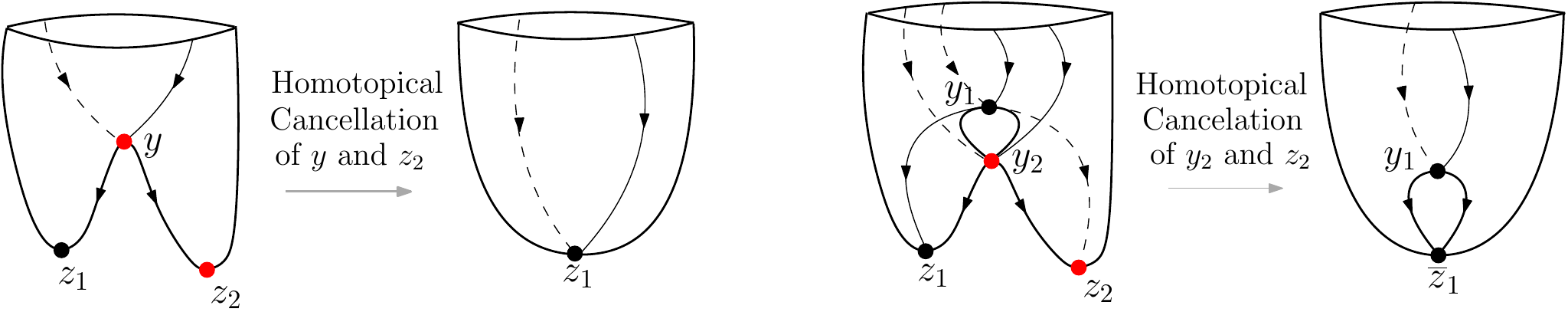}
 \caption{ Dynamical homotopical cancellations in the smooth  and singular cases.}
 \label{fig:exemplo_morse_intro2}
 \end{figure}

In order for these homotopies to be well defined, we will consider a  larger class of GS-singularities which include $n$-sheet cone, Whitney, double and triple attractors and repellers. For simplicity, henceforth, we will continue to refer to  these as GS-singularities for GS-flows.
 Since these more general GS-flows have not been previously considered in the literature, these fundamental concepts have to be established herein in order to get the theory off the ground.
This in itself is already quite a formidable endeavor, since the GS-singularities comprise a large class of different singularity types, which must be dealt with in a case by case analysis.


The main contribution of this work is that, with the introduction of a generalization of these concepts, several homotopical cancellation theorems for GS-flows are proven.
In our opinion, what is most striking in these theorems, is that the dynamical homotopical cancellations within the flow occur in consonance with the algebraic cancellations of the unfolding, i.e., with the turning of the pages, of the associated spectral sequence.
In order to appreciate the beauty of these results, we finalize this paper with three examples from the realms of flows with cone, Whitney and double crossing singularities. See Section~\ref{DDCTSS}.


This paper is organized as follows. Section~\ref{sec:GSVF} is an introduction to GS-vector fields and their associated flows. In order to define a Gutierrez-Sotomayor chain complex, we need to establish a regularization process of GS-singularities, referred to as its Morsification, which is presented in Section~\ref{sec:MGSVFIB}. In Section~\ref{sec:GSCC}, we make use of the Morsification process to introduce GS-intersection numbers and hence obtain a differential for a chain complex generated by the GS-singularities, which we refer to as a GS-chain complex.
Next, we prove  local dynamical homotopical cancellation  theorems for GS-singularities in Section~\ref{DCTGSF}.
Moreover, in Section~\ref{DDCTSS}, we generalize the theory developed in~\cite{BLMdRS1, BLMdRS2} to obtain global  homotopical  cancellation theorems for flows on singular surfaces with $\mathcal{R}$, $\mathcal{C}$, $\mathcal{D}$, $\mathcal{W}$ or $\mathcal{T}$ singularities. This is accomplished by associating the algebraic cancellations that occur in a spectral sequence of a filtered GS-chain complex of a GS-flow with the dynamical homotopical  cancellations that occur within the flow.
%
The  flow chart in Figure~\ref{fig:diagramafinal} provides an overview of the development of the results, and helps to understand the interrelationships among the sections.

 \begin{figure}[h!]
 \centering
  \includegraphics[scale=1.7]{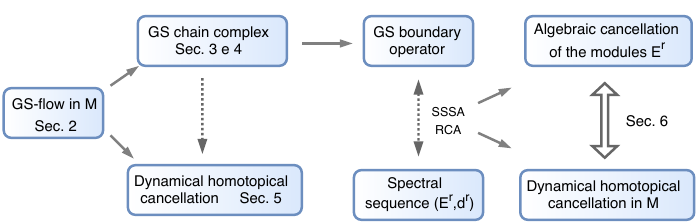}
 \caption{Flow chart: an overview of the development of the results herein.}
 \label{fig:diagramafinal}
 \end{figure}

\section{Gutierrez-Sotomayor Flows}\label{sec:GSVF}

\subsection{Gutierrez-Sotomayor Vector Fields}

In~\cite{GS}, Gutierrez and Sotomayor presented a characterization for manifolds with singularities where the degeneracy is restricted in order to admit  only those that appear in a stable manner.  This means that the regularity conditions in the definition of the smooth surfaces in $\mathbb{R}^3$, given in terms of implicit functions and immersions, are broken stably. Hence the following singularities arise.  

\begin{definition}
A subset $M\subset\mathbb{R}^l$ is called a
two-dimensional  \emph{\textbf{ \emph{manifold with simple singularities}}} 
if for every point $p\in M$ there are a neighbourhood $V_p$ of $p$ in $M$ and 
a local chart, a $C^{\infty}$-diffeomorphism $\Psi:V_p\rightarrow \mathcal{G}$ such that
$\Psi(p)=0$, where $\mathcal{G}$ is one of the following subsets of $\mathbb{R}^3$:

$\mathcal{R}=\{(x,y,z);z=0\}$, plane;

$\mathcal{C}=\{(x,y,z);z^2-y^2-x^2=0\}$, cone;

$\mathcal{W}=\{(x,y,z);zx^2-y^2=0\}$, Whitney's umbrella;

$\mathcal{D}=\{(x,y,z);xy=0\}$, double crossing;

$\mathcal{T}=\{(x,y,z);xyz=0\}$, triple crossing.
\end{definition}

We denote by $M(\mathcal{G})$ the set of points $p\in M$ such that
$\Psi(p)=0$ for a local  chart  $\Psi:V_p\rightarrow \mathcal{G}$, where 
$\mathcal{G}=\mathcal{R}$, $\mathcal{C}$, $\mathcal{D}$, $\mathcal{W}$ or $\mathcal{T}$. 
Thus $M(\mathcal{R})$ is a smooth two-dimensional manifold called the \emph{\textbf{regular part}} of $M$, 
$M(\mathcal{D})$ is a one-dimensional smooth manifold, 
while $M(\mathcal{C})$, $M(\mathcal{W})$ and $M(\mathcal{T})$ are discrete sets. The \emph{\textbf{singular part}} of $M$, $\mathscr{SP}(M)$, is the union of all non regular singularities and folds, i.e. the union $ M(\mathcal{C})\cup  M(\mathcal{W}) \cup M(\mathcal{D}) \cup M(\mathcal{T})$.   Also, the set $M$ endowed with the partition $\{M(\mathcal{G}), \mathcal{G}\}$ is a stratified set in the sense of Thom.

A vector field $X$ of class  $C^r$ on $\mathbb{R}^l$ is said to be {\it tangent} to a manifold $M\subset \mathbb{R}^l$ with simple singularities if it is tangent to the smooth submanifolds
$M(\mathcal{G})$, for all $\mathcal{G}$. The space of such vector fields is denoted by
$\mathfrak{X}^r(M)$ and it is  endowed with the $C^r$-compact open topology.

In~\cite{GS}, Gutierrez and Sotomayor characterized a  set of structurally stable vector fields $\Sigma^{r}(M)$ contained in $\mathfrak{X}^r(M)$. This set contains vector fields with finitely many hyperbolic singularities and periodic orbits,  as well as,  singular limit cycles with  no saddle connections and the additional property that 
the $\alpha$ and $\omega$-limit sets of a trajectory is either a singularity, a periodic orbit or a singular cycle.

In this paper,  we do not wish to consider  vector fields with  periodic orbits  or limit cycles. Hence, we will consider vector fields having only  the hyperbolic simple singularities as shown in Figure~\ref{hiperbolicos}.  

\begin{figure}[h!]
    \centering
           \includegraphics[width=0.6\textwidth]{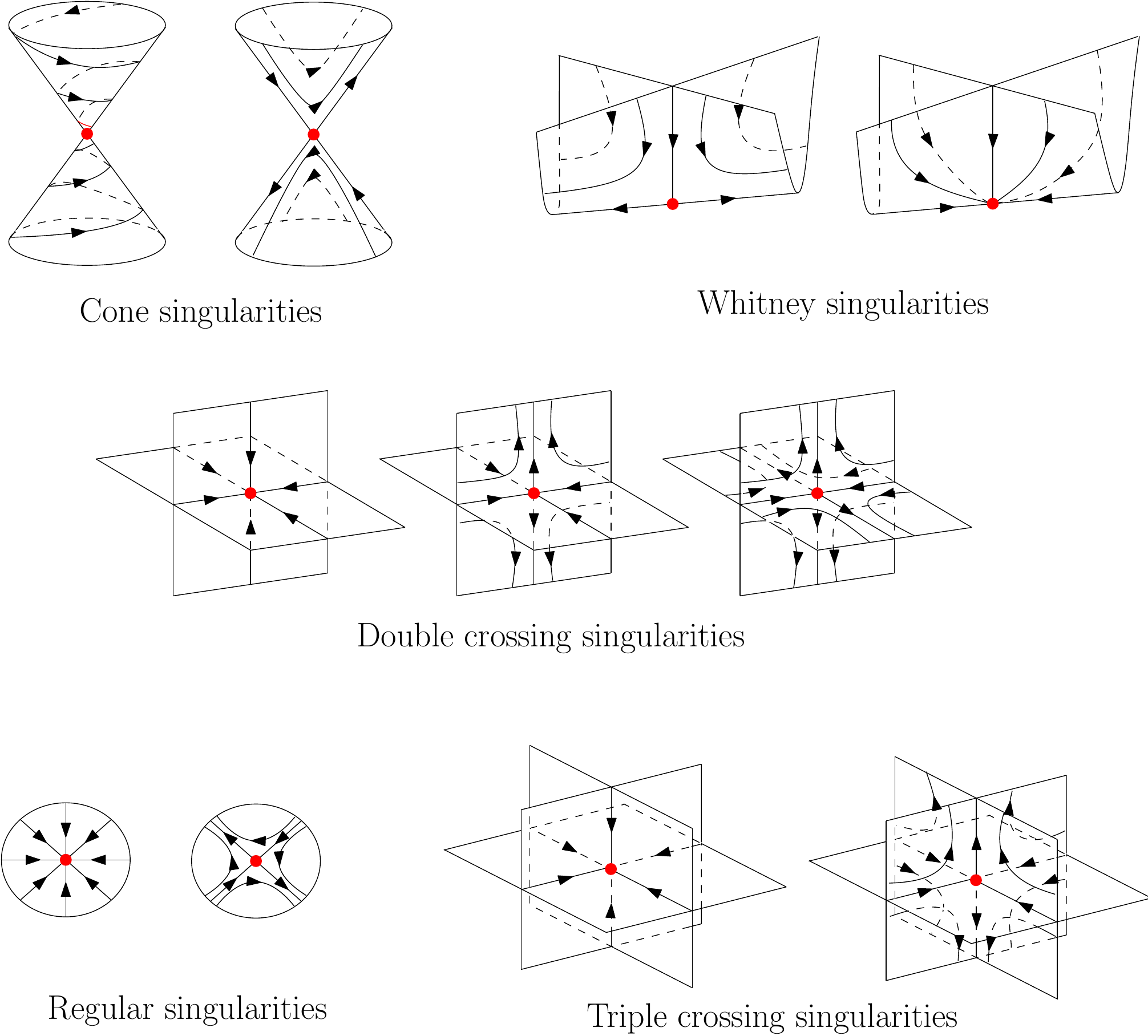}
    \caption{Local types of GS-singularities.}
    \label{hiperbolicos}
\end{figure}   

\subsection{Isolating blocks for GS-singularities}

Given a vector field $X\in \Sigma^{r}(M)$, we refer to the associated  flow as the  Gutierrez-Sotomayor flow  $\varphi_{X}$ on $M$, GS-flow  with GS-singularities for short. 

An \emph{\textbf{isolating block}} for a GS-singularities $p$ of a GS-flow $\varphi$ is an isolating neighborhood  $N \subset M$ of $p$ such that the exit set $N^- =\{x\in N; \varphi([0.T),x)\nsubseteq N, \forall T>0\}$
is closed.
The existence of  isolating blocks for GS-singularities is a consequence of the existence of  Lyapunov functions $f$ in a neighborhood of $p$. Hence,  if $f(p)=c$, let $\epsilon >0$ be such that  there are no critical values in $[c-\epsilon,c+\epsilon]$, then  the connected component  of  $f^{-1}([c-\epsilon,c+\epsilon])$ which  contains $p$, $N$,  is an isolating block  for  $p$.  Also, $N^- =f^{-1}(c-\epsilon)\cap N$.  It is worth noting that an isolating block  can also be defined for a maximal invariant set of a GS-flow. See~\cite{Con, Montufar, dRM} for more details. 

The next theorem characterizes the relation between the first Betti number of the boundary of an isolating block for the singularity  $p$, with the number of boundary components and the ranks of the homology Conley index. The proof can be found in~\cite{ Montufar, dRM}.

\begin{theorem}[Poincaré-Hopf equality]\label{PH-equality}
Let $(N,N^{-})$ be an index pair for a GS-singularity $p$ and $(h_0,h_1,h_2)$ be the ranks of the homology Conley index of
$p$. Then
\begin{equation}
(h_2-h_1+h_0)-(h_2-h_1+h_0)^{\ast}=e^+-\mathcal{B}^+-e^-+\mathcal{B}^-,
\end{equation}
where $^\ast$ indicates the index of the time-reversed flow, $e^+$ (resp., $e^-$) is
the number of entering (resp., exiting) boundary components of $N$ and
$\mathcal{B}^+=\sum_{k=1}^{e^+}b_k^+ $ (resp., $ 
(\mathcal{B}^-=\sum_{k=1}^{e^-}b_k^-)$, where $b_k^+(b_k^-)$ is the first  Betti number of the
$k$-th entering (resp., exiting) boundary components of $N$.
\end{theorem}

For each type of GS-singularity, we now define its nature  which corresponds to the local behavior of the flow on a chart around the singularity.


\begin{definition} Let $p$ be a  GS-singularity of  
\begin{enumerate} 
\item regular or  cone type, denote its \textbf{nature} by 
$a$ (resp., $r$) if $p$  is an attractor (resp., repeller); 
by $s$ if $p$ is a saddle. 
\item Whitney type, denote its \textbf{nature} by 
 $a$ (resp., $r$) if $p$ is an attractor (resp., repeller); 
 $s_s$ (resp., $s_u$) if $p$ is a saddle and its stable (resp., unstable) manifold is singular.
  \item  double crossing typem denote its \textbf{nature} by  
 $a^2$ (resp., $r^2$) if $p$ is an attractor (resp., repeller); 
  $sa$ (resp., $sr$) if $p$ is a  saddle  formed by a regular saddle and a regular attractor (resp., repeller); 
   $ss_s$ (resp., $ss_u$)  if $p$ is a saddle formed by two regular saddles which are identified along their stable (resp., unstable) manifolds.
  \item triple crossing type, denote its \textbf{nature} by  
 $a^3$ (resp., $r^3$) if $p$ is an attractor (resp., repeller); 
  $ssa$ (resp., $ssr$) if $p$ is a saddle  formed by two regular saddles and a regular attractor (resp., repeller). 
\end{enumerate}
\end{definition}

%
%
%

In~\cite{Montufar,dRM}, the construction  of  isolating blocks was undertaken for each GS-singularity according to their type   $\mathcal{C}$, $\mathcal{W}$, $\mathcal{D}$ or $\mathcal{T}$, and their nature according to the table below.

\begin{multicols}{2}

\noindent
\begin{tabular}{|c|c|c|c|c|}
\hline
type   &     nature      & $e_v^-$ &	$e_v^+$  & weight       \\ \hline \hline         
$\mathcal{D}$  &  $\textbf{$a^2$}$        &   0       &        1        &  $b_1^+=3$       \\  
\cline{2-5}  &   ${sa}$       &   1     &  1       &  $b_1^+=b_1^- + 2$            \\  
\cline{2-5}   &    $ {sa}$      &   2     &  1       &  $b_1^+=b_1^- + b_2^- + 1$          \\  
\cline{2-5}     &   $ss_s$    &  1     &  1       &  $b_1^+=b_1^- + 2$          \\  
\cline{2-5}  &   $ss_s$    &  1     &  2       &  $b_1^-=b_1^+ + b_2^+ - 3$     \\  
\cline{2-5} 
&   $ ss_s$    &  2     &  1       &  $b_1^+=b_1^- + b_2^- + 1$                     \\  \cline{2-5}       &   $ss_s$   &  2     &  2       &  $b_1^+ + b_2^+ = b_1^- + b_2^- + 2$               \\  
\cline{2-5}    &    $ss_s$   &  3     &  1       &  $b_1^+\!=\!b_1^- \!+\! b_2^- \!+\! b_3^-$      \\  
\cline{2-5}   &  $ss_s$     &  3     &  2       &  $b_1^+ \!+\! b_2^+ \!=\! b_1^- \!+\! b_2^- \!+\! b_3^- \!+\! 1$          \\  
\cline{2-5} &   $ss_s$    &  4     &  1       &  $b_1^+ \!=\! b_1^- \!+\! b_2^- \!+\! b_3^- \!+\! b_4^- \!-\! 1$    \\  \cline{2-5}      & $ss_s$    &   4     &  2       &  $b_1^+ \!+\! b_2^+ \!=\! b_1^- \!+\! b_2^- \!+\! b_3^- \!+\! b_4^-$   \\  
\cline{2-5}       & & & & Reversed  flow \\ \cline{2-5}    &   $ss_u$    &  2     &  4       &  $b_1^- \!+ \!b_2^- \!= \! b_1^+ \!+\! b_2^+ \!+\! b_3^+ \!+\! b_4^+$      \\  
\cline{2-5}     &   $ss_u$    &  1     &  4       &  $b_1^-\!=\!b_1^+ \!+\! b_2^+ \!+\! b_3^+ \!+\! b_4^+ \!-\! 1$         \\  
\cline{2-5}  &   $ss_u$     &  2     &  3       &  $b_1^- \!+\! b_2^- \!=\! b_1^+ \!+\! b_2^+ \!+\! b_3^+ \!+\! 1$     \\  
\cline{2-5}   &  $ss_u$    &  1     &  3       &  $b_1^-=b_1^+ + b_2^+ + b_3^+$     \\  \cline{2-5}   &   $ss_u$   &  2     &  2       &  $b_1^- \!+\! b_2^- \!=\! b_1^+ \!+\! b_2^+ \!+\! 2$     \\
\cline{2-5} &   $ss_u$    &  1     &  2       &  $b_1^-=b_1^+ + b_2^+ + 1$    \\  
\cline{2-5}     &   $ss_u$    &  2     &  1       &  $b_1^+=b_1^- + b_2^- - 3$     \\ 
\cline{2-5}  &   $ss_u$    &  1     &  1       &  $b_1^-=b_1^+ + 2$         \\  
\cline{2-5}     &   ${sr}$        &  1     &  2       &  $b_1^-=b_1^+ + b_2^+ + 1$  \\ 
\cline{2-5}  &   ${sr}$        &  1     &  1       &  $b_1^-=b_1^+ + 2$           \\  
\cline{2-5}  &    {$r^2$}    &   1     &  0       &  $b_1^-=3$                                                      
\\   \hline 
\end{tabular}

 \begin{tabular}{|c|c|c|c|c|}
\hline
type  &     nature       & $e_v^-$ &	$e_v^+$  & weight           \\ \hline \hline
$\mathcal{R}$  &  $ {a}$       &     0       &        1        &  $b_1^+=1$       \\  
\cline{2-5}   &   $ {s}$      &      1       &        1        &  $b_1^-=b_1^+$    \\  \cline{2-5}       &   $ {s}$      &      1       &        2        &  $b_1^-=b_1^+ + b_2^+ - 1$     \\ \cline{2-5}      &   $ {s}$      &      2       &        1        &  $b_1^+=b_1^- + b_2^- - 1$      \\  
\cline{2-5} &   $ {r}$      &      1       &        0        &  $b_1^-=1$             \\ \hline 
$\mathcal{C}$  &  $ {a}$      &      0       &        2        &  $b_1^+=b_2^+=1$    \\  \cline{2-5}   &   $ {s}$      &     1       &        1        &  $b_1^-=b_1^+$         \\  \cline{2-5}  &   $ {s}$      &     2       &        2        &  $b_1^-+b_2^-=b_1^+ + b_2^+$    \\  
\cline{2-5} &   $ {r}$       &    2       &        0        &  $b_1^-=b_2^-=1$       \\    \hline 
$\mathcal{W}$  &  $ {a}$       &    0       &        1        &  $b_1^+=2$    \\   
\cline{2-5}         &   $s_s$    &   1       &        1       &  $b_1^+=b_1^- + 1$                   \\    \cline{2-5}         &   $s_s$   &   2       &        1       &  $b_1^+=b_1^- + b_2^-$    \\    \cline{2-5} & & & & Reversed  flow \\ \cline{2-5}         &   $s_u$    &   1       &        2       &  $b_1^-=b_1^+ + b_2^+$      \\   
\cline{2-5}        &   $s_u$   &   1       &        1       &  $b_1^-=b_1^+ + 1$                    \\   
\cline{2-5}  &   $ {r}$        &   1       &        0        &  $b_1^-=2$                                \\    \hline 
$\mathcal{T}$  &  $a^3$     &   0     &     1       &  $b_1^+=7$                                 \\  
\cline{2-5}            &   $ {ssa}$    &   1     &     1       &  $b_1^+=b_1^- + 2$                 \\  \cline{2-5}           &   $ {ssa}$    &   2     &     1       &  $b_1^+=b_1^- + b_2^- + 1$    \\  \cline{2-5} & & & & Reversed  flow \\ \cline{2-5}      &   $ {ssr}$    &   1     &     2       &  $b_1^-=b_1^+ + b_2^+ + 1$    \\   \cline{2-5}          &  $ {ssr}$     &   1     &     1       &  $b_1^-=b_1^+ + 2$             \\ 
\cline{2-5}       &   $r^3$       &   1     &     0       &  $b_1^-=7$                                \\                      
\hline
\end{tabular}

\end{multicols}

\subsection{Super Attractors and Repellers }


In this work, we will study  homotopical cancellations within an isolating block  containing the maximal invariant set  of three GS-singularities, one saddle and two attractors (resp. repellers), and their connecting orbits. This homotopy produces a super attractor (resp., super  repeller) singularity.
 See Figure~\ref{fig:exemplo_cone_intro2}.


\begin{figure}[h!]
\centering
\includegraphics[scale=0.5]{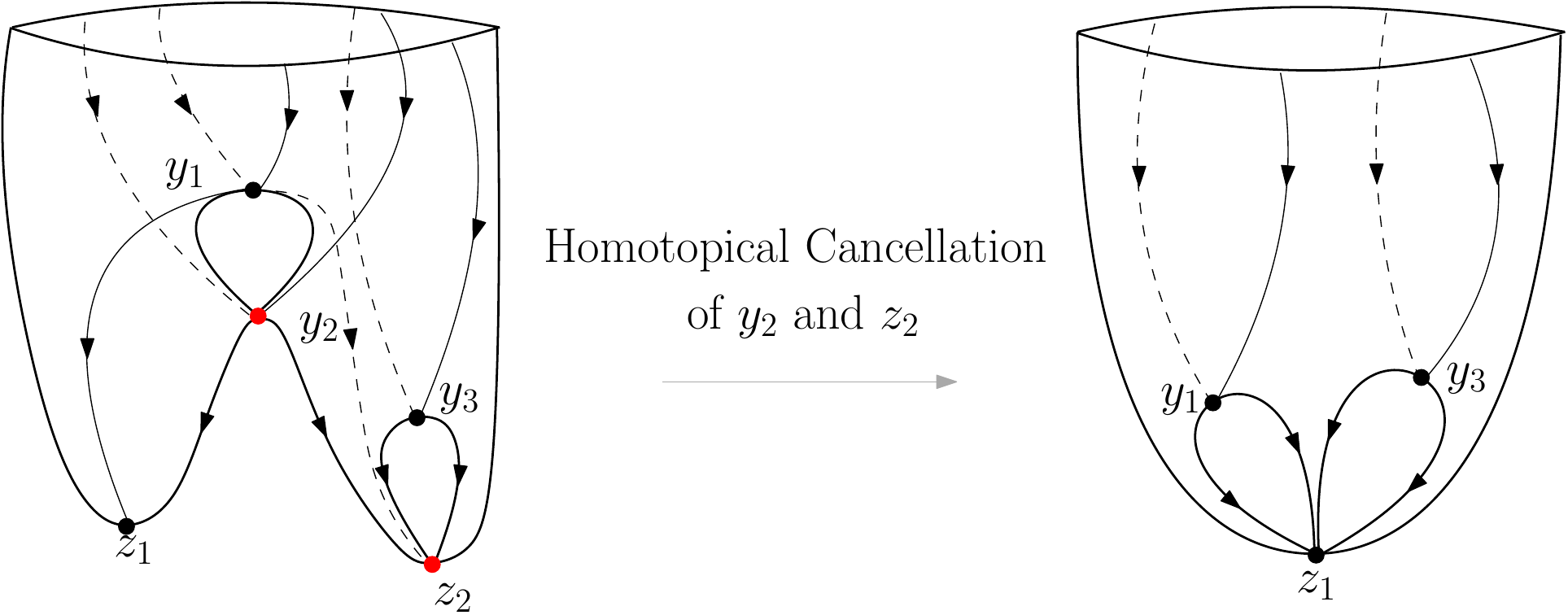}
\caption{Homotopical cancellation of a saddle cone and an attracting cone singularities.}
\label{fig:exemplo_cone_intro2}
\end{figure}

Let $D^A\subseteq\mathbb{R}^2$ ($D^R\subseteq\mathbb{R}^2$) be the unit disc of center $p=(0,0)$ and $X$ the attracting radial (resp., repelling) vector field on $D$ with attracting (resp., repelling) singularity $p$.

\begin{definition}
A generalized GS-singularity $p$ is:

\begin{enumerate}
\item a  \emph{\textbf{ \emph{super attractor}}}  (resp.,  \emph{\textbf{ \emph{super repeller}}}) of type:
\begin{enumerate}
    \item $n$-sheet cone  of attracting  (resp. repelling) nature when obtained by identifying the center points $p_i$ of $n$ discs $D_i, i=2,...,n$, where $D_i$ has defined on it an  attracting (resp., repelling) radial vector field. 
    \item $n$-sheet Whitney  of attracting  (resp., repelling) nature when obtained by identifying the center points $p_i$ and some radii of the $n$ discs $D^A_i$ (resp., $D^R_i$), $i=1,..., n$, where $D^A_i$ (resp., $D^R_i$) has defined on it an attracting (resp. repelling) radial vector field. Moreover, $n-2$ discs $D^A_i$ (resp., $D^R_i$) have the property that exactly  two radii are identified to raddi of two distinct discs. The remaining discs have the property that exactly one radius is identified to a radius of another disc. See Figure~\ref{GS-geral}.
    \item $n$-sheet double crossing of attracting  (resp. repelling) nature, $n=2,3, \ldots$, when obtained by identifying  the center points $p_i$ of $n$ discs $D^A_i,i=0,...,n-1$, where each $D^A_i$ (resp., $D^R_i$ ) is defined as above. Moreover, we identify exactly one diameter of each disc $D_i, i=1,...,n-1$ to  distinct diameters  $d_{i}$ of the  disc $D_0$, i.e. $D_i\cap D_j \setminus \{p\}=\emptyset, $  and $D_ {i} \cap D_0 = d_{i} $, $ i\neq j, i,j=1,\ldots n-1$. See Figure~\ref{GS-geral}.
    \item $n$-sheet triple crossing  of attracting  (resp., repelling) nature, $n=2k+1$, when obtained by identifying the center points $p_i$ of $n$ discs $D_0,D_i^1, D_i^2, i=1,...,n$, where each disc is defined as above. Moreover, consider the sets of distinct diameters  $\{d_{0,i}^1,d_{0,i}^2\}$ in $D_0$, $\{d_i^1, \partial_i^1\}$ in $D_i^1$ and $\{d_i^2,\partial_i^2\}$ in $D_i^2$, $i=1,\ldots, n$. We identify the diameters $\partial_i^1$ and $\partial_i^2$, the diameters $d_i^1$ and $d_{0,i}^1$, and the diameters $d_i^2$ and $d_{0,i}^2$, so that  all  of discs $D_i^1,D_i^2$ are pairwise disjoint, i.e., $(D_i^1\cup D_i^2) \cap (D_j^1\cup D_j^2) =\emptyset, i\neq j, j=1,\ldots n$. See Figure~\ref{GS-geral}.
\end{enumerate}
    \item  a  \emph{\textbf{ \emph{$\mathcal{C}$-type}}} (resp.,  \emph{\textbf{ \emph{$\mathcal{W}$,$\mathcal{D}$, $\mathcal{T}$-type}}})  singularity of saddle nature if it is a $\mathcal{C}$-type (resp., $\mathcal{W}$,$\mathcal{D}$, $\mathcal{T}$-type)  GS-singularity of saddle nature.
\end{enumerate}
\end{definition}

Given an  $n$-sheet generalized GS-singularity $p$, define the   \emph{\textbf{ {{singularity type number}}}} $m(p)$ of $p$ as 
$n-1$ if $p$ is of  $\mathcal{C}$-type or $\mathcal{D}$-type; $n$ if $p$ is of $\mathcal{W}$-type; $k$ if $p$ is of  $\mathcal{T}$-type, where $n=2k+1$.
 Note that a $\mathcal{C}$-type (resp., $\mathcal{W}$,$\mathcal{D}$, $\mathcal{T}$-type)   singularity of saddle nature has type number equal to $1$.  Also, a regular singularity always has type number equal to zero.

\begin{figure}[H]
    \centering
        \includegraphics[width=0.85\textwidth]{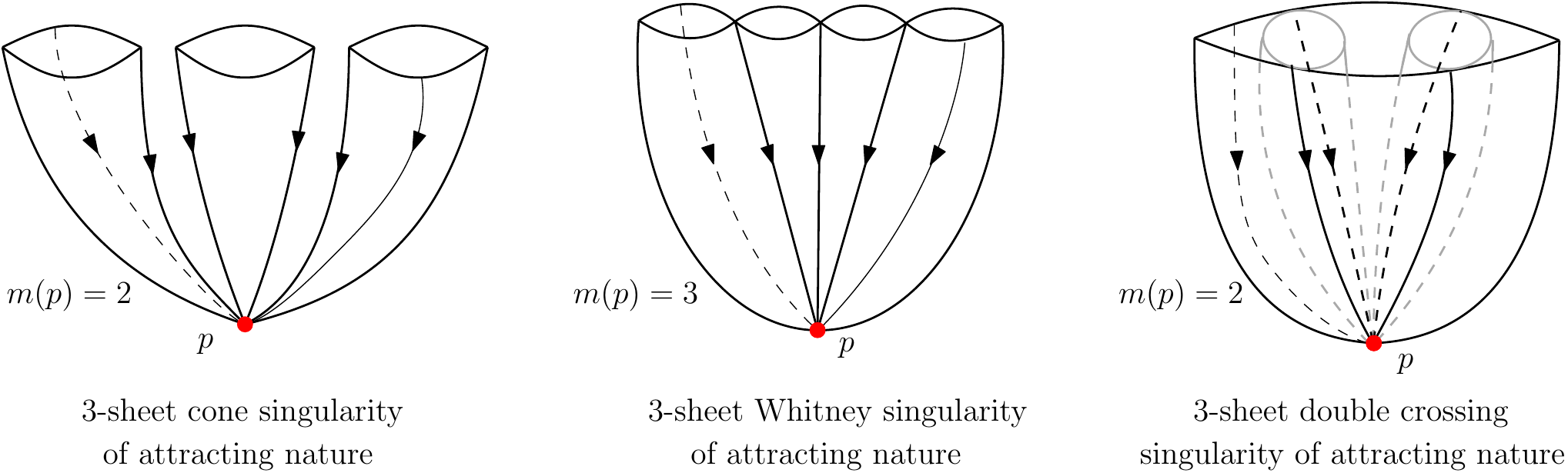}
    \caption{Examples of super attractor GS-singularities.}\label{GS-geral}
\end{figure}

We now define the nature of super attractors and repellers:

\begin{definition}
Let $p\in M$ be a super attractor or a super repeller singularity.  Denote its {\textbf{{nature}}} by:
%
 \begin{itemize}
 \item $a$ (resp., $r$) if $p$ is an attracting (resp., repelling) $n$-sheet cone or Whitney;
 \item $a^n$ (resp.,  $r^n$) if $p$ is an  attracting (resp., repelling) $n$-sheet double or triple  crossing.
 \end{itemize}
\end{definition}



\begin{definition}
Denote by $\mathfrak{M}(\mathcal{GS})$ the set of two-dimensional stratified manifold with generalized  GS-singularities. 
Given $M \in \mathfrak{M}(\mathcal{GS})$, define the set $\mathfrak{X}^{r}_{\mathcal{GS}}(M)$ of generalized GS-vector fields on $M$ so that for each  $X\in\mathfrak{X}^{r}_{\mathcal{GS}}(M)$ the following conditions are satisfied:

\begin{enumerate}
\item $X$ has finitely many generalized GS-singularities;
\item $X$ has no periodic orbits nor limit cycles;
\item The $\alpha$ and $\omega$- limit set of every trajectory of $X$ is a generalized GS-singularity.
\end{enumerate}
\end{definition}


The corresponding flow  $\varphi_{X}$ associated to a GS-vector field $X\in \mathfrak{X}_{\mathcal{GS}}^{r}(M)$ is called a \textbf{\emph{Gutierrez-Sotomayor flow}} on $M$, \textbf{\emph{GS-flow}} for short. 

Since we are interested in working with vector fields that possess only one type of generalized GS-singularities in addition to regular singularities, we establish the following notation for subsets of $\mathfrak{M}(\mathcal{GS})$ and $\mathfrak{X}^{r}_\mathcal{GS}(M)$:
 $\mathfrak{M}(\mathcal{GC})$ (resp., $\mathfrak{M}(\mathcal{GW})$, $\mathfrak{M}(\mathcal{GD})$, $\mathfrak{M}(\mathcal{GT})$) denotes the  set of  stratified 2-manifolds with generalized GS-singularities of regular and cone (resp., Whitney, double crossing, triple crossing) types; $\mathfrak{X}_{\mathcal{GC}}(M)$
(resp., $\mathfrak{X}_{\mathcal{GC}}(W)$, $\mathfrak{X}_{\mathcal{GC}}(D)$, $\mathfrak{X}_{\mathcal{GC}}(T)$) denotes the set of all  vector fields on  $M \in \mathfrak{M}(\mathcal{GC})$ (resp., $\mathfrak{M}(\mathcal{GW})$, $\mathfrak{M}(\mathcal{GD})$, $\mathfrak{M}(\mathcal{GT})$)  which only possess regular and generalized cone (resp., Whitney, double crossing, triple crossing) singularities. 

Hereafter we will refer to generalized GS-flows as GS-flows omitting the term ``generalized''.

%


\section{Morsification of Gutierrez-Sotomayor Flows on isolating blocks}
\label{sec:MGSVFIB}

Let $M\in\mathfrak{M}(\mathcal{GS})$ be a compact stratified 2-manifold  and  
$X\in\mathfrak{X}_{\mathcal{GS}}(M)$ be a GS-vector field on $M$, where $\mathcal{S}=\mathcal{C}$, $\mathcal{W}$, $ \mathcal{D}$ or $ \mathcal{T}$.  Consider the  Gutierrez-Sotomayor flow $\varphi_ {X}$ on $M$ associated to $X$.
In this section, our goal is to establish a regularization process of the GS-singularities which will produce a  smooth 2-manifold $\widetilde{M}$ together with a smooth flow  with regular singularities. We refer to this process as  the Morsification of  GS-singularities. 

\begin{definition} 
Let $M\in\mathfrak{M}(\mathcal{GS})$ be a compact stratified 2-manifold, $X\in\mathfrak{X}_{\mathcal{GS}}(M)$ a GS-vector field on $M$ and $\varphi_{X}$ the GS-flow  associated to $X$.  An isolating block $(N,\varphi_X)$ admits a {\textbf{{Morsification}}} 
if there exists a quadruple $(\widetilde{N},\varphi_{\widetilde{X}}, \mathfrak{h},\mathfrak{p})$ such that 

\begin{enumerate}
\item $\widetilde{N}$ is a smooth 2-manifold;
\item $\widetilde{\varphi}$ is a smooth flow on $\widetilde{N}$ with only  regular singularities;
\item $\mathfrak{h}:N \rightarrow \widetilde{N}$ is a  multivalued map such that $\mathfrak{h}$ 
restricted to
$$N\diagdown \{ \mathscr{SP}(N) \cup \{ x\in N \mid   \omega(x) = p \  or \  \alpha(x)=p , \text{where} \ p \ \text{ is a saddle cone singularity} \} \} $$
%
 is a homeomorphism;
\item $\mathfrak{p}:\widetilde{N}\rightarrow N$ is the projection map and   $\mathfrak{h}\circ\mathfrak{p} = id|_{\widetilde{N}}$.
\end{enumerate}

\end{definition}

In this case, one says that  $(N,\varphi_X)$  admits a Morsification to $(\widetilde{N},\varphi_{\widetilde{X}})$, or that $(\widetilde{N},\varphi_{\widetilde{X}})$ is a Morsification of $(N,\varphi)$.

\begin{theorem}\label{teo_morsificacao_fluxo}
Let $M\in\mathfrak{M}(\mathcal{GS})$ be a singular 2-manifold, $X\in\mathfrak{X}_{\mathcal{GS}}(M)$ a GS-vector field on $M$ and $\varphi_{X}$ the GS-flow  associated to $X$, where $\mathcal{S} = \mathcal{C}, \mathcal{W}, \mathcal{D}$ or $\mathcal{T}$.   Given a GS-singularity $p$ and  an isolating block $(N,\varphi_X)$ for $p$, there  exists   a Morsification  $(\widetilde{N},\varphi_{\widetilde{X}})$, where $\widetilde{N}$ is an isolating block w.r.t. the regularized flow  $\varphi_{\widetilde{X}}$.
\end{theorem}

Now we procedure to the proof of Theorem~\ref{teo_morsificacao_fluxo}, which will be   done in the following subsections for each type of singularities. 

\subsection{Morsification of Cone Singularities}

Let $p$ be a cone singularity in $M\in\mathfrak{M}(\mathcal{GC})$ 
and $N$ be an isolating block for $p$ with GS-flow $\varphi_{X}$ where $X\in\mathfrak{X}_{\mathcal{GC}}(M)$. Consider the boundaries $N^-$  and $N^+$ of  $N$ which constitute the exit  and entering sets of $\varphi_{X}$, respectively. Next it is shown how to Morsify the GS-flow on $N$ to obtain a regular flow on a smooth isolating block $\widetilde{N}$.  Considering a Morsification of all  isolating blocks for singularities of $M$, one can  glue them together  to form a flow on a  smooth 2-manifold $\widetilde{M}$.

\begin{proposition}\label{prop_cone}
Let $M\in\mathfrak{M}(\mathcal{GC})$ be a singular 2-manifold, $X\in\mathfrak{X}_{\mathcal{GC}}(M)$ a GS-vector field on $M$ and $\varphi_{X}$ the GS-flow  associated to $X$.   Given a cone singularity $p$ and  an isolating block $N$ for $p$, there  exists a Morsification $(\widetilde{N},\varphi_{\widetilde{X}})$ where  $\partial \widetilde{N} = \partial N$.
\end{proposition}

\begin{proof} The proof is done by  constructing  a quadruple $(\widetilde{N},\varphi_{\widetilde{X}}, \mathfrak{h},\mathfrak{p})$ for  each type of singularity.

\begin{itemize}
\item[1)] Let $p$ be a repelling (resp. attracting) $n$-sheet cone singularity.
\end{itemize}

Consider a 2-sphere with $n$-holes $\widetilde{N}$, with exit set $\widetilde{N}^- = \sqcup_{j=1}^n \widetilde{N}_j^-$ (resp., entering set $\widetilde{N}^+ = \sqcup_{j=1}^n \widetilde{N}_j^+$) homeomorphic to $N^- = \sqcup_{j=1}^n N_j^-$ (resp. $N^+ = \sqcup_{j=1}^n N_j^+$) and containing a regular repelling (resp., attracting) singularity
$\tilde{p}$  and  regular saddle singularities $\tilde{p}'_i,$ $ i=1,$\ldots$,n-1$.
For each $j=1,$\ldots$,n$, the components of the exit set  $N_j^-, \widetilde{N}_j^-$ (resp., entering sets $N_j^+, \widetilde{N}_j^+$) are homeomorphic to  $S^1$.
Denote the homeomorphisms which preserve counterclockwise orientation on the boundary by $h_j^-:N^-_j\rightarrow\widetilde{N}^-_j$ (resp., $h_j^+:N^+_j\rightarrow\widetilde{N}^+_j$).
Let $\varphi_{\widetilde{X}}$ be a flow on $\widetilde{N}$ that satisfies the following conditions:
for each $i$, there are two orbits
$\tilde{u_1}(\tilde{p},\tilde{p}'_i)$ and $\tilde{u_2}(\tilde{p},\tilde{p}'_i)$ 
such that $\omega(\tilde{u}_1)=\tilde{p_i}'=\omega(\tilde{u}_2)$ and $\alpha(\tilde{u}_1)=\tilde{p}=\alpha(\tilde{u}_2)$ 
(resp., $\omega(\tilde{u}_1)=\tilde{p}=\omega(\tilde{u}_2)$ and $\alpha(\tilde{u}_1)=\tilde{p_i}'=\alpha(\tilde{u}_2)$). See Figure~\ref{flow_cone_r}. 

For each $i=1,$\ldots$,n-1$, chose points $x_i$, $y_i$ where $x_i\in N_i^-$ and $y_i\in N_{i+1}^-$ (resp., $x_i\in N_i^+$ and $y_i\in N_{i+1}^+$). 
Denote by $A=\{\{x_i,y_i\} \mid  i=1,$\ldots$,n-1\}$ the set of these points.
Given $x\in N\setminus \{p\}$, there exists $x^-\in N_j^-$ (resp., $x^+ \in N_j^+$), for some $j=1,\dots ,n$, where $x$ belongs to the orbit $u(p,x^-)$ (resp., $u(x^+,p)$). Define the multivalued map $\mathfrak{h}:N\rightarrow\widetilde{N}$ by:
$$ \mathfrak{h}(u(p,x)) = \left\{
\begin{array}{ll}
u(\tilde{p}, h_j^-(x))  \ \ (\text{resp.,} \ u(h_j^+(x),\tilde{p} )), &  \text{if} \ x\notin A \ \\
\{u(\tilde{p}, \tilde{p}'_i), u(\tilde{p}'_i,h_j^-(x))\}\ \  (\text{resp.,} \ \{u(h_j^+(x),\tilde{p}'_i), u(\tilde{p}'_i,\tilde{p})\}) , & \text{if} \ x\in A
\end{array} 
\right..
$$
Note that $\mathfrak{h}$ is a multivalued extension of the homeomorphisms $h_j^-$ (resp., $h_j^+$), i.e., $\mathfrak{h}|_{N_j^-}=h_j^-$ (resp., $\mathfrak{h}|_{N_j^+}=h_j^+$).

Consider the closed region
$$D_{i}=\{u(\tilde{p},\tilde{p}'_i),u(\tilde{p}'_i,h_j^-(x)) \mid x\in A, i=1,\dots ,n-1, j=1,\dots ,n\}$$ 
$$\text{(resp.,} \  D_{i}=\{u(h_j^+(x),\tilde{p}'_i),u(\tilde{p}'_i,\tilde{p}) \mid x\in A, i=1,\dots ,n-1, j=1,\dots ,n\}).$$
Define the projection map $\mathfrak{p}:\widetilde{N}\rightarrow N$ by
$$ \mathfrak{p}(u(\tilde{x},\tilde{y})) = \left\{
\begin{array}{ll}
\mathfrak{h}^{-1}(u(\tilde{x},\tilde{y})), &  \text{if} \ u(\tilde{x},\tilde{y})\notin D_i \\
u(p,(h^-_j)^{-1}(\tilde{y}))\  \text{(resp.,} \ u((h_j^+)^{-1}(\tilde{x}),p)),  & \text{if} \ u(\tilde{x},\tilde{y})\in D_i,  \ y\neq \tilde{p}'_i \ 
 \text{(resp.,} \ x\neq \tilde{p}'_i) \\
p,& \text{if} \ \tilde{x},\tilde{y}\in \{\tilde{p},\tilde{p}'_i\}
\end{array} 
\right.
$$

\begin{figure}[H]
    \centering
        \includegraphics[width=0.55\textwidth]{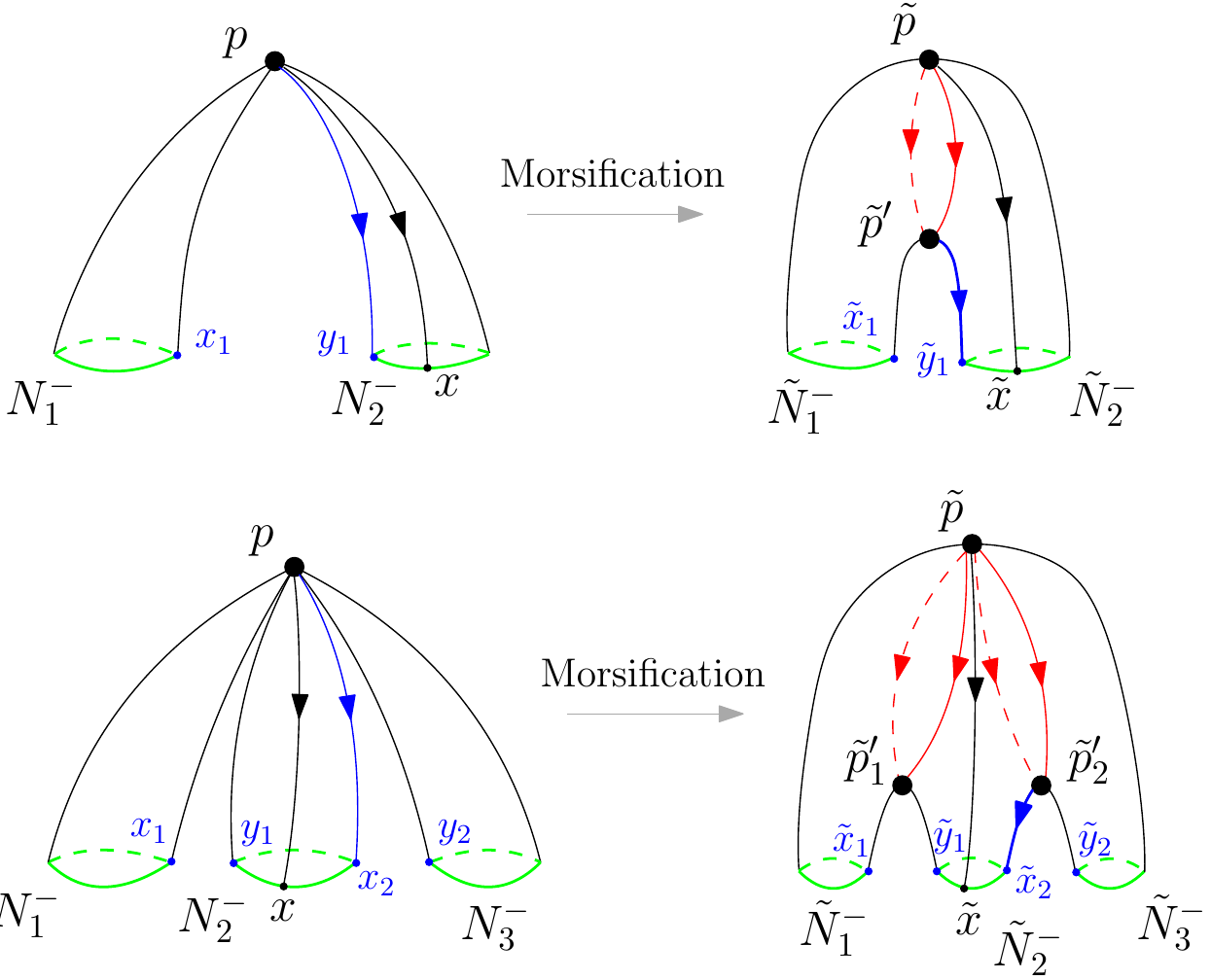}
    \caption{Isolating blocks for  repelling  cone singularities and their Morsifications.}\label{flow_cone_r}
\end{figure}

\begin{itemize}
 \item[2)] Let $p$ be a saddle cone singularity and $N$ its isolating block. 
\end{itemize}
%

 There are two cases to consider, the first being the case where the boundary of the exit and entering sets of $N$ are  disconnected and the second where they are  connected.

\begin{enumerate}

\item[2.1)] Consider the case where the boundaries $N^-$ and $N^+$ of the singular block $N$ are both  disconnected, i.e. $N^-_i\simeq S^1$ and $N^+_i\simeq S^1, i=1,2$
 The Morsified block $\widetilde{N}$, is a sphere with 4 holes, corresponding to the boundaries  $\widetilde{N}^-_i\simeq S^1$ and $\widetilde{N}^+_i\simeq S^1, i=1,2$, corresponding to the connected component of  the exit set $\widetilde{N}^-$ and  entering set $\widetilde{N}^+$, respectively.  See Figure~\ref{flow_cone_s}.

For each $i=1,2$, note that $W^u(p)\cap N^-_i$ is a unique point. Denote this point by $x_i^-$ and consider $u(p,x_i^-)$ the orbit that connects $p$ and $x^-_i$. Similarly, consider $x^+_i=W^s(p)\cap N^+_i$ and $u(x_i^+,p)$ the orbit that connects $x^+_i$ and $p$.
See Figure~\ref{flow_cone_s}. Let $A=\{x_i^-, x_i^+ \mid i=1,2\}$.

Consider a multivalued map 
$\mathfrak{h}^-_i:N^-_i\rightarrow\widetilde{N}^-_i$, such that 
 $\mathfrak{h}_i^-(x_i^-)=\{a_i^-,b_i^- \mid a_i\neq b_i\}$,
$\mathfrak{h}_i^-({N_i^-\setminus\{x_i^-\}})= \widetilde{N}_i^-\setminus[a_i^-,b_i^-]$, 
and $\mathfrak{h}^-_i$ restricted to  $N_i^-\setminus\{x_i^-\}$ 
is a homeomorphism which preserves the counterclockwise orientation on the boundaries.
Similarly, consider a multivalued map 
$\mathfrak{h}^+_i:N^+_i\rightarrow\widetilde{N}^+_i$,
where $\mathfrak{h}_i^+(x_i^+)=\{a_i^+,b_i^+ \mid a_i^+\neq b_i^+\}$ and 
$\mathfrak{h}_i^+({N_i^+\setminus\{x_i^+\}}) = \widetilde{N}_i^+\setminus[a_i^+,b_i^+]$.
Given $x\in N\setminus \{p\}$ such that $x\notin W^u(p) \cup W^s(p)$, there exist $x^+ \in N^+_i$ and $x^- \in N_i^-$ such that 
$x$ belongs to the orbit $u(x^+, x^-)$. If $x\in W^u(p) \cup W^s(p)$ then 
$x$ is on the orbit $u(x_i^+,p)$ or $u(p,x_i^-)$, for some $i=1,2$. 
Define the multivalued map $\mathfrak{h}:N \rightarrow \widetilde{N}$ by
$$ \mathfrak{h}(u(x,y)) = \left\{
\begin{array}{ll}
u(\mathfrak{h}_i^+(x),\mathfrak{h}_i^-(y)), & \text{if} \ x,y\notin A\cup \{p\} \\
u(\mathfrak{h}_i^+(x),\tilde{p})\cup u(\mathfrak{h}_i^+(x),\tilde{p}'), &  \text{if} \ x\in A \ \text{and} \ x^-=p \\
u(\tilde{p},\mathfrak{h_i}^-(y))\cup u(\tilde{p}',\mathfrak{h}_i^-(y)),&  \text{if} \ x=p \ \text{and} \ y\in A \\
\{\tilde{p},\tilde{p}'\},&  \text{if} \ x = p = y
\end{array} 
\right. .
$$

Consider $\varphi_{ij}: (a^+_i,b^+_i)\rightarrow (a^-_j,b^-_j)$ a homeomorphism which preserves the orientation, where $i,j=1,2$, with $i\neq j$. Given
$\tilde{x}^+\in (a^+_1,b^+_1)$, 
let $\varphi_{12} (\tilde{x}^+)=\tilde{x}^-\in (a^-_2,b^-_2)$. 
Consider $u(\tilde{x}^+,\tilde{x}^-)$ an orbit that connects $\tilde{x}^+$ and $\tilde{x}^-$.
Analogously, given
$\tilde{x}^+\in (a^+_2,b^+_2)$, 
let $\varphi_{21} (\tilde{x}^+)=\tilde{x}^-\in (a^-_1,b^-_1)$. 
Consider $u(\tilde{x}^+,\tilde{x}^-)$ an orbit that connects $\tilde{x}^+$ and $\tilde{x}^-$.

Consider the closed region $D_{ij}=\varphi_{ij}(a_i^+,b_i^+)\cup \{\mathfrak{h}(u(x,y)) \mid x,y\in A\cup\{p\}\}$. Define the projection map $\mathfrak{p}:\widetilde{N}\rightarrow N$  by
$$ \mathfrak{p}(u(\tilde{x},\tilde{y})) = \left\{
\begin{array}{ll}
\mathfrak{h}^{-1}(u(\tilde{x},\tilde{y})), \ \ \  \text{if} \ u(\tilde{x},\tilde{y})\notin D_{ij} \\
\{u(x_i^+,p),u(p,x_j^-)\}, \ \ \ \text{if} \ u(\tilde{x},\tilde{y})\in D_{ij}
\end{array} 
\right. .
$$

\begin{figure}[H]
    \centering
        \includegraphics[width=0.63\textwidth]{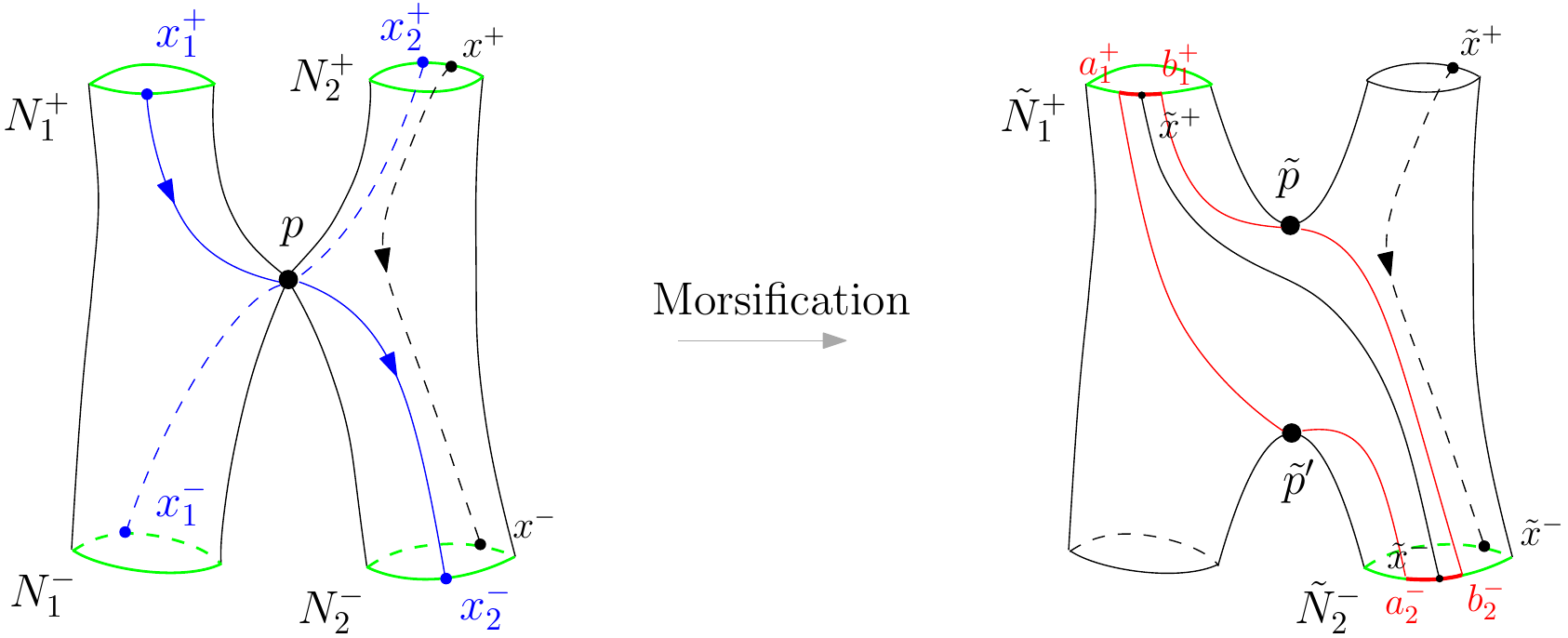}
    \caption{Isolating block for a saddle cone singularity and its Morsification.}\label{flow_cone_s}
\end{figure}

\item[2.2)] Now, consider the case where the block $N$  has connected boundaries  $N^-\simeq S^1$ and $N^+\simeq S^1$. The Morsified block $\widetilde{N}$ is a torus minus 2 disks, i.e, with  boundaries  $\widetilde{N}^-\simeq S^1$ and $\widetilde{N}^+\simeq S^1$,  corresponding to  the exit set and  entering set, respectively. See Figure~\ref{flow_cone_s2}. \label{morsificar_cone}

Let $x^-_1,x^-_2\in N^-$ be the points in $W^u(p)\cap N^-$ and $u(p,x_i^-)$ be the orbit that connects $p$ and $x^-_i, i=1,2$. Consider $x^+_1, x^+_2\in N^+$ as  points in $W^u(p)\cap N^+$ and $u(x_i^+,p)$ the orbit that connects $x^+_i, i=1,2$ and $p$. See Figure~\ref{flow_cone_s2}. Let $A=\{x_i^+,x_i^- \mid i=1,2\}$.
Consider the arcs $C_1^- = (x_1^-, x_2^-)$ and $C_2^- = (x_2^-, x_1^-)$ in $N^-$  as well as  $C_1^+ = (x_1^+, x_2^+)$ and $C_2^+ = (x_2^+, x_1^+)$ in $N^+$ with counterclockwise orientation.

Consider a  multivalued map
$\mathfrak{h}^-:N^-\rightarrow\widetilde{N}^-$,
where $\mathfrak{h}^-(x_i^-)=\{a_i^-,b_i^-\}$,  
$\mathfrak{h}^-(C_1^-) = (b_1^-,a_2^-)$, $\mathfrak{h}^-(C_2^-) = (b_2^-,a_1^-)$, and 
 $\mathfrak{h}^-$ restricted to $N^-\setminus\{x_1^-,x_2^-\}$ 
is a homeomorphism which preserves the counterclockwise orientation on the boundaries.
Similarly, consider a  multivalued map
$\mathfrak{h}^+:N^+\rightarrow\widetilde{N}^+$,
where 
$\mathfrak{h}^+(x_i^+)=[a_i^+,b_i^+]$, 
$\mathfrak{h}^-(C_1^+) = (b_1^+,a_2^+)$, $\mathfrak{h}^-(C_2^+) = (b_2^+,a_1^+)$
and $\mathfrak{h}^+(x_i^+)$ restricted to $N^+\setminus\{x_1^+, x_2^+\}$ 
is a homeomorphism that preserves the counterclockwise orientation on boundaries. 
Given $x\in N\setminus \{p\}$ and $x\notin W^u(p) \cup W^s(p)$, there exist $x^+ \in N^+_i$ and $x^- \in N_i^-$ such that 
$x$ belongs to the  orbit $u(x^+, x^-)$. If $x\in W^u(p) \cup W^s(p)$ then 
$x$ is in the orbit $u(x_i^+,p)$ or $u(p,x_i^-)$, for some $i=1,2$. 
Define the multivalued map $\mathfrak{h}:N \rightarrow \widetilde{N}$ by
$$ \mathfrak{h}(u(x,y)) = \left\{
\begin{array}{ll}
u(\mathfrak{h}_i^+(x),\mathfrak{h}_i^-(y)), &  \text{if} \ x,y\notin A\cup \{p\} \\
u(\mathfrak{h}_i^+(x^),\tilde{p})\cup u(\mathfrak{h}_i^+(x),\tilde{p}'), & \text{if} \ x\in A \ \text{and} \ x^-=p \\
u(\tilde{p},\mathfrak{h_i}^-(y))\cup u(\tilde{p}',\mathfrak{h}_i^-(y)), & \text{if} \ x=p \ \text{and} \ y\in A \\
\{\tilde{p},\tilde{p}'\},  & \text{if} \ x = p = y
\end{array} 
\right..
$$

Let $\varphi_{ij}: (a^+_i,b^+_i)\rightarrow (a^-_j,b^-_j)$ be a homeomorphism which preserves orientation, where $i,j=1,2, i\neq j$. Given
$\tilde{x}^+\in (a^+_1,b^+_1)$, 
let $\varphi_{12} (\tilde{x}^+)=\tilde{x}^-\in (a^-_2,b^-_2)$. 
Let $u(\tilde{x}^+,\tilde{x}^-)$ be an orbit that connects $\tilde{x}^+$ and $\tilde{x}^-$.
Analogously, given
$\tilde{x}^+\in (a^+_2,b^+_2)$, 
let $\varphi_{21} (\tilde{x}^+)=\tilde{x}^-\in (a^-_1,b^-_1)$. 
Let $u(\tilde{x}^+,\tilde{x}^-)$  be an orbit that connects $\tilde{x}^+$ and $\tilde{x}^-$.

Consider the closed region $D_{ij}= \varphi_{ij}(a_i^+,b_i^+)\cup \{\mathfrak{h}(u(x,y)) \mid x,y\in A\cup\{p\}\}$. Define the  projection map $\mathfrak{p}:\widetilde{N}\rightarrow N$ by
$$ \mathfrak{p}(u(\tilde{x},\tilde{y})) = \left\{
\begin{array}{ll}
\mathfrak{h}^{-1}(u(\tilde{x},\tilde{y})), & \text{if} \ u(\tilde{x},\tilde{y})\notin D_{ij} \\
\{u(x_i^+,p),u(p,x_j^-)\}, & \text{if} \ u(\tilde{x},\tilde{y})\in D_{ij}
\end{array} 
\right..
$$

\end{enumerate}
%

\begin{figure}[H]
    \centering
        \includegraphics[width=0.63\textwidth]{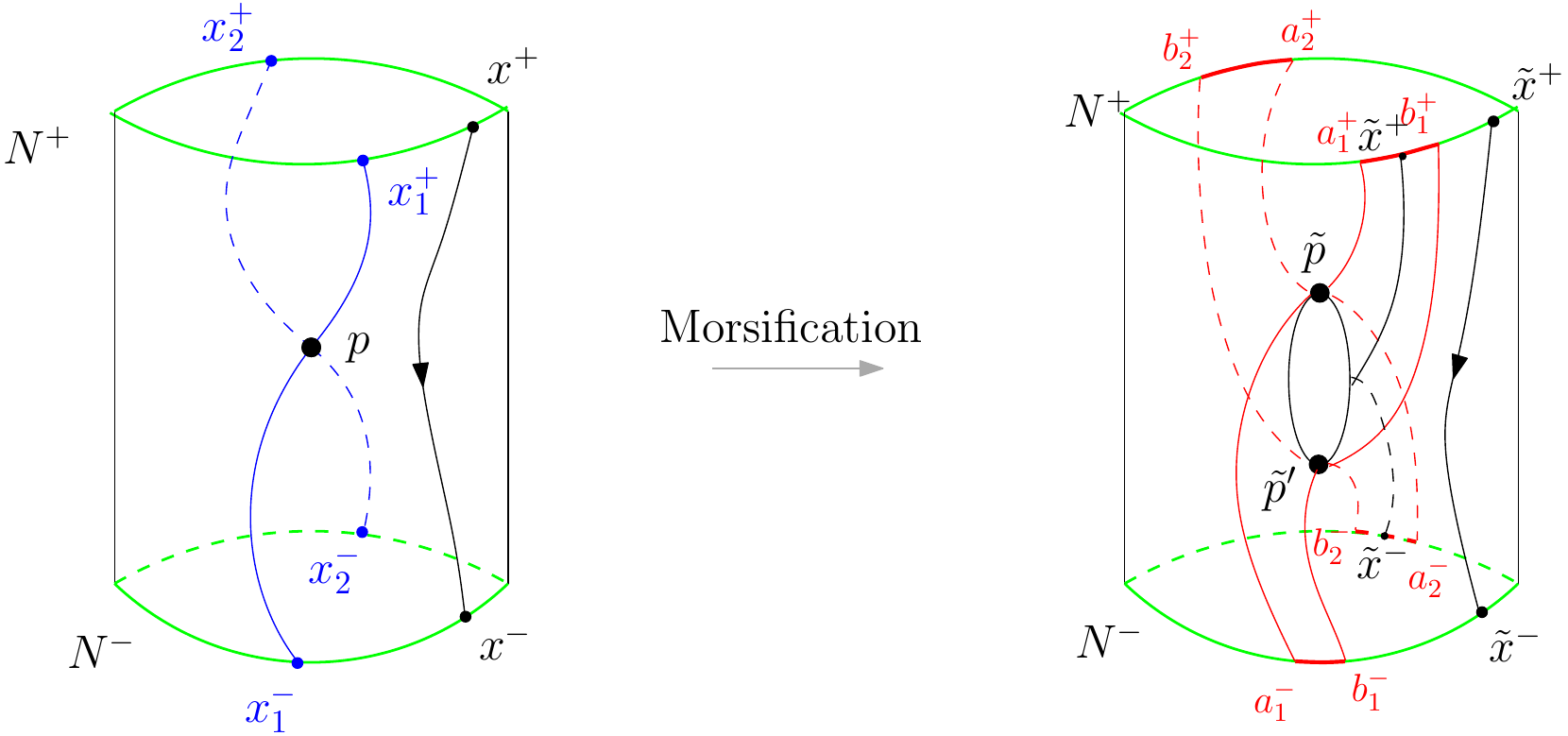}
    \caption{Isolating block for a saddle cone singularity and its Morsification.}\label{flow_cone_s2}
\end{figure}

\end{proof}

Combinatorially the isolating blocks for cone singularities together with its Morsification can be seen as the   Lyapunov (semi)graphs in Figure~\ref{fig_graph_morsef-conev}. 

\begin{figure}[H]
    \centering
        \includegraphics[width=0.8\textwidth]{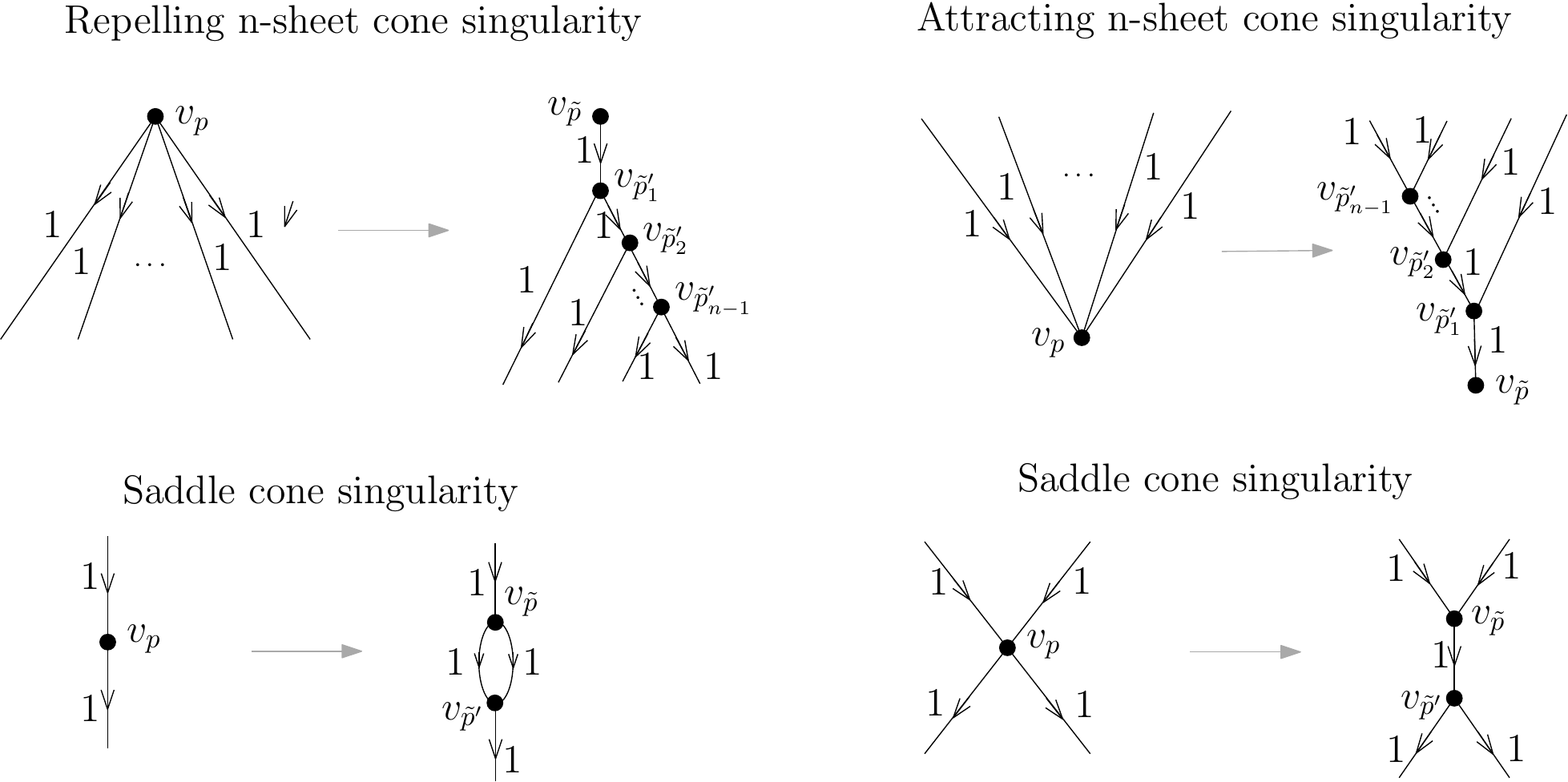}
    \caption{Morsification of a Lyapunov semigraph with a vertex associated to a cone singularity.}\label{fig_graph_morsef-conev}
\end{figure}

\subsection{Morsification of Whitney Singularities}

Let $p$ be a Whitney singularity in $M\in\mathfrak{M}(\mathcal{GW})$ 
and $N$ be an isolating block for $p$ with GS-flow $\varphi_{X}$, where $X\in\mathfrak{X}_{\mathcal{GW}}(M)$. Consider the boundaries $N^-$  and $N^+$ of the block $N$ which constitute the exit and entering sets of $\varphi_{X}$, respectively. Next it is shown how to Morsify the GS-flow on $N$ to obtain a regular flow on a smooth isolating block $\widetilde{N}$.

\begin{proposition}\label{prop_whitney}
Let $M\in\mathfrak{M}(\mathcal{GW})$ be a singular 2-manifold, $X\in\mathfrak{X}_{\mathcal{GW}}(M)$ a GS-vector field on $M$ and $\varphi_{X}$ the GS-flow  associated to $X$.   Given a Whitney singularity $p$ and  an isolating block $N$ for $p$, there  exists a Morsification $(\widetilde{N},\varphi_{\widetilde{X}})$  such that each singular orbit of $\varphi_{X}$ admits a duplication of the orbits in $N$.
\end{proposition}

\begin{proof} The proof follows by considering each type of singularity and  constructing a regular  isolating block with a smooth flow defined on it.
 \begin{itemize}

\item[1)] Let $p$ be a repelling (resp. attracting) $n$-sheet Whitney singularity and $N$ its isolating block.
 \end{itemize}

Consider a regular isolating block $\widetilde{N}$, homeomorphic to $D^2$, containing a regular repelling (resp. attracting) singularity $\tilde{p}$ with exit set $\widetilde{N}^-$ (resp. entering set $\widetilde{N}^+$) homeomorphic to $S^1$, as in  Figure~\ref{flow_whitney}.

Let $x^-_i\in N^-$ (resp. $x^+_i\in N^+$) be points associated to the singular orbit $u(p,x_i^-)$ (resp., $u(x_i^+,p)$) that connects $p$ and $x^-_i,$ (resp. $p$ and $x^+_i$) for $i=1,$\ldots$,n-1$.  
Define $A=\{x_i^-; i=1,\dots ,n-1\}$ (resp., $A=\{x_i^+; i=1,\dots ,n-1\}$).
Consider the arcs $C_1 = (x_1,x_1)$, $C_2^1 = (x_1,x_2)$, $C_2^2 = (x_2,x_1)$, $\ldots$ , 
$C_{n-1}^1=(x_{n-2},x_{n-1})$, $C_{n-1}^2 = (x_{n-1},x_{n-2})$ and $C_n = (x_{n-1},x_{n-1})$
in $N^-$ (resp., $N^+$) oriented counterclockwise.

Consider  a multivalued map
$\mathfrak{h}^-:N^-\rightarrow\widetilde{N}^-$, 
where $\mathfrak{h}^-(x_i^-)=\{a_i^-,b_i^-\}$,
$\mathfrak{h}^-(C_1)=(a_1^-,b_1^-)$, $\mathfrak{h}^-(C_2^1)=(a_1^-,a_2^-)$,
$\mathfrak{h}^-(C_2^2)=(b_2^-,b_1^-)$, $\ldots$ , $\mathfrak{h}^-(C_{n-1}^1)=(a_{n-2}^-,a_{n-1}^-)$,
$\mathfrak{h}^-(C_{n-1}^2)=(b_{n-1}^-,b_{n-2}^-)$,  $\mathfrak{h}^-(C_n^1)=(a_{n-1}^-,b_{n-1}^-)$, 
and $\mathfrak{h}^-$ restricted to $N^-\setminus\bigcup_{i=1}^{n-1}\{x_i^-\}$ 
is a homeomorphism which preserves the counterclockwise orientation on the boundaries.
Similarly, consider a multivalued map
$\mathfrak{h}^+:N^+\rightarrow\widetilde{N}^+$, 
where $\mathfrak{h}^+(x_i^+)=\{a_i^+,b_i^+\}$ 
$\mathfrak{h}^+(C_1)=(a_1^+,b_1^+)$, $\mathfrak{h}^+(C_2^1)=(a_1^+,a_2^+)$,
$\mathfrak{h}^+(C_2^2)=(b_2^+,b_1^+)$, $\ldots$ , $\mathfrak{h}^+(C_{n-1}^1)=(a_{n-2}^+,a_{n-1}^+)$,
$\mathfrak{h}^+(C_{n-1}^2)=(b_{n-1}^+,b_{n-2}^+)$, $\mathfrak{h}^+(C_n^1)=(a_{n-1}^+,b_{n-1}^+)$.  
and $\mathfrak{h}^+$ restricted to $N^+\setminus\bigcup_{i=1}^{n-1}\{x_i^+\}$
is a homeomorphism which preserves the counterclockwise orientation on the boundaries.
Given $x\in N\setminus \{p\}$, there exists $x^-\in N^-$ (resp., $x^+ \in N^+$), where $x$ belongs to the  orbit $u(p,x^-)$ (resp., $u(x^+,p)$). Define the multivalued map $\mathfrak{h}:N\rightarrow\widetilde{N}$ by:
$$ \mathfrak{h}(u(p,x)) = \left\{
\begin{array}{ll}
u(\tilde{p}, \mathfrak{h}^-(x)) \ \ (\text{resp.,} \ u(\mathfrak{h}^+(x),\tilde{p} )), & \text{if} \ x\notin A \ \\
\{u(\tilde{p},a_i^-), u(\tilde{p},b_i^-)\} \ \  (\text{resp.,} \ \{u(a_i^+,\tilde{p}), u(b_i^+,\tilde{p})\},&   \text{if} \ x\in A
\end{array} 
\right.
$$
%
%
%
%

Consider the closed set $D=\{u(\tilde{p},a_i^-),u(\tilde{p},b_i^-);i=1,\dots ,n-1\}$ (resp., $D=\{u(a_i^+,\tilde{p}),u(b_i^+,\tilde{p});i=1,\dots ,n-1\}$). Define the projection map 
$\mathfrak{p}:\widetilde{N}\rightarrow N$ as
$$ \mathfrak{p}(u(\tilde{x},\tilde{y})) = \left\{
\begin{array}{ll}
\mathfrak{h}^{-1}(u(\tilde{x},\tilde{y})), & \text{if} \ u(\tilde{x},\tilde{y})\notin D \ \\
u(p, x_i^- ) \ \  (\text{resp.} \ u(x_i^+,p)), & \text{if} \ u(\tilde{x},\tilde{y})\in D
\end{array} 
\right. .
$$

\begin{figure}[H]
    \centering
        \includegraphics[width=0.52\textwidth]{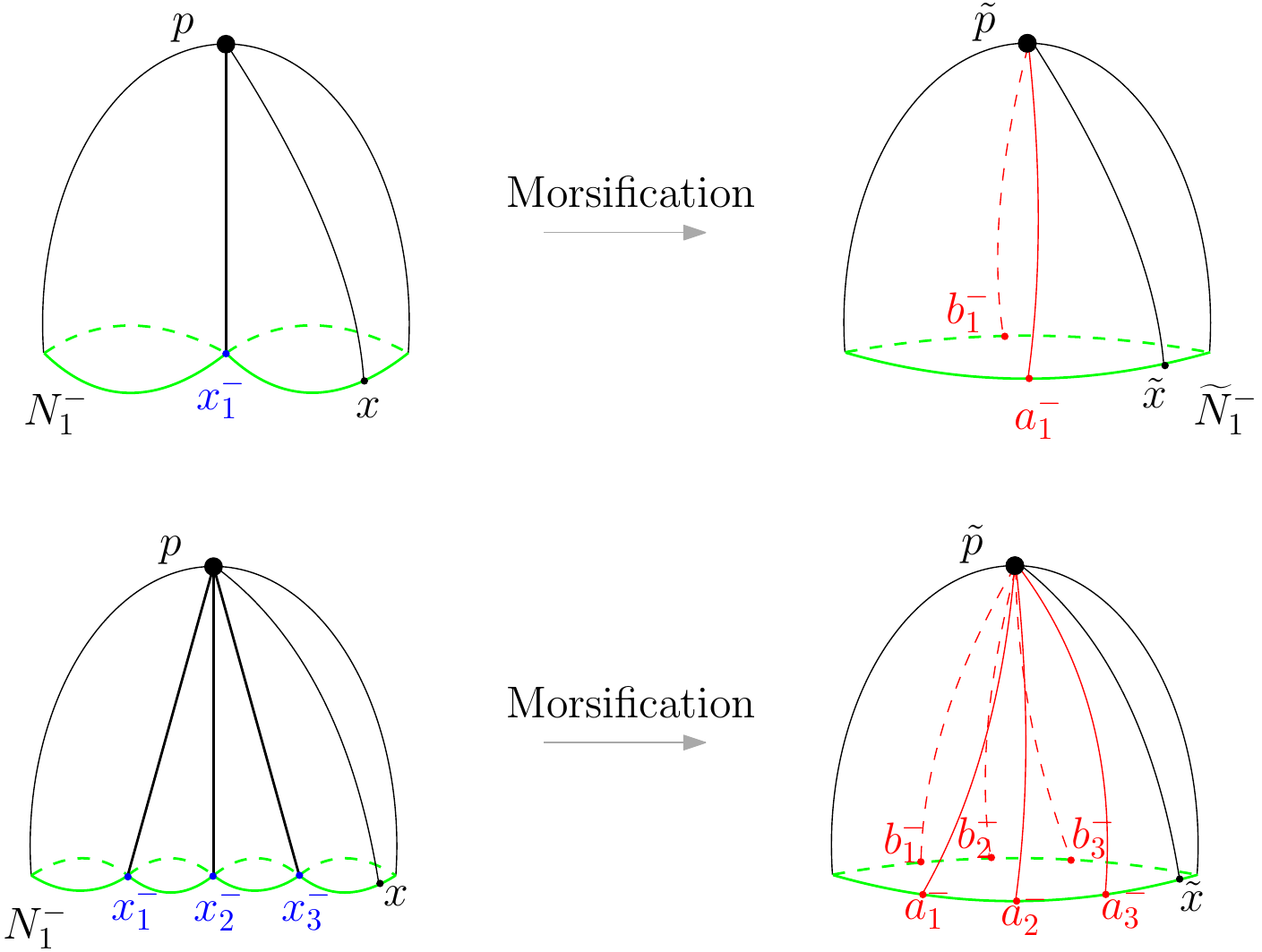}
    \caption{Isolating blocks for a repelling  Whitney singularity and its Morsification.}\label{flow_whitney}
\end{figure}

Another possible Morsification for a repelling $n$-sheet Whitney singularity is a disjoint union of $n$ repeller disks. 


%
\begin{itemize}
\item[2)]  Let $p$ be a saddle Whitney singularity of  $s_s$-nature. (If  $p$ has  $s_u$-nature the prove is completely analogous by using the reverse flow.) 
\end{itemize}
Let $N$ be an isolating   block for $p$.
One has two cases to consider, first when the exit  set of $N$ is disconnected and secondly when its is connected.


%
\begin{enumerate}
\item[2.1)] Consider the case where  the exit set  $N^-$ is disconnected, i.e. $N^-_i$ is  homeomorphic to $S^1$ for $ i=1,2$. The regular  isolating  block $\widetilde{N}$ is a sphere with 3 holes containing a regular singularity of saddle nature $\tilde{p}$ with entering set $\widetilde{N}^+$  homeomorphic to $S^1$ and exit set $\widetilde{N}^-$ with exactly two boundary components homeomorphic to $S^1$.  See Figure~\ref{flow_w_s2}. 
Let $x^+\in N^+$ 
be the point belonging to a singular orbit $u(x^+,p)$. 
Define the multivalued map 
$\mathfrak{h}^+:N^+\rightarrow\widetilde{N}^+$
by $\mathfrak{h}^+(x^+)=\{a^+,b^+\}$ 
and $\mathfrak{h}^+({N^+\setminus\{x^+\}}) = \widetilde{N}^+\setminus\{a^+,b^+\}$, such that  $\mathfrak{h}^+$ restricted to $N^+\setminus\{x^+\}$
is a homeomorphism which preserves the counterclockwise orientation on the boundaries. Consider the trivial homeomorphisms $h_i:N_i^-\rightarrow \widetilde{N}_i^-$.
Define the multivalued map  $\mathfrak{h}:N\rightarrow\widetilde{N}$ as
$$ \mathfrak{h}(u(x,y)) = \left\{
\begin{array}{ll}
u(\mathfrak{h}^+(x),h_j(y)), &   \text{if} \ x\neq x^+ \ \text{and} \ x\neq p \\
u(\tilde{p},h_j(y)), &   \text{if} \ x\neq x^+ \ \text{and} \ x= p \\
\{u(a^+,\tilde{p}), u(b^+,\tilde{p})\}, &  \text{if} \ x=x^+
\end{array} 
\right.
$$

Consider the closed set $D=\{u(a^+,\tilde{p}),u(b^+,\tilde{p})\}$. Define the projection map $\mathfrak{p}:\widetilde{N}\rightarrow N$ by
$$ \mathfrak{p}(u(\tilde{x},\tilde{y})) = \left\{
\begin{array}{ll}
\mathfrak{h}^{-1}(u(\tilde{x},\tilde{y})), &  \text{if} \ u(\tilde{x},\tilde{y})\notin D \\
u(x^+,p), &  \text{if} \ u(\tilde{x},\tilde{y})\in D
\end{array} 
\right. '.
$$
%

\begin{figure}[H]
    \centering
        \includegraphics[width=0.62\textwidth]{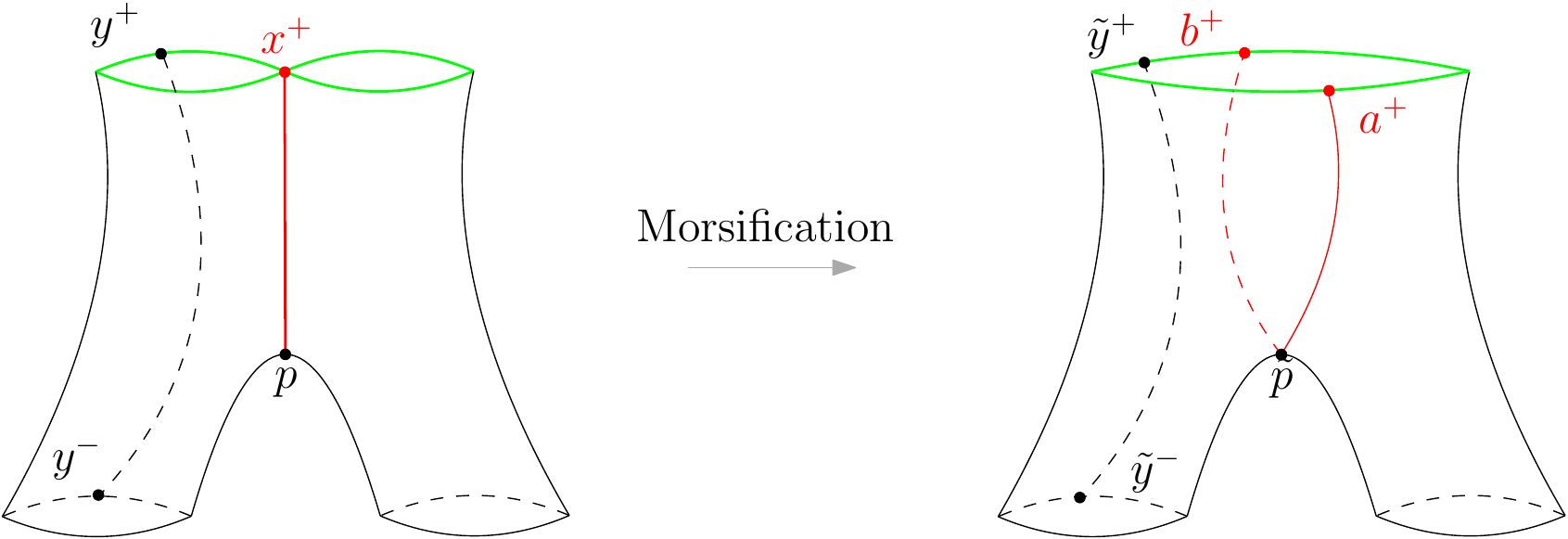}
   \caption{Isolating blocks for a saddle Whitney singularity and its Morsification.}\label{flow_w_s2}
\end{figure}

\item[2.2)] Consider the case where  the exit set  $N^-$  is connected, homeomorphic to $S^1$. The regular  isolating  block $\widetilde{N}$ is a sphere with 3 holes containing a regular singularity of saddle nature $\tilde{p}$ with exit set $\widetilde{N}^-$  homeomorphic to $S^1$ and entering set $\widetilde{N}^+$ with exactly two boundary components homeomorphic to $S^1$. See Figure~\ref{flow_w_s1}. 
Let $x^+\in N^+$ 
be the point in the singular orbit $u(x^+,p)$. 
Consider a  multivalued map 
$\mathfrak{h}^+:N^+\rightarrow\widetilde{N}^+$,
such that  $\mathfrak{h}^+(x^+)=\{a^+,b^+\}$, 
$\mathfrak{h}^+({N^+\setminus\{x^+\}}) = \widetilde{N}^+\setminus\{a^+,b^+\}$, and $\mathfrak{h}^+$ restricted to $N^+\setminus\{x^+\}$
is a homeomorphism which preserves the counterclockwise orientation on the boundaries.
Considering the trivial homeomorphism $h:N^-\rightarrow \widetilde{N}^-$, one defines the multivalued map $\mathfrak{h}:N\rightarrow\widetilde{N}$ as
$$ \mathfrak{h}(u(x,y)) = \left\{
\begin{array}{ll}
u(\mathfrak{h}(x),h_j(y)), &  \text{if} \ x\neq x^+ \ \text{and} \ x\neq p \\
u(\tilde{p},h_j(y)), &  \text{if} \ x\neq x^+ \ \text{and} \ x= p \\
\{u(a^+,\tilde{p}), u(b^+,\tilde{p})\}, & \text{if} \ x=x^+
\end{array} 
\right..
$$

Consider the closed set $D=\{u(a^+,\tilde{p}),u(b^+,\tilde{p})\}$. Define the projection map $\mathfrak{p}:\widetilde{N}\rightarrow N$ by
$$ \mathfrak{p}(u(\tilde{x},\tilde{y})) = \left\{
\begin{array}{ll}
\mathfrak{h}^{-1}(u(\tilde{x},\tilde{y})), & \text{if} \ u(\tilde{x},\tilde{y})\notin D \\
u(x^+,p), &  \text{if} \ u(\tilde{x},\tilde{y})\in D
\end{array} 
\right. .
$$

\begin{figure}[H]
    \centering
        \includegraphics[width=0.6\textwidth]{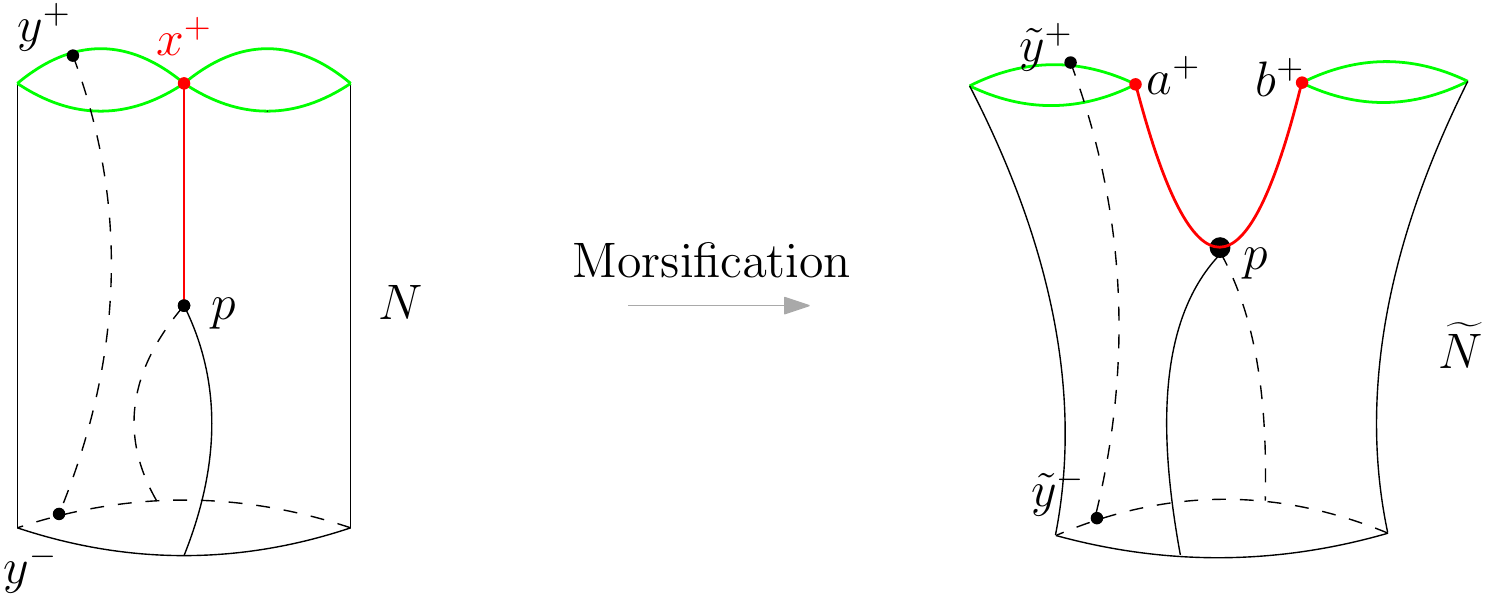}
   \caption{Isolating blocks for  a saddle  Whitney singularity and its Morsification. }\label{flow_w_s1}
\end{figure}

\end{enumerate}

\end{proof}

%
%

Combinatorially the isolating blocks for Whitney singularities together with its Morsification can be seen as the   Lyapunov (semi)graphs in Figure~\ref{fig_graph_morsef-whitney}. 

\begin{figure}[H]
    \centering
        \includegraphics[width=0.7\textwidth]{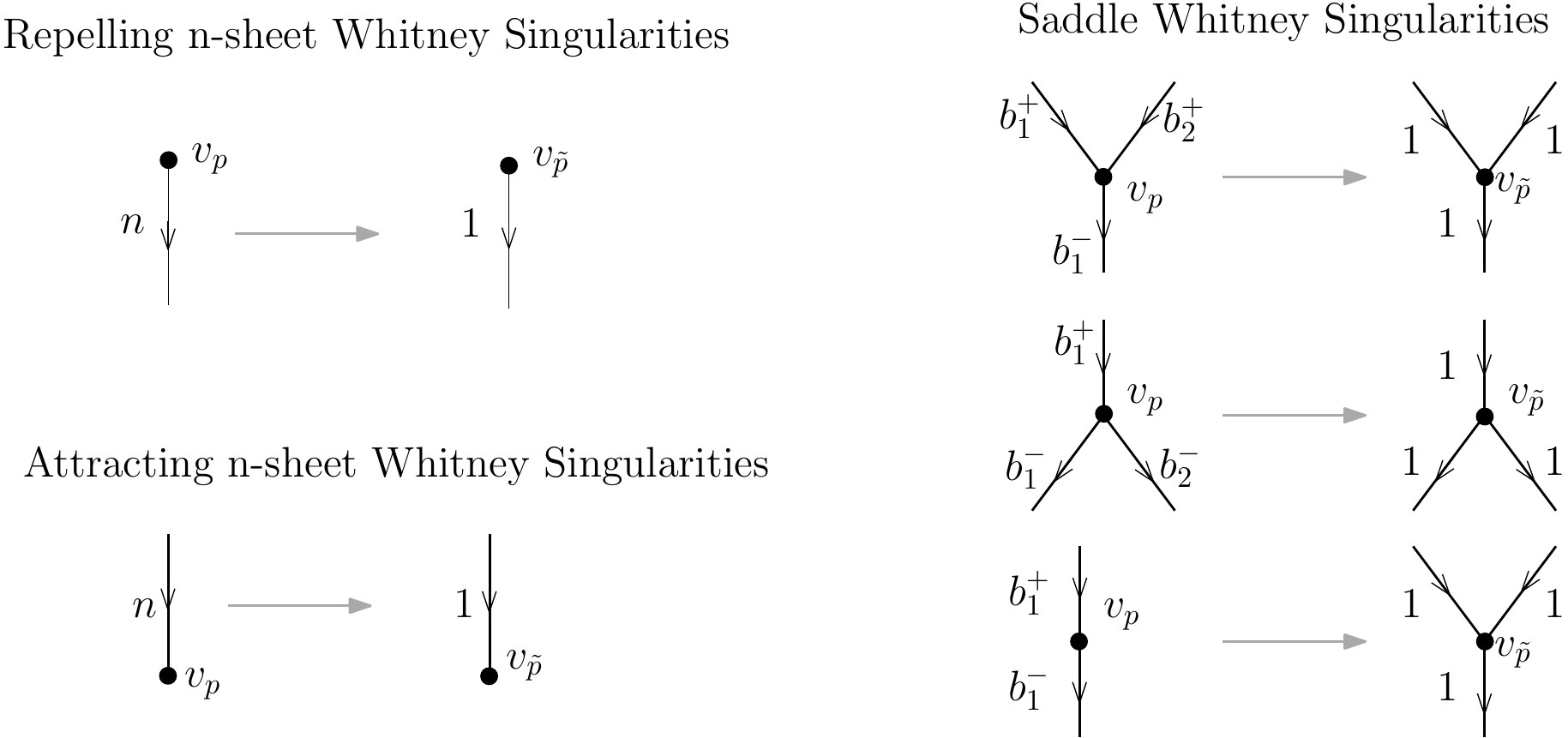}
    \caption{Morsification of a Lyapunov semigraph with vertex associated to a Whitney singularity.}\label{fig_graph_morsef-whitney}
\end{figure}

\subsection{Morsification of Double and Triple Crossing Singularities}

Let $p$ be a double crossing singularity in $M\in\mathfrak{M}(\mathcal{GD})$ 
and $N$ be an isolating block for $p$ with GS-flow $\varphi_{X}$ where $X\in\mathfrak{X}_{\mathcal{GD}}(M)$. Consider the boundaries $N^-$  and $N^+$ of the isolating block $N$ which constitute the exit  and entering sets of $\varphi_{X}$, respectively. Next it is shown how to Morsify the GS-flow on $N$ to obtain a regular flow on a smooth isolating block $\widetilde{N}$.  Considering a Morsification of all  isolating blocks of singularities of $M$, one can  glue them together  to form a flow on a disconnected  smooth surface $\widetilde{M}$.

\begin{proposition}\label{prop_duplo}
Let $M\in\mathfrak{M}(\mathcal{GD})$ be a singular 2-manifold, $X\in\mathfrak{X}_{\mathcal{GD}}(M)$ a GS-vector field on $M$ and $\varphi_{X}$ the GS-flow  associated to $X$.   Given a double crossing singularity $p$ and  an isolating block $N$ for $p$, there  exists a Morsification $(\widetilde{N},\varphi_{\widetilde{X}})$  such that each singular orbit of $\varphi_{X}$ admits a duplication of  orbits in $N$.
\end{proposition}

\begin{proof}

Consider the different  type of double crossing singularities. 
\begin{enumerate}

\item[1)] Let $p$ be a repelling (resp., attracting) $n$-sheet double crossing singularity. Consider a smooth block formed by $n$ disjoint discs,  $\widetilde{N}\simeq \sqcup_{i=1}^n D^2_i$, containing $n$ repelling (resp., attracting) regular singularities $\tilde{p}_1,$\ldots$,\tilde{p}_n$ and having exit set   $\widetilde{N}^-$ (resp. entering set  $\widetilde{N}^+$) homeomorphic  to  a disjoint union of $n$ circles, as in Figure~\ref{flow_duplo}.

Let $A_i=\{x_i,y_i\}$, where  $x_i,y_i\in N^-$ (resp., $x_i,y_i\in N^+$) be the points associated to the singular orbits $u(p,x_i)$ (resp., $u(x_i,p)$) that connects  $p$ and $x_i$,  
and  $u(p,y_i)$ (resp., $u(y_i,p)$) that connects $p$ and $y_i$, 
$i=1,$\ldots$,n-1$. 

Consider the external arcs $C_1^1 = (y_1,x_1)$, $C_2^1 = (x_1,x_2)$, $C_2^2 = (y_2,y_1)$, $\ldots$ , 
$C_{n-1}^1=(x_{n-2},x_{n-1})$, $C_{n-1}^2 = (y_{n-1},y_{n-2})$ and $C_n^1 = (x_{n-1},y_{n-1})$
in $N^-$ (resp., $N^+$) as well as 
$c_1^1 = (y_1,x_1)$, $c_1^2 = (x_1,y_1)$, $c_2^1 = (y_2,x_2)$, 
$c_2^2 = (x_2,y_2)$ ,$\ldots$, 
$c_{n-1}^1=(y_{n-1},x_{n-1})$ and $c_{n-1}^2 = (x_{n-1},y_{n-1})$
in $N^-$ (resp., $N^+$) with counterclockwise orientations.

Define the multivalued map 
$\mathfrak{h}^-:N^-\rightarrow\widetilde{N}^-$, 
by $\mathfrak{h}^-(x_i)=\{a_i,c_i\}$, $\mathfrak{h}^-(y_i)=\{b_i,d_i\}$, 
$\mathfrak{h}^-(C_1)=(b_1,a_1)$, $\mathfrak{h}^-(C_2^1)=(a_1,a_2)$,
$\mathfrak{h}^-(C_2^2)=(b_2,b_1)$, $\ldots$ , $\mathfrak{h}^-(C_{n-1}^1)=(a_{n-2},a_{n-1})$,
$\mathfrak{h}^-(C_{n-1}^2)=(b_{n-1},b_{n-2})$, $\mathfrak{h}^-(C_n)=(a_{n-1},b_{n-1})$,
 $\mathfrak{h}^-(c_1^1)=(d_1,c_1)$, $\mathfrak{h}^-(c_1^2)=(c_1,d_1)$,
 $\mathfrak{h}^-(c_2^1)=(d_2,c_2)$, $\mathfrak{h}^-(c_2^2)=(c_2,d_2)$,$\ldots$,
 $\mathfrak{h}^-(c_{n-1}^1)=(d_{n-1},c_{n-1})$ and $\mathfrak{h}^-(c_{n-1}^2)=(c_{n-1},d_{n-1})$,
such that  $\mathfrak{h}^-$ restricted to  $N^-\setminus\bigcup_{i=1}^{n-1}\{x_i,y_i\}$ 
is a homeomorphism which preserves orientation on the boundaries.
Analogously, define the multivalued map  $\mathfrak{h}^+:N^+\rightarrow\widetilde{N}^+$.
Define the  multivalued map   $\mathfrak{h}:N\rightarrow\widetilde{N}$ by
$$ \mathfrak{h}(u(p,x)) = \left\{
\begin{array}{ll}
u(\tilde{p}_1,\mathfrak{h}^-(x)), \text{if} \  x\in C^*_i \\
u(\tilde{p}_i,\mathfrak{h}^-(x)), \text{if} \  x\in c^*_{i-1} \\
\{u(\tilde{p}_1,\mathfrak{h}^-(x)), u(\tilde{p}_{i+1},\mathfrak{h}^-(x))\}, \text{if} \ x\in A_i
\end{array} 
\right.
$$

The case where $p$ is an attracting double singularity, the prove follows analogously. 

If  $x\in A_i$, then $x=x_i$ ou $x=y_i$. Without loss of generality, suppose that  $x=x_i$. Hence, the orbit  $u(p,x_i)$  is mapped by  $\mathfrak{h}$ to  $u(\tilde{p}_1, a_i)$ and $u(\tilde{p}_{i+1},c_i)$.

Finally, consider the closed set $D=\bigcup_{i=1}^n D_i$, where
$$D_i =\{u(\tilde{p}_1,\mathfrak{h}^-(x)),u(\tilde{p}_{i+1},\mathfrak{h}^-(x));x\in A_i\}.
$$
Define the map $\mathfrak{p}:\widetilde{N}\rightarrow N$ by
$$ \mathfrak{p}(u(\tilde{x},\tilde{y})) = \left\{
\begin{array}{ll}
\mathfrak{h}^{-1}(u(\tilde{x},\tilde{y})), \text{if} \ u(\tilde{x},\tilde{y})\notin D \\
u(p,\mathfrak{h}^{-1}(\tilde{y})), \text{if} \ u(\tilde{x},\tilde{y})\in D_i
\end{array} 
\right. .
$$
See Figure~\ref{flow_duplo}.

\begin{figure}[H]
    \centering
        \includegraphics[width=0.50\textwidth]{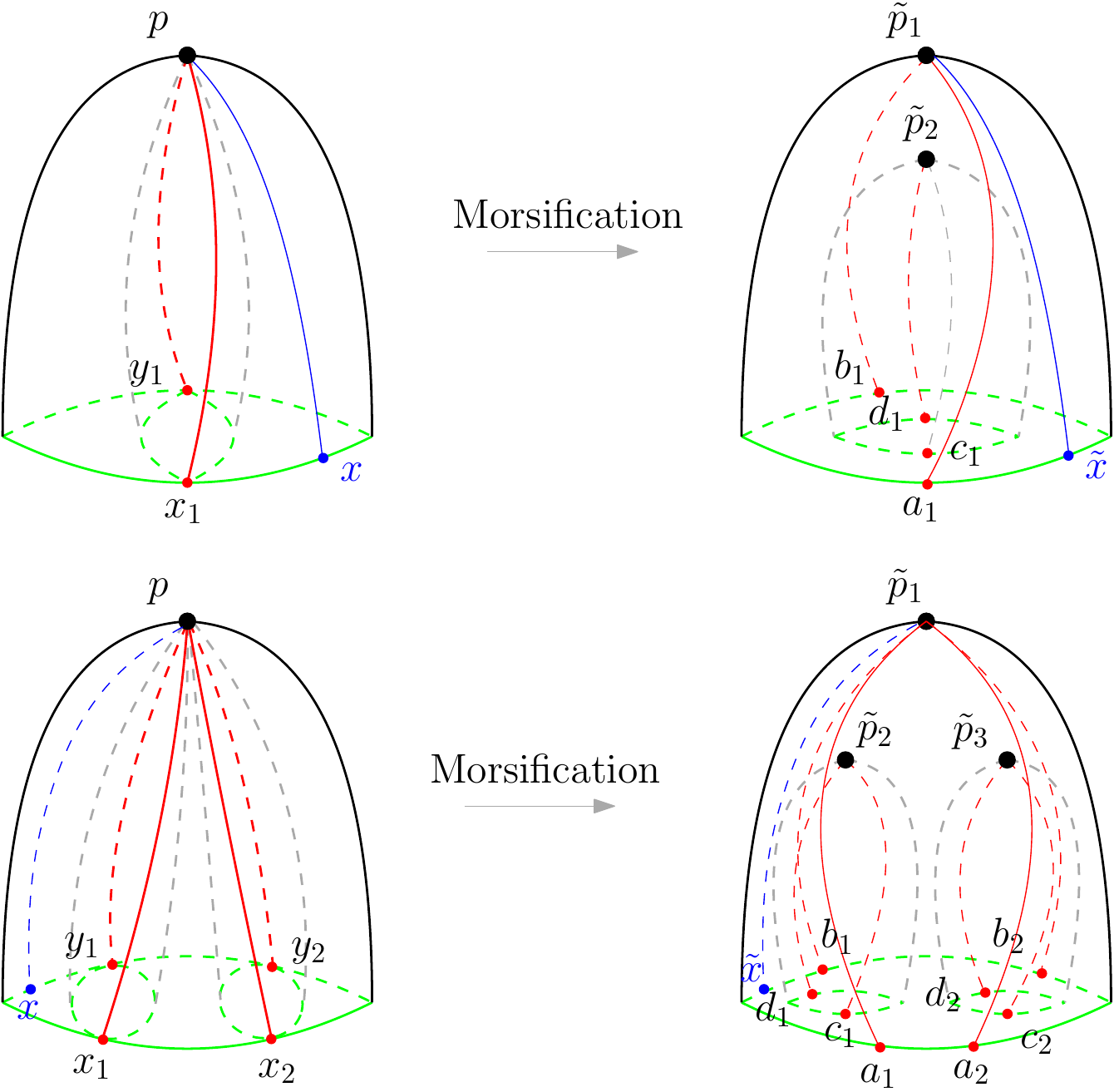}
   \caption{Isolating blocks for  a repelling $2$-sheet and $3$-sheet double crossing singularities and their Morsification.}\label{flow_duplo}
\end{figure}

\item[2)] Let  $p$ be a saddle double crossing singularity of $sa$  (resp., $sr$) nature.  Let $N$ be an isolating block for $p$, and consider the smooth disjoint  block 
$\widetilde{N}$, containing two regular  singularities  $\tilde{p}_1$ and $\tilde{p}_2$, where  $\tilde{p}_1$  is a saddle and $\tilde{p}_2$ is  an attractor (resp., repeller).

The isolating block  $\widetilde{N}$  has $\widetilde{N}^-$  and  $\widetilde{N}^+$ as exit and entering sets, respectively, where each connected component is homeomorphic to 
$S^1.$ There are two cases to be considered. 

\begin{itemize}
\item $\widetilde{N}^-$ (resp., $\widetilde{N}^+$) is disconnected;
\item $\widetilde{N}^-$  (resp., $\widetilde{N}^+$)  connected. 
\end{itemize}
\begin{enumerate}
\item[2.1] Let $N^-_i, i=1,2$, be the connected components of the exit set. 

Let $x^+,y^+\in N^+$ be the points of the singular orbits  $u(x^+,p)$ e $u(y^+,p)$.
Consider the external arcs $C_1 = (y^+,x^+)$, $C_2 = (x^+,y^+)$, 
and the internal arcs
$c_1= (y^+,x^+)$, $c_2 = (x^+,y^+)$ in $N^+$ with counterclockwise orientations. 
Define  the multivalued map 
$\mathfrak{h}^+:N^+\rightarrow\widetilde{N}^+$
by $\mathfrak{h}^+(x^+)=\{a^+,c^+\}$, $\mathfrak{h}^+(y^+)=\{b^+,d^+\}$ and
$\mathfrak{h}^+({N^+\setminus\{x^+,y^+\}}) = \widehat{N}^+\setminus\{a^+,b^+,c^+,d^+\}$, such that the map  $\mathfrak{h}^+$ restricted to $N^+\setminus\{x^+,y^+\}$ is a homeomorphism which preserves the orientation on the boundaries. Consider the trivial homeomorphisms $h_j:N_j^-\rightarrow \widetilde{N}^-_j$.
Define the multivalued map $\mathfrak{h}:N\rightarrow \widetilde{N}$ by
$$ \mathfrak{h}(u(x,y)) = \left\{
\begin{array}{ll}
u(\mathfrak{h}^+(x),h_i(y)), \text{if} \ x\in C_i \\
u(\mathfrak{h}^+(x),\tilde{p}_2), \text{if} \ x\in c_i  \\
\{u(\mathfrak{h}^+(x),\tilde{p}_1),u(\mathfrak{h}^+(x),\tilde{p}_2)\}, \text{if} \ x\in \{x^+,y^+\}
\end{array} 
\right..
$$

If $x_1\in C_i$, the orbit  $u(x_1,y_1)$ is mapped by  $\mathfrak{h}$  to $u(\tilde{x}_1,\tilde{y_1})$, where $\tilde{x}_1=\mathfrak{h}^+(x_1)$ and $\tilde{y}_1=h_i(y_1)$. If $x_2\in c_i$, the orbit
$u(x_2,p)$ is mapped by  $\mathfrak{h}$ to  $u(\tilde{x}_2,\tilde{p}_2)$, where $\tilde{x}_2=\mathfrak{h}^+(x_2)$.

If $x=x^+$, the orbit $u(x^+,p)$ is mapped by $\mathfrak{h}$ to $u(a^+,\tilde{p}_1)$ and $u(c^+,\tilde{p}_2)$. Analogously for  $x=y^+$.

Finally, consider the closed set $D=\{u(\mathfrak{h}^+(x),\tilde{p}_1),u(\mathfrak{h}^+(x),\tilde{p}_2)\}$. Define $\mathfrak{p}:\widetilde{N}\rightarrow N$ by
$$ \mathfrak{p}(u(\tilde{x},\tilde{y})) = \left\{
\begin{array}{ll}
\mathfrak{h}^{-1}(u(\tilde{x},\tilde{y})), \text{if} \ u(\tilde{x},\tilde{y})\notin D \\
u(\mathfrak{h}^{-1}(\tilde{x}),p), \text{if} \ u(\tilde{x},\tilde{y})\in D
\end{array} 
\right..
$$
See Figure~\ref{flow_duplo2}.

\begin{figure}[H]
    \centering
        \includegraphics[width=0.6\textwidth]{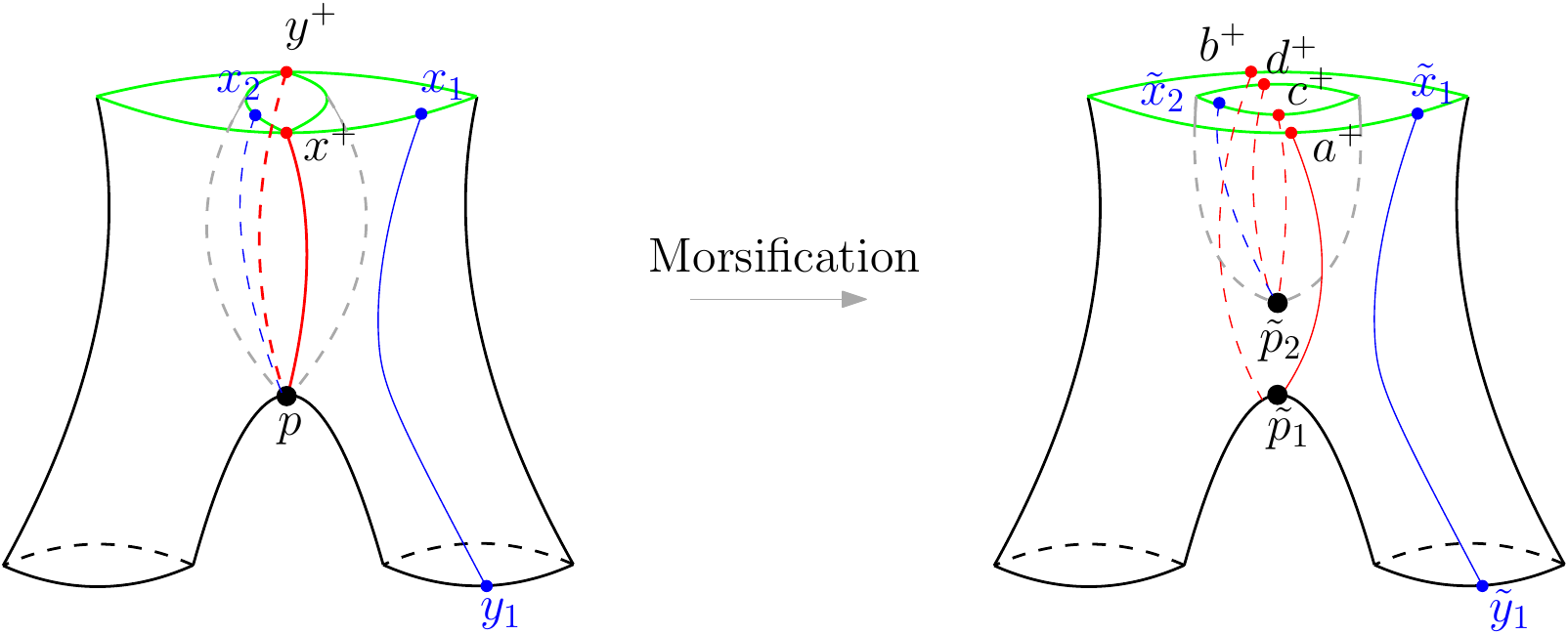}
   \caption{Isolating block for  a saddle double crossing singularity of $sa$-nature and its Morsification.}\label{flow_duplo2}
\end{figure}

\item[2.2] Suppose that $N^-$ and $N^+$ are both connected.  Let $\widetilde{N}$ be the disjoint union of a $2$-sphere minus three discs, $\widetilde{N}_1$, and an attracting disc, $\widetilde{N}_2$.   The entering set of $\widetilde{N}$ is a disjoint union of three circles $C_i$, $i=1,2,3$, where the entering set of $\widetilde{N}_1$ is $C_1$ and $C_2$ and the entering set of   $\widetilde{N}_2$ is $C_3$.

Let $x^+, y^+\in N^+$ be the points on the singular orbits $u(x^+,p)$ and $u(y^+,p)$.
Consider the arcs   $c_1 = (x^+,x^+)$, $c_2 = (y^+,y^+)$ 
$c_3= (y^+,x^+)$, $c_4 = (x^+,y^+)$ in $N^+$ with counterclockwise orientation. Define the multivalued map 
$\mathfrak{h}^+:N^+\rightarrow\widetilde{N}^+$, 
by $\mathfrak{h}^+(x^+)=\{a^+,c^+\}$, $\mathfrak{h}^+(y^+)=\{b^+,d^+\}$,
$\mathfrak{h}^+(c_1) = {C_i}$, for $i=1,2$, $\mathfrak{h}^+(c_3)$ is the arc $(c^+, d^+)$ in $ \widetilde{C_3}$, and  $\mathfrak{h}^+(c_4)$ is the arc $(d^+,c^+)$ in $ \widetilde{C_3}$,  such that the map $\mathfrak{h}^+$ restricted to $N^+\setminus\{x^+,y^+\}$ is a homemomorphism which preserve orientation on the boundary.
Consider the trivial homemomorphism $\mathfrak{h}^-:N^-\rightarrow\widetilde{N}^-$.
Define the multivalued map $\mathfrak{h}:N\rightarrow\widetilde{N}$ by
$$ \mathfrak{h}(u(x,y)) = \left\{
\begin{array}{ll}
u(\mathfrak{h}^+(x),\mathfrak{h}^-(y)), \ \text{if} \ x\in c_1\cup c_2 \\
u(\mathfrak{h}^+(x),\tilde{p}_2), \ \text{if} \ x\in c_3\cup c_4  \\
\{u(\mathfrak{h}^+(x),\tilde{p}_1),u(\mathfrak{h}^+(x),\tilde{p}_2)\}, \text{if} \ x\in \{x^+,y^+\}
\end{array} 
\right. .
$$

If $x=x^+$, the orbit $u(x^+,p )$  is mapped by $\mathfrak{h}$ to $u(a^+,\tilde{p}_1)$ and $u(c^+,\tilde{p}_2)$. If  $x=y^+$,  the orbit $u(y^+,p )$  is mapped by $\mathfrak{h}$ to $u(b^+,\tilde{p}_1)$ and $u(d^+,\tilde{p}_2)$.

Finally consider the closed set  $D=\{u(\mathfrak{h}^+(x),\tilde{p}_1),u(\mathfrak{h}^+(x),\tilde{p}_2)\}$. Define $\mathfrak{p}:\widetilde{N}\rightarrow N$ by
$$ \mathfrak{p}(u(\tilde{x},\tilde{y})) = \left\{
\begin{array}{ll}
\mathfrak{h}^{-1}(u(\tilde{x},\tilde{y})), \text{if} \ u(\tilde{x},\tilde{y})\notin D \\
u(\mathfrak{h}^{-1}(\tilde{x}),p), \text{if} \ u(\tilde{x},\tilde{y})\in D
\end{array} 
\right..
$$
See Figure~\ref{flow_duplo3}.



\end{enumerate}
%

\begin{figure}[H]
    \centering
        \includegraphics[width=0.65\textwidth]{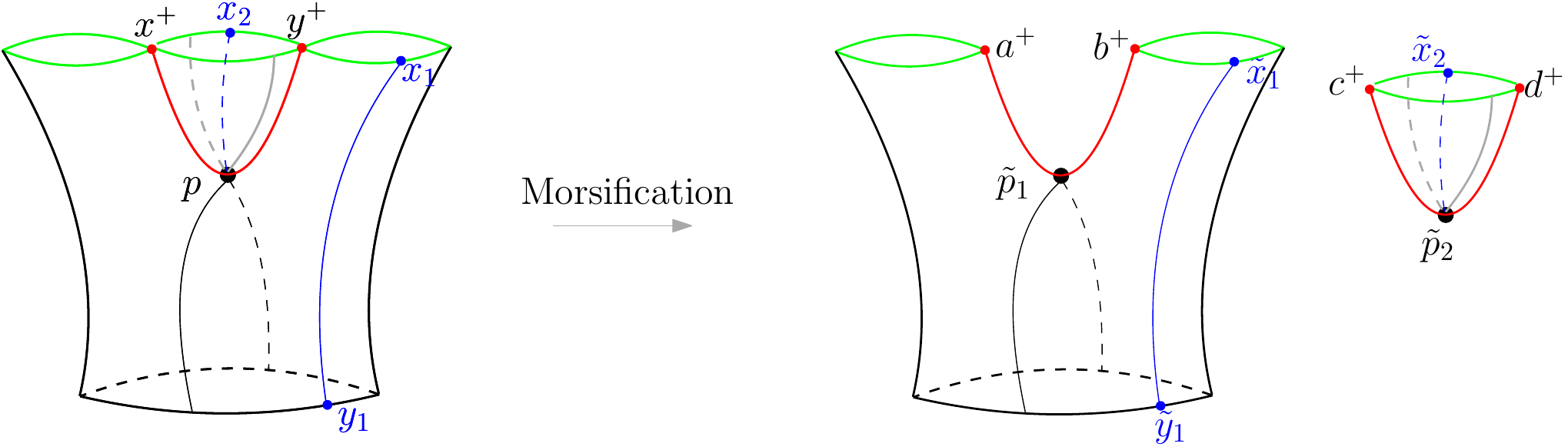}
   \caption{Isolating block for  a saddle double crossing singularity of $sa$-nature and its Morsification.}\label{flow_duplo3}
\end{figure}

\item[3)] Let $p$ be a  double crossing saddle singularity of $ss_s$-nature ($ss_u$-nature).
Let  $N$ be an isolating block for  $p$ and  consider the smooth isolating  block $\widetilde{N}$, containing two regular  saddle singularities   $\tilde{p}_1$ and $\tilde{p}_2$.

The isolating block  $\widetilde{N}$  has $\widetilde{N}^-$  and  $\widetilde{N}^+$ as exit and entering sets, respectively. We will consider the following cases:
\begin{itemize}
\item $\widetilde{N}^+$ (resp., $\widetilde{N}^-$) is connected;
\item $\widetilde{N}^+$  (resp., $\widetilde{N}^-$) is disconnected. 
\end{itemize}
\begin{enumerate}
\item[3.1)] Consider the isolating block $N$ in Figure~\ref{flow_duplo4} with exit set homeomorphic to four disjoint circles $N^{-}_{ij}$, where  
 $N^-_{1j}$ are the external boundaries and  $N^-_{2j}$ are the internal boundaries, $j=1,2$.
Let  $\widetilde{N}$ the a smooth isolating block formed by the disjoint union of two isolating blocks $\widetilde{N}_1$ and $\widetilde{N}_2$, for the regular saddle and attracting singularities $\tilde{p}_1$ and $\tilde{p}_2$, respectively.   In $\widetilde{N}_1$ the exit set is $\widetilde{N}^-_{1j}$ and the entering set  $\widetilde{N}^+_1$ and in  $\widetilde{N}_2$ the exit set is $\widetilde{N}^-_{2j}$ and the entering set is $\widetilde{N}^+_2$, homeomorphic to  $S^1$, $j=1,2$.

Let $x^+,y^+\in N^+$ be  points on the singular orbits 
$u(x^+,p)$ and $u(y^+,p)$.
Consider the external arcs  $C_1 = (y^+,x^+)$, $C_2 = (x^+,y^+)$ 
and the internal arcs
$c_1= (y^+,x^+)$, $c_2 = (x^+,y^+)$ in $N^+$  with counterclockwise orientation. Define the multivalued map  
$\mathfrak{h}^+:N^+\rightarrow\widetilde{N}^+$
by $\mathfrak{h}^+(x^+)=\{a^+,c^+\}$, $\mathfrak{h}^+(y^+)=\{b^+,d^+\}$ e 
$\mathfrak{h}^+({N^+\setminus\{x^+,y^+\}}) = \widetilde{N}^+\setminus\{a^+,b^+,c^+,d^+\}$, such that  $\mathfrak{h}^+$ restricted to $N^+\setminus\{x^+,y^+\}$ is an orientation preserving homeomorphism. Consider the trivial homeomorphisms  $h_{ij}:N^-_{ij}\rightarrow\widetilde{N}^-_{ij}$. Define the multivalued mpa  $\mathfrak{h}:N\rightarrow \widetilde{N}$ by 
$$ \mathfrak{h}(u(x,y)) = \left\{
\begin{array}{ll}
u(\mathfrak{h}^+(x),h_{ij}(y)), \text{if} \ x\in C_i\cup c_i \ \text{and} \ x\neq p \\
u(\tilde{p}_1,h_{1j}(y)), \text{if} \ y\in N^-_{1j} \ \text{and} \ x=p \\
u(\tilde{p}_2,h_{2j}(y)), \text{if} \ y\in N^-_{2j} \ \text{and} \ x=p \\
\{u(\mathfrak{h}^+(x),\tilde{p}_1),u(\mathfrak{h}^+(x),\tilde{p}_2)\}, \text{se} \ x\in \{x^+,y^+\}
\end{array} 
\right.
$$

If $x_1\in C_i$, the  orbit $u(x_1,y_1)$ is mapped by  $\mathfrak{h}$ to $u(\tilde{x}_1,\tilde{y_1})$,where $\tilde{x}_1=\mathfrak{h}^+(x_1)$ and  $\tilde{y}_1=h_{1j}(y_1)$. If $x_2\in c_i$, the orbit
$u(x_2,y_2)$ is mapped by  $\mathfrak{h}$ to $u(\tilde{x}_2,\tilde{y}_2)$, where $\tilde{x}_2=\mathfrak{h}^+(x_2)$ and $\tilde{y_2}=h_{2j}(y_2)$.

If $x\in N^-_{1j}$, the orbit $u(p,x)$ is mapped by  $\mathfrak{h}$ to $u(\tilde{p}_1,\tilde{x})$, where  $\tilde{x}=h_{1j}(x)$. If $x\in N^-_{2j}$, the orbit  $u(x,p)$ is mapped by $\mathfrak{h}$ to $u(\tilde{p}_2,\tilde{x})$, where $\tilde{x}=h_{2j}(x)$. 

If $x=x^+$, the orbit $u(x^+,p)$ is mapped by  $\mathfrak{h}$ to $u(a^+ ,\tilde{p}_1)$ and $u(c^+,\tilde{p}_2)$. Similarly, if  $x=y^+$.

Finally, consider the closed set   $D=\{u(\mathfrak{h}^+(x),\tilde{p}_1),u(\mathfrak{h}^+(x),\tilde{p}_2)\}$. Define $\mathfrak{p}:\widetilde{N}\rightarrow N$ by
$$ \mathfrak{p}(u(\tilde{x},\tilde{y})) = \left\{
\begin{array}{ll}
\mathfrak{h}^{-1}(u(\tilde{x},\tilde{y})), \text{se} \ u(\tilde{x},\tilde{y})\notin D \\
u(\mathfrak{h}^{-1}(\tilde{x}),p), \text{se} \ u(\tilde{x},\tilde{y})\in D
\end{array} 
\right.
$$
See Figure~\ref{flow_duplo4}.
\end{enumerate}
%

\begin{figure}[H]
    \centering
        \includegraphics[width=0.55\textwidth]{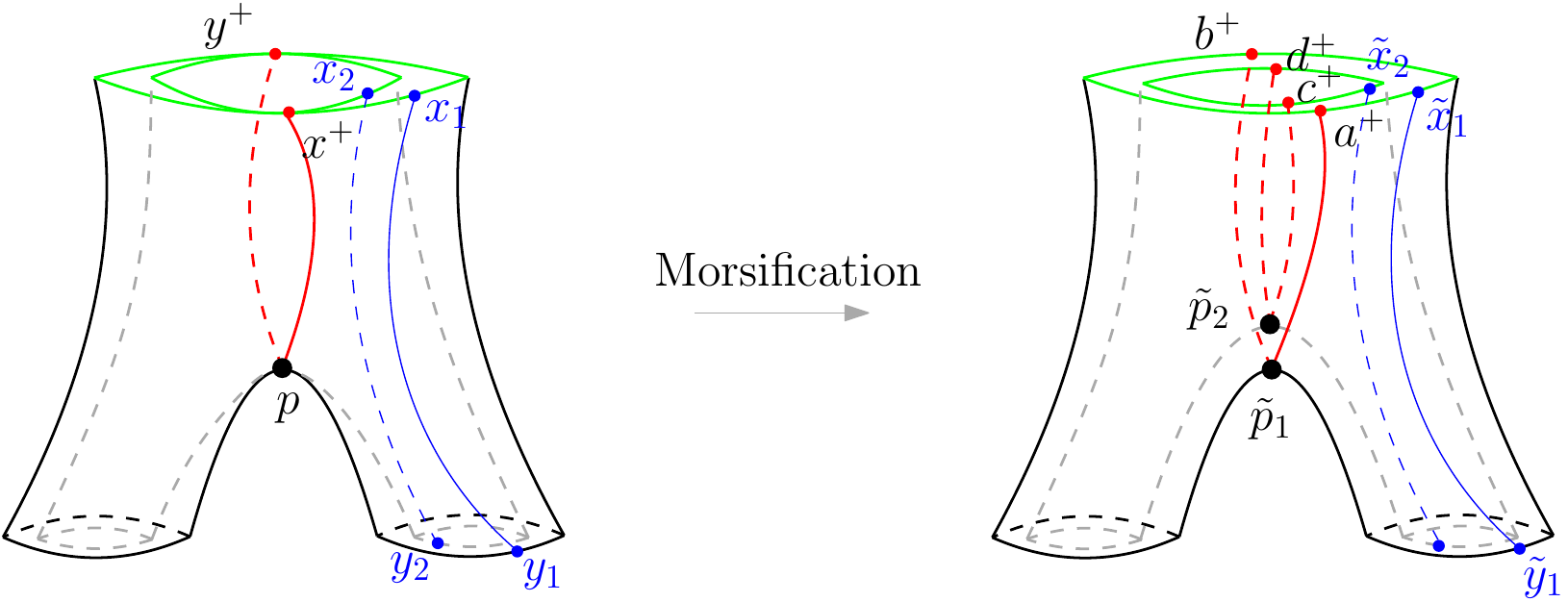}
   \caption{Isolating blocks for  a saddle double crossing singularity of $ss_s$-nature and its Morsification.}\label{flow_duplo4}
\end{figure}

\item[3.2)] This case follows with a similar proof. See Figure~\ref{flow_duplo5}.

\begin{figure}[H]
    \centering
        \includegraphics[width=0.55\textwidth]{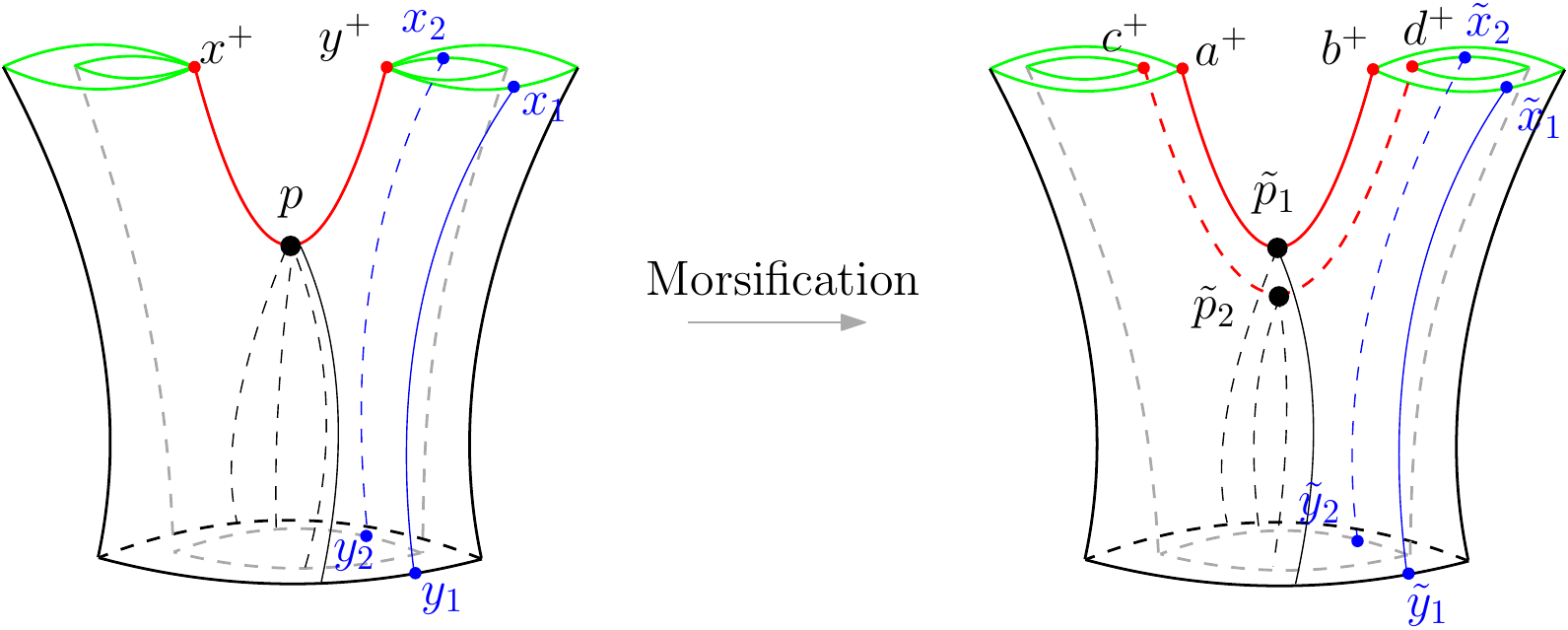}
   \caption{Isolating blocks for  a saddle double crossing singularity of $ss_s$-nature and its Morsification.}\label{flow_duplo5}
\end{figure}

\end{enumerate}
\end{proof}

\begin{proposition}\label{prop_triplo}
Let $M\in\mathfrak{M}(\mathcal{GT})$ be a singular 2-manifold, $X\in\mathfrak{X}_{\mathcal{GT}}(M)$ a GS-vector field on $M$ and $\varphi_{X}$ the GS-flow  associated to $X$.   Given a triple crossing singularity $p$ and  an isolating block $N$ for $p$, there  exists a Morsification $(\widetilde{N},\varphi_{\widetilde{X}})$  such that each singular orbit of $\varphi_{X}$ admits a triplication of  orbits in $N$.
\end{proposition}

\begin{proof}
The proof follows the same steps as the previous one.
\end{proof}

Combinatorially the isolating blocks for double and triple singularities together with its Morsification can be seen as the   Lyapunov (semi)graphs in Figures~\ref{fig_graph_morsef-double} and~\ref{fig_graph_morsef-triple}. By considering the opposite direction on the graphs in Figure~\ref{fig_graph_morsef-double}, we obtain the graphs for attracting  double crossing singularities.

\begin{figure}[H]
    \centering
        \includegraphics[width=0.8\textwidth]{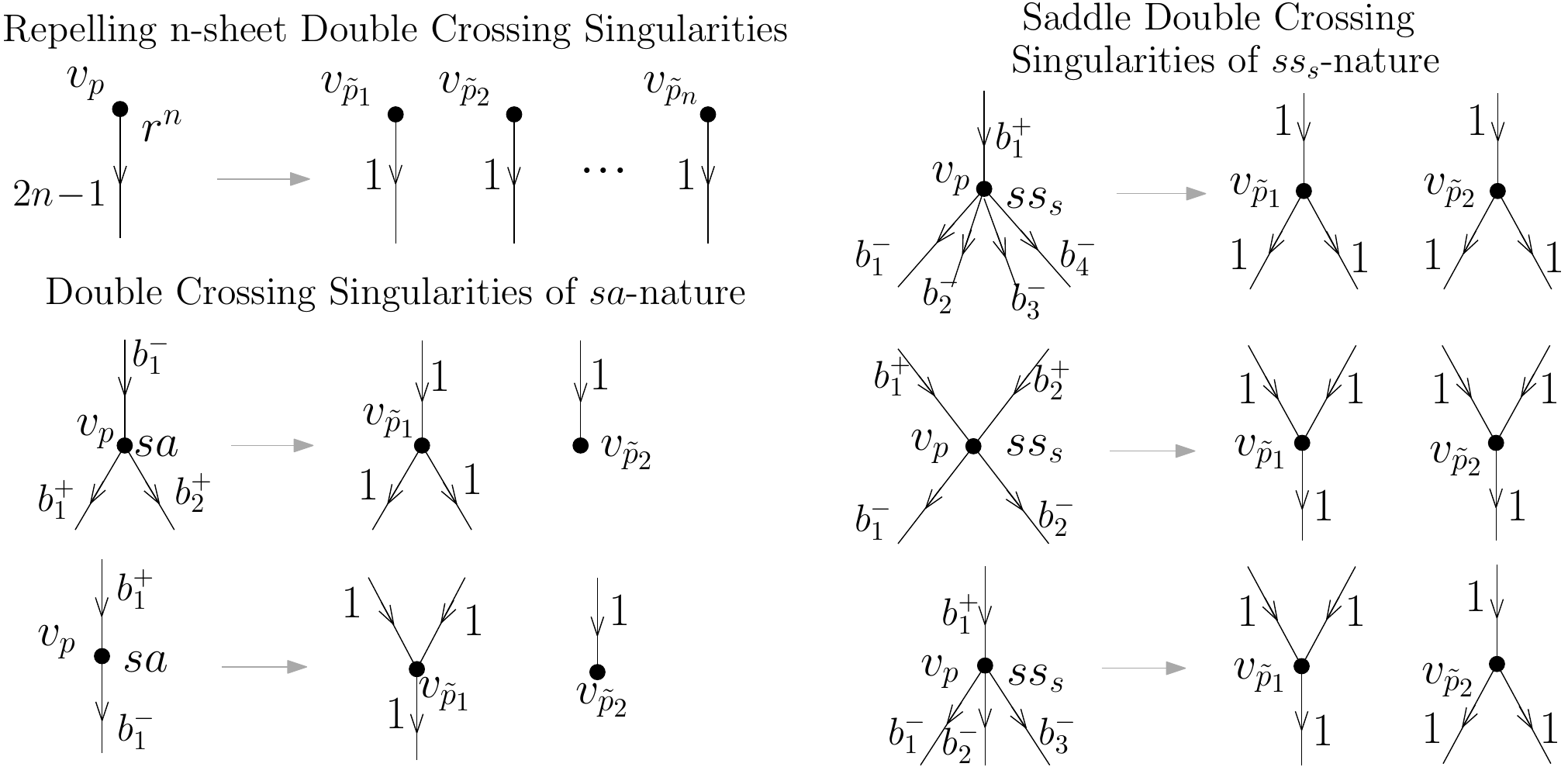}
    \caption{Morsification of a Lyapunov semigraph with vertex associated to a double crossing singularity.}\label{fig_graph_morsef-double}
\end{figure}

There are other isolating blocks for saddle double crossing singularities which are not consider in this work.

\begin{figure}[H]
    \centering
        \includegraphics[width=0.8\textwidth]{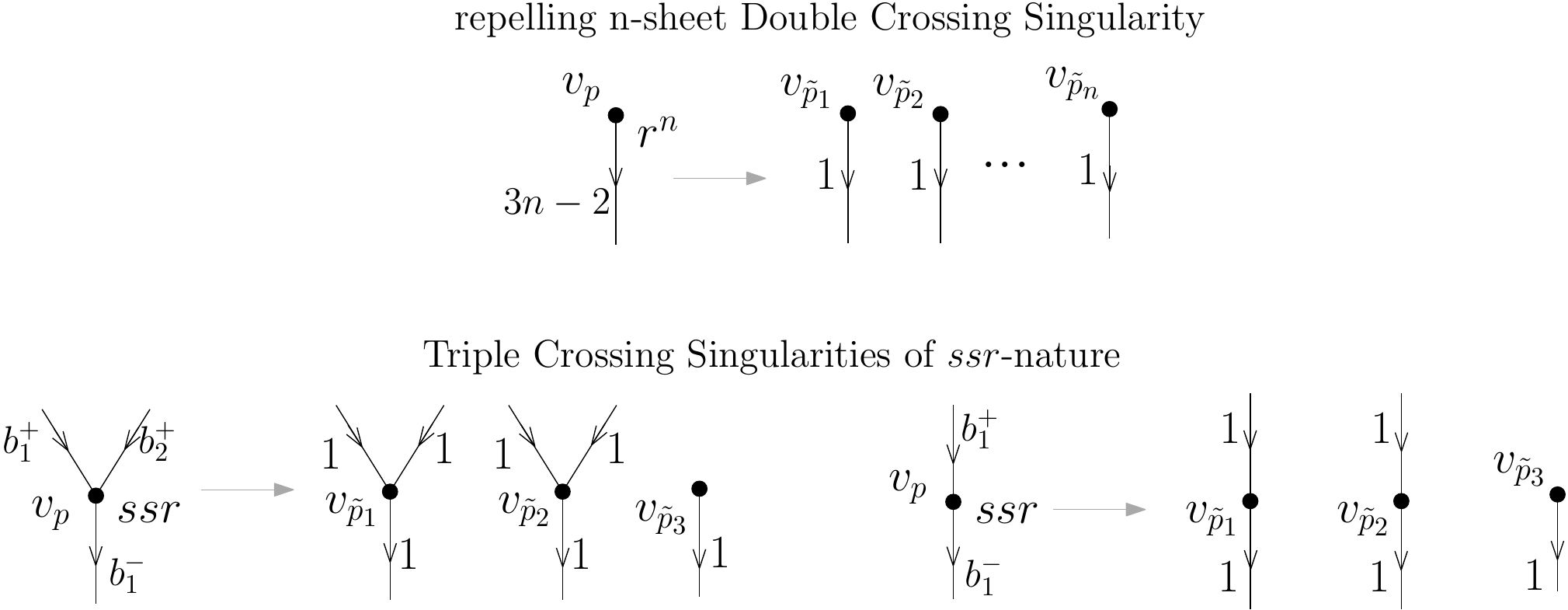}
    \caption{Morsification of a Lyapunov semigraph with vertex associated to a triple crossing singularity.}\label{fig_graph_morsef-triple}
\end{figure}

\section{Gutierrez-Sotomayor Chain Complex}\label{sec:GSCC}

\subsection{Morse Chain Complex}

Let $M$ be a smooth $n$-manifold.
A smooth function $f:M \rightarrow \mathbb{R}$ is called a {\it Morse function} if each critical point of $f$ is nondegenerate, i.e. the Hessian matrix of $f$ at  $p$, $H^{f}_{p}$,   is non-singular. The Morse index $ind_{f}(p)$ of a critical point $p$ is the dimension of the maximal subspace where $H^{f}_{p}$  is negative definite.   Moreover, if $M$ is a closed manifold, then the set of critical points of  a Morse function is finite. 

Fix  a Riemannian metric $g$ on $M$ and let $f:M\rightarrow \mathbb{R}$ be a smooth Morse function. The identity $g(\nabla f,\cdot) = df(\cdot)$
uniquely determines a gradient vector field $\nabla f$ on $M$. Denote the flow associated to $-\nabla f $ by $\varphi_{f}$, which is called the  {\it negative gradient flow}. 
The singularities of the vector field $-\nabla f$ correspond to the critical points of $f$.

A Morse function $f$ is called a {\it Morse-Smale function} if, for each $x,y \in Crit(f)$,   the unstable manifolds  of $\varphi_{f}$ at $x$, $W^{u}(x)$, and the stable manifold of $\varphi_{f}$ at $y$, $W^{s}(y)$, intersect transversally. We define $(f,g)$ as a Morse-Smale pair.
Hereafter, in this subsection, assume that $f$ is a Morse-Smale function, unless stated otherwise. In this case,  the negative gradient flow $\varphi_{f}$ is also called a {\it Morse flow}.

Given $x,y \in Crit(f)$, the {\it connecting manifold} of $x$ and $y$ is given by  
$ {\mathcal{M}}_{ xy}  := W^u(x) \cap W^s(y).$  The connecting manifold ${\mathcal{M}}_{ xy} $ is the set containing  all points $p\in M$ such that $\omega(p) = y$ and $\alpha(p)=x$.
The {\it moduli space} between $x$ and $y$ is defined by
${{\mathcal M}}^{x}_{y}(a):={\mathcal M}_{ xy} \cap f^{-1}(a),$ 
 where $a$ is a regular value between $f(x)$ and $f(y)$.  
The space ${{\mathcal M}}^{x}_{y}(a)$ is  a set of points that are in 1-1 correspondence to the orbits running from $x$ to $y$. For different choices of regular values  $a_{1}, a_{2}$ there is a natural identification between $\mathcal {M}^{x}_{y}(a_{1})$ and  $\mathcal {M}^{x}_{y}(a_{2})$ given by the flow. Hence, one uses the notation $\mathcal {M}^{x}_{y}$ for the moduli space.
Whenever $f$ is a Morse-Smale function, the connecting manifolds and the moduli spaces are orientable closed submanifolds of $M$ of 
 dimensions $ \mathrm{dim} ({\mathcal{M}}_{ xy} )= ind_f (x) - ind_f (y),  $ and $\  \mathrm{dim} ({{\mathcal M}}^{x}_{y}) = ind_f (x) - ind_f (y) - 1$, respectively.

Once orientations are chosen for  $W^{u}(x)$ and $W^{u}(y)$, these induce an orientation on ${\mathcal M}_ {xy}$  denoted by $[{\mathcal M}_ {xy}]_{ind}$, for $x,y\in Crit(f)$.
The procedure given in~\cite{Weber} to obtain this orientation is:
\begin{enumerate}
\item[(1)] If $ind_{f}(y)>0$ , then 
\begin{enumerate}
    \item   Let  ${\mathcal V}_{{\mathcal M}_ {xy}}W^{s}(y)$ be the  normal bundle of $W^{s}(y)$ restricted to ${\mathcal M}_{xy}$. Consider the fiber ${\mathcal V}_{y}W^{s}(y)$ with an orientation given by the isomorphism
    $$ T_{y}W^{u}(y) \oplus T_{y}W^{s}(y) \simeq T_{y}M \simeq {\mathcal V}_{y}W^{s}(y) \oplus T_{y}W^{s}(y).$$
    The orientation on the fiber at $y$ determines an orientation on the normal bundle  ${\mathcal V}_{{\mathcal M}_{xy}}W^{s}(y)$ restricted to the submanifold ${\mathcal M}_{xy}$.

    \item The orientation on ${\mathcal M}_ {xy}$ is determined by the isomorphism
 $T_{{\mathcal M}_{xy}}W^{u}(x) \simeq  T{\mathcal M}_{xy}  \oplus {\mathcal V}_{{\mathcal M}_{xy}}W^{s}(y). $
\end{enumerate}
\item[(2)] If $ind_{f}(y)=0$, then ${\mathcal
V}_{y}W^{s}(y)=0$. Hence, $T_{{\mathcal M}_{xy}}W^{u}(x) \simeq T{\mathcal M}_{xy}.$
\end{enumerate}
Note that there are no restrictions on the orientability of the manifold  $M$.

Given $x,y\in Crit(f)$ with  $ind_{f} (x) - ind_{f} (y) = 1$, let  $u \in
{{\mathcal M}}^{x}_{y}$. The {\it characteristic sign} $n_{u} $ of the orbit ${\mathcal O}(u)$ through $u$ is defined via the identity $[{\mathcal O}(u)]_{ind} = n_u [\dot{u}]$, where $[\dot{u} ]$ and $[{\mathcal O}(u)]_{ind}$ denote the orientations on $\mathcal{O}(u)$ induced by the flow and by $\mathcal{M}_{xy}$, respectively.  The \emph{intersection number} of $x$ and $y$ is defined by
$$n(x,y) = \displaystyle\sum_{u
\in {{\mathcal M}}^{x}_{y}} n_u.
$$

The intersection number between $x$ and $y$ counts, with sign, the flow lines from $x$ to $y$. In the literature there are other ways to count such flow lines with orientations, for example, see~\cite{BH1}.

Fix an arbitrary   orientation for the unstable manifolds $W^{u}(x)$, for each $x \in Crit(f)$, and denote by $Or$ the set of these choices. 
The {\it Morse graded group} $C= \{C_{k}(f)\}$  is defined as the free abelian groups generated by the critical points of $f$  and graded by their Morse index, i.e.,
$$ C_{k}(f) := \bigoplus_{x \in Crit_k(f)} \mathbb{Z} \langle x\rangle ,$$
where $\langle x\rangle$ denotes the pair consisting of the critical point $x$ of $f$ and the orientation chosen on $W^{u}(x)$.
The {\it Morse boundary operator}
$\partial_{k}(x) : {C}_{k}(f) \longrightarrow {C}_{k-1}(f)$
is given on a generator $x$ of ${C}_{k}(f)$ by
\begin{equation}\label{eq_ope_bor_def}
\partial_k\langle x\rangle := \displaystyle\sum_{y \in
Crit_{k-1}(f)} n(x,y) \langle y\rangle ,
\end{equation}
and it is extended by linearity to general chains.

The pair $(C_{\ast}(f),\partial_{\ast})$ is a chain complex, that is, $\partial $ is of degree $-1$ and $\partial \circ \partial = 0$. This chain complex is called a  {\it Morse chain complex}. 

The proof that  $\partial \circ \partial = 0$ follows  by analyzing  the  1-dimensional connected components of the moduli space $\mathcal{M}^{x}_{z}$, where $x\in Crit_{k}(f)$ and $z \in Crit_{k-2}(f)$, which can be either diffeomorphic to $(0,1)$ or to $S^1$, as in Figure~\ref{mod-space-morse}.
In~\cite{Weber}, it is proved that if $(u,v)$ and $(\tilde{u},\tilde{v})$ are two broken flow lines corresponding to the ends of a noncompact connected component of $\mathcal{M}^x_z$, then the {\it null cycle condition} is satisfied, i.e. $n_un_v + n_{\tilde{u}}n_{\tilde{v}}=0$.

\begin{figure}[h]
\centering
\includegraphics[scale=1]{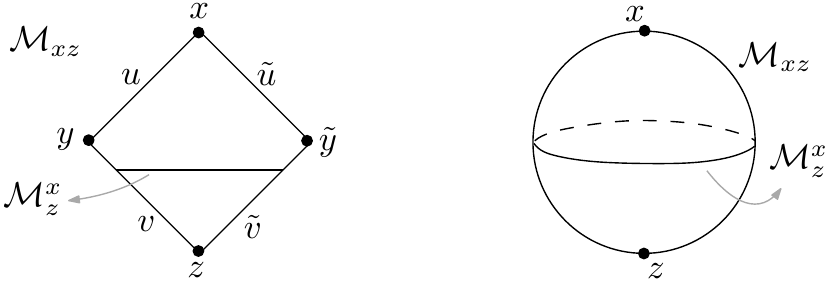}
\caption{Possible connected components of $\mathcal{M}_{xz}$ for  $x\in Crit_{k}(f)$ and $z \in Crit{k-2}(f)$. }\label{mod-space-morse}
\end{figure}

The {\it Morse homology groups} with integer coefficients are defined by
$$ HM_{k}(M,f,g,Or;\mathbb{Z}) = \dfrac{\Ker \ \partial_{k} }{\text{Im}  \ \partial_{k+1}}, \quad \forall k\in \mathbb{Z}.
$$
In~\cite{Weber}, it was proved  that,  for two choices of Morse-Smale pairs $(f^{1},g^{1})$ and $(f^{2},g^{2})$ with  orientations $Or^{1}$ and $Or^{2}$ on all unstable manifolds, the associated Morse homology groups $HM_{k}(M,f^{1},g^{1},Or^{1};\mathbb{Z})$ and $HM_{k}(M,f^{2},g^{2},Or^{2};\mathbb{Z})$ are naturally  isomorphic, for all $k\in \mathbb{Z}$. Hence, this homology is simple denoted by $HM_{\ast}(M,\mathbb{Z})$. Moreover, one has that 
$ HM_{\ast}(M;\mathbb{Z}) \cong H^{sing}(M;\mathbb{Z}),$
i.e., the Morse homology of $M$ is isomorphic to the singular homology of $M$.

\subsection{Gutierrez-Sotomayor Chain Complex}

Let $M\in\mathfrak{M}(\mathcal{GS})$ be a compact singular 2-manifold, $X\in\mathfrak{X}_{\mathcal{GS}}(M)$ a GS-vector field on $M$ and $\varphi_{X}$
 the Gutierrez-Sotomayor flow on $M$ associated to $X$. In this section, one  defines a chain complex for a given GS-flow analogous to the Morse chain complex of a Morse-Smale flow. We will start by obtaining the characteristic signs of the flow lines on $M$ from the characteristic signs of the flow lines on the smooth surface $\widetilde{M}$ obtained by a Morsification process. Subsequently, it is possible to define a GS-chain group and a GS-boundary map, as in the Definition~\ref{def_complex_GS}. 


Given $x,y \in Sing(X)$, define the \emph{\textbf{connecting manifold}} of $x$ and $y$ by
$${\mathcal{M}}_{ xy} := {\mathcal M}_ {xy}(X,M) := W^u(x) \cap W^s(y),$$
where $W^s$, $W^u$ are the stable and unstable sets of the singularity, respectively. In other words, the connecting manifold ${\mathcal{M}}_{ xy} $ is composed by the points $p \in M$ such that $\omega(p) = y$ and $\alpha(p)=x$.
The \emph{\textbf{moduli space}} between the singularities $x$ and $y$ is defined as the quotient of the connecting manifold ${\mathcal{M}}_{ xy}$ by the natural action of $\mathbb{R}$ on the flow lines, i.e.,
$$ {\mathcal{M}}_{y}^{x} := {\mathcal M}_{xy} / \mathbb{R}.$$
%

Define the nature numbers of a GS-singularity as follows:
\begin{definition} Denote by 
$Sing(X)$  the set of singularities of a vector field $X\in\mathfrak{X}_{\mathcal{GS}}(M)$.
Given $p\in Sing(X)$, 
define $\eta_k(p)$ as the $k$-th \textbf{nature number} of $p$, where:
\begin{itemize}
\item $k=2$ represents the repelling nature $r$;
\item $k=1$ represents the saddle nature $s$;
\item $k=0$ represents the attracting nature $a$.
\end{itemize}
\end{definition} 

Two singularities  $x$  and $y$ are said to be \emph{\textbf{consecutive}}  if $\eta_k(x)$ and $\eta_{k-1}(y)$ are both non zero, for some $k=1,2$.

For example, if $p$ is a  triple crossing singularity of  $ssa$ nature, we have that $\eta_2(p)=0$, $\eta_1(p)=2$ e $\eta_0(p)=1$.

Note that in the Morse-Smale case, each singularity of  index $k$ has only one nature, implying that it contributes with only one generator  for the $k$-th Morse chain group.  
The same holds for  cone singularities and  Whitney singularities.  However, this is not the case for the double and triple crossing singularities, since they have at least two natures. Hence, these type of singularities will have more than one generator  in the GS-chain groups associated with them. 
Moreover, for a  double or triple crossing singularity $x$,  the singularities associated to $x$ by the Morsification process are in one-to-one correspondence with the collection of nature numbers of $x$.

In order to distinguish the generators provided by  a GS-singularity $x$, we denote the generators of the nature of a singularity $x$ by 
$$\{h_k^i(x) \mid \ i=1,\dots ,\eta_k(x),  k=0,1,2\},$$
where $h_k^i(x)$ represents a generator of $k$-nature of the singularity $x$. The advantage  of this notation is that this set  $\{h_k^i(x) \mid  x\in Sing(X), \ i=1,\dots ,\eta_k(x)\}$ will generate the  $k$-chain group of  the GS-chain complex that we define below.

\begin{definition}\label{def_complex_GS}     
Given a GS-flow $\varphi_X$, the \textbf{Gutierrez-Sotomayor chain group}  $C^{\mathcal{GS}}_k(M,X) $ with integer coefficients  graded by the nature of the singularities is the free abelian group generated by the set  of GS-singularities $Sing(X)$ of the vector field  $X$, i.e.:
$$ C^{\mathcal{GS}}_k(M,X) := \bigoplus_{x \in Sing(X)} \Bigg( \bigoplus_{i=1}^{\eta_k(x)}\mathbb{Z} \langle h_k^i(x)\rangle \Bigg), \ \ \ \   k \in\mathbb{Z},$$
where $h_k^i(x)$ denotes a generator associated to the $k$-nature of the  singularity $x$.
The $ k$-th \textbf{Gutierrez-Sotomayor boundary map},
$\Delta^{\mathcal{GS}}_{k}:{C}^{\mathcal{GS}}_{k}(M,X) \rightarrow {C}^{\mathcal{GS}}_{k-1}(M,X)$, is given on a generator $h_k^i(x)$ by
$$\Delta^{\mathcal{GS}}_k\langle h_k^i(x)\rangle := \displaystyle\sum_{y \in Sing(X)}\Bigg(\displaystyle\sum_{j=1}^{\eta_{k-1}(y)} n(h_k^i(x),h_{k-1}^j(y)) \langle h_{k-1}^j(y)\rangle \Bigg),$$
and it is extended by linearity to general chains.
\end{definition}

The overarching idea is to make use of the Morsification of the GS-flow defined in Section~\ref{sec:MGSVFIB}, in order to define a GS-intersection number from the   Morse counterpart.  Intersection numbers depends heavily on the smooth structure of the manifold, i.e, the existence of tangent and normal bundles.
The number $n(h_k^i(x),h_{k-1}^j(y))$ is called  the GS-intersection number  of the  generators $h_k^i(x)$ and $h_{k-1}^j(y)$ and will be defined in the following subsections as the sum $\displaystyle\sum n_u,$ over all flow lines  $u\in \mathcal{M}^{h_k^i(x)}_{h_{k-1}^j(y)}$, where $n_u$ is the GS-characteristic sign  of the flow line  $u$. This process relies on the Morsification process.

Furthermore, in the following subsections, we prove that, given $M\in \mathfrak{M}(\mathcal{GS}) $ and $ X\in \mathfrak{X}_{\mathcal{GS}}(M)$, the pair  $(C^{\mathcal{GS}}_*(M,X) , \Delta^{\mathcal{GS}}_*)$ 
is  a chain complex which we refer to as a \textbf{\emph{Gutierrez-Sotomayor chain complex}}. 


Throughout this section, a GS-chain complex will be defined for flows associated to vector fields  restricted to flows associated to vector fields $X$ in $\mathfrak{X}_{\mathcal{GC}}(M)$, $\mathfrak{X}_{\mathcal{GW}}(M)$, $\mathfrak{X}_{\mathcal{GD}}(M)$  and $\mathfrak{X}_{\mathcal{GT}}(M)$.

\subsection{Gutierrez-Sotomayor complex for cone  singularities}

In the previous section, we defined the Morsification process of a given   GS-flow containing only regular and cone type singularities on  a singular 2-manifold $M\in\mathfrak{M}(\mathcal{GC})$  in order to obtain a smooth manifold   $\widetilde{M}\in\mathfrak{M}(\mathcal{R})$ with a smooth Morse flow $\widetilde{\varphi}_{\widetilde{X}}$  on it.
Therefore, one can attach to each flow line of $\widetilde{\varphi}_{\widetilde{X}}$  a characteristic sign. Now, the idea is to transfer  these signs to the corresponding flow lines of the singular flow $\varphi_X$.

We now define the transfer process of  characteristic signs from the Morse setting to the GS-setting.

\begin{definition}[Characteristic signs of flows lines of $\mathfrak{X}_{\mathcal{GC}}(M)$]\label{characteristicGS} Consider  $x,y \in Sing(X)$ singularities of consecutive natures, where  $X\in \mathfrak{X}_{\mathcal{GC}}(M)$.
The \textbf{GS-characteristic sign} $n_{u}$ of a flow line $u \in \mathcal{M}_{xy} $ is defined as follows:
\begin{enumerate}
%
\item Let   $x$ be a singularity of repeller nature and $y$ a cone singularity  of saddle nature. Denote by $\tilde{y},\tilde{y}'$ the singularities associated to $y$ by the Morsification process, and by   $\tilde{x}$ the singularity of repeller nature associated to $x$.  It is easy to see that $\mathcal{M}^{x}_{y}\approx\widetilde{\mathcal{M}}^{\tilde{x}}_{\tilde{y}}\approx\widetilde{\mathcal{M}}^{\tilde{x}}_{\tilde{y}'}$.
 Hence, given $u\in\mathcal{M}^{x}_{y}$ there are corresponding flow lines   $\tilde{u}\in\widetilde{\mathcal{M}}^{\tilde{x}}_{\tilde{y}}$ and $\tilde{u}'\in\widetilde{\mathcal{M}}^{\tilde{x}}_{\tilde{y}'}$ in the Morsified flow.
Define
$$ n_u :=
\left\{
\begin{array}{ll}
n_{\tilde{u}}, & \text{if}  \ \ n_{\tilde{u}}=n_{\tilde{u}'}\\
\ 0, & \text{if}  \ \  n_{\tilde{u}}\neq n_{\tilde{u}'}\\
\end{array}
\right. 
$$
\item Let   $x$ be a cone singularity  of saddle nature and $y$ a singularity of attractor nature. Denote by $\tilde{x},\tilde{x}'$ the singularities associated to $x$ by the Morsification process, and by   $\tilde{y}$ the singularity of repeller nature associated to $y$.  It is easy to see that $\mathcal{M}^{x}_{y}\approx\widetilde{\mathcal{M}}^{\tilde{x}}_{\tilde{y}}\approx\widetilde{\mathcal{M}}^{\tilde{x}'}_{\tilde{y}}$.
Hence, given $u\in\mathcal{M}^{x}_{y}$ there are corresponding flow lines $\tilde{u}\in\widetilde{\mathcal{M}}^{\tilde{x}}_{\tilde{y}}$ and $\tilde{u}'\in\widetilde{\mathcal{M}}^{\tilde{x}'}_{\tilde{y}}$ in the Morsified flow.
Define
$$ n_u :=
\left\{
\begin{array}{ll}
n_{\tilde{u}}, & \text{if}  \ \ n_{\tilde{u}}=n_{\tilde{u}'}\\
\ 0, & \text{if}  \ \  n_{\tilde{u}}\neq n_{\tilde{u}'}\\
\end{array}
\right. 
$$
\item For the other  cases, one has that $\mathcal{M}^{x}_{y}\approx\widetilde{\mathcal{M}}^{\tilde{x}}_{\tilde{y}}$. For each
$u\in\mathcal{M}^{x}_{y}$, define  $n_u := n_{\tilde{u}}$. 
\end{enumerate}
\end{definition}

Once the characteristic signs are well defined for the flow lines of $\varphi_{X}$  with $X\in\mathfrak{X}_{\mathcal{GC}}(M)$,  the \emph{\textbf{GS-intersection number}} between consecutive singularities $x$ and $y$ is defined as
$$n(x,y):=\sum_{u\in\mathcal{M}^x_y} n_u,$$
This sum is finite since the moduli space is compact.
i.e., the sum of the characteristic signs of the flow lines connecting $x$ to $y$. Thus, the GS-boundary map $
 \Delta_*^{\mathcal{GC}}$, described  in  Definition~\ref{def_complex_GS}, is well defined for GS-flows  on  
$M\in\mathfrak{M}(\mathcal{GC})$. 
Now, we will prove that $(C^{\mathcal{GC}}_*(M,X),\Delta_*^{\mathcal{GC}})$ is in fact a chain complex for $X\in\mathfrak{X}_{\mathcal{GC}}(M)$ and $M\in\mathfrak{M}(\mathcal{GC})$.

\begin{lemma}\label{lemmagrafoGS}
Let $G_{\mathcal{GC}}$ be the  graph of  the  matrix associated to the GS-boundary map $\Delta^{\mathcal{GC}}$ and let  $y\in Sing(X)$ be a cone singularity of saddle nature.
Therefore, the incidence degree of the vertex  $v_y$ is null or $2$, and in the latter case, the two edges are both positively or negatively incident in $v_y$.  Consequently, there is no cycle in $G({\Delta^\mathcal{GC}})$ containing  $v_y$.
\end{lemma}

\begin{proof}
Let $y\in Sing(X)$   be a cone singularity of saddle nature. Since, $y$ has saddle nature, the positively and negatively incident edges appear in pairs and up to two positively incident edges and two negatively incident edges. Hence, the incident degree of  $v_y$ belongs to the set $\{0,2,4\}$.

Suppose that the incident degree of the vertex $v_y$ is four.  This is also the case for the vertices  $v_{\tilde{y}}$ and  $v_{\tilde{y}'}$ of the Morsified boundary map $\widetilde{\Delta}$, where  $\tilde{y},\tilde{y}'$ are the singularities associated to $y$ by the Morsification process. Therefore, there are  two vertices 
$v_{\tilde{x}_1}$, $v_{\tilde{x}_2}$     corresponding to  repeller singularities and two vertices 
$v_{\tilde{z}_1}$, $v_{\tilde{z}_2}$ corresponding to attractor singularities which   belong to  distinct cycles in $G(\widetilde{\Delta})$ involving the vertices  $v_{\tilde{y}}$ and  $v_{\tilde{y}'}$. Moreover, these cycles can be chosen matching the ends of noncompact connected components of  the moduli spaces $\widetilde{\mathcal{M}}_{\tilde{x}_1\tilde{z}_1}$ and $\widetilde{\mathcal{M}}_{\tilde{x}_2\tilde{z}_2}$. Let   $(\tilde{u}_i,\tilde{v}_i)\in\widetilde{\mathcal{M}}^{\tilde{x}_i}_{\tilde{y}}\times\widetilde{\mathcal{M}}^{\tilde{y}}_{\tilde{z}_i}$ and 
$(\tilde{u}'_i,\tilde{v}'_i)\in\widetilde{\mathcal{M}}^{\tilde{x}_i}_{\tilde{y}'}\times\widetilde{\mathcal{M}}^{\tilde{y}'}_{\tilde{z}_i}$, 
be the broken orbits that correspond the ends of such component of  $\widetilde{\mathcal{M}}_{\tilde{x}_i\tilde{z}_i}$,  for $i=1,2$, see Figure~\ref{figura_lema}.

\begin{figure}[H]
    \centering
        \includegraphics[width=0.45\textwidth]{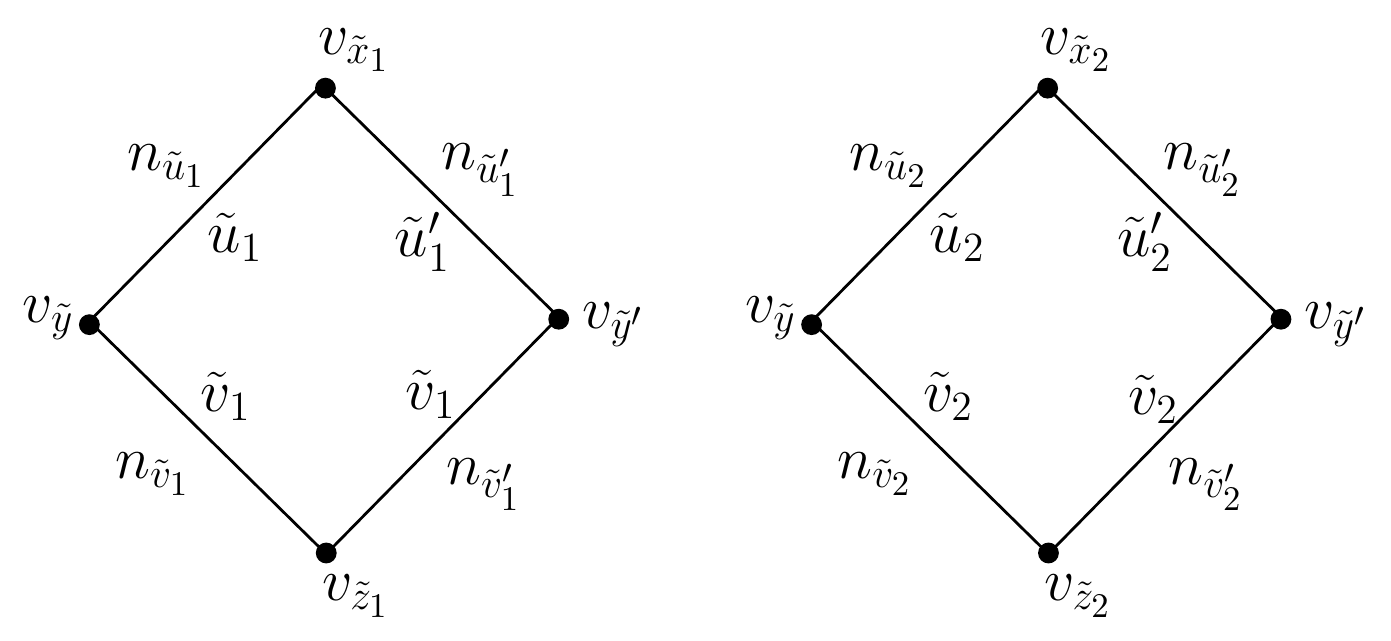} \vspace{-0.3cm}
   \caption{Connected components of  $\widetilde{\mathcal{M}}_{\tilde{x}_1\tilde{z}_1}$ and $\widetilde{\mathcal{M}}_{\tilde{x}_2\tilde{z}_2}$. }\label{figura_lema}
\end{figure}

Since these spaces correspond to moduli spaces of order two of a Morse flow, then 
$n_{\tilde{u}_i}.n_{\tilde{v}_i}+n_{\tilde{u}'_i}.n_{\tilde{v}'_i}=0$, for $i=1,2$, as proved in~\cite{Weber}. Thus, for each $i=1,2$, there are exactly two possibilities for the characteristic signs:
$$ \left\{
\begin{array}{ll}
n_{\tilde{u}_i} = n_{\tilde{u}'_i} \\
n_{\tilde{v}_i} \neq n_{\tilde{v}'_i}
\end{array}
\right. 
\hspace{.5cm}
or
\hspace{.7cm}
\left\{
\begin{array}{ll}
n_{\tilde{u}_i} \neq n_{\tilde{u}'_i} \\
n_{\tilde{v}_i} = n_{\tilde{v}'_i}
\end{array}.
\right. 
$$

If the first (resp., second) possibility holds, then by the sign transfer process, one has that
$n_{v_i} =0$   (resp., $n_{u_i} =0$). See Figure~\ref{figura_lema2}. In any case, it contradicts  the assumption that there are four incident edges to the vertex $v_y$ of
$G(\Delta^{\mathcal{GC}})$. Hence, the incidence degree of $v_y$  is either $0$ or $2$.

\begin{figure}[H]
    \centering
        \includegraphics[width=0.40\textwidth]{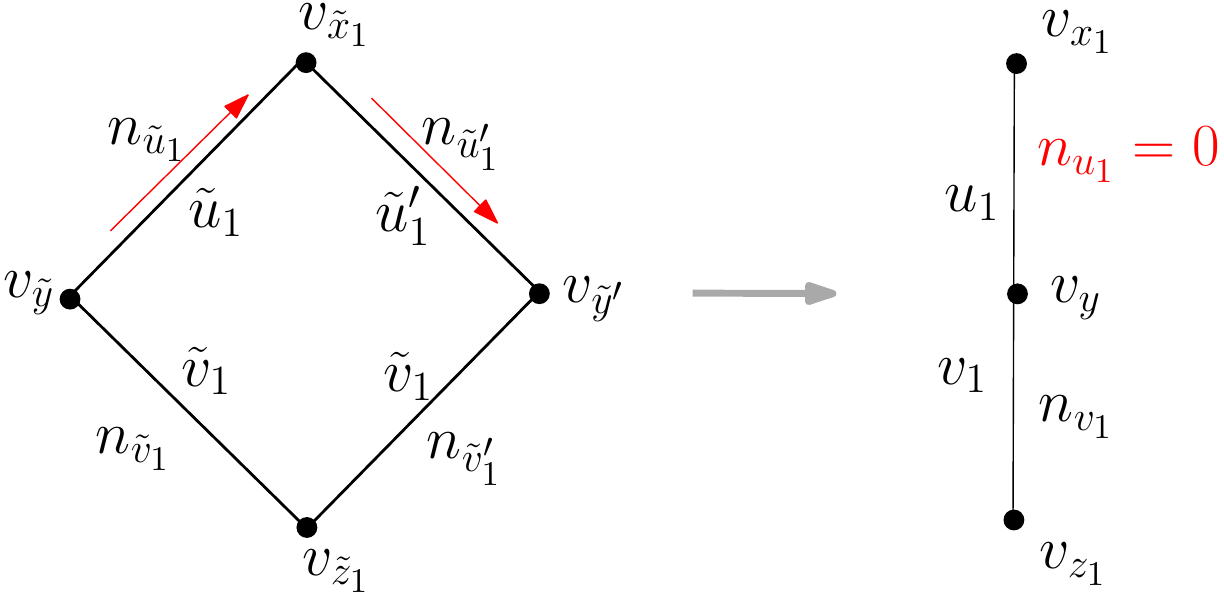}\hspace{2cm}
           \includegraphics[width=0.40\textwidth]{fig_lemma1-2.pdf}
        \vspace{-0.3cm}
        
   \caption{Transfer of characteristic signs of the ends of the noncompact connected components of  $\widetilde{\mathcal{M}}_{\tilde{x}_1\tilde{z}_1}$ and $\widetilde{\mathcal{M}}_{\tilde{x}_2\tilde{z}_2}$ }\label{figura_lema2}
\end{figure}

Now, we need to prove that if the incidence degree of  $v_y$  is $2$ then both edges are positively incident or both are negatively incident to $v_y$. Since  the characteristic signs of the two flow lines on the unstable manifold of a saddle are opposite, then $n_{\tilde{v}_2}= - n_{\tilde{v}_1}$ and $ n_{\tilde{v}'_2} = - n_{\tilde{v}'_1}$. Thus: 
$$ \left\{
\begin{array}{ll}
n_{\tilde{u}_1} = n_{\tilde{u}'_1} \\
n_{\tilde{v}_1} \neq n_{\tilde{v}'_1}
\end{array}
\right. 
\hspace{.5cm}
\Rightarrow
\hspace{.7cm}
\left\{
\begin{array}{ll}
n_{\tilde{u}_2} = n_{\tilde{u}'_2} \\
n_{\tilde{v}_2} \neq n_{\tilde{v}'_2}
\end{array}
\right. 
$$
and
$$ \left\{
\begin{array}{ll}
n_{\tilde{u}_1} \neq n_{\tilde{u}'_1} \\
n_{\tilde{v}_1} = n_{\tilde{v}'_1}
\end{array}
\right. 
\hspace{.5cm}
\Rightarrow
\hspace{.7cm}
\left\{
\begin{array}{ll}
n_{\tilde{u}_2} \neq n_{\tilde{u}'_2} \\
n_{\tilde{v}_2} = n_{\tilde{v}'_2}
\end{array}.
\right. 
$$
In other words, if the incidence degree of  $v_y$ is two then the two edges are  positively or negatively incident to it.
\end{proof}

The previous lemma is essential in the proof that  $(C^{\mathcal{GC}}_*(M,X),\Delta_*^{\mathcal{GC}})$ is a chain complex.

\begin{theorem}\label{prop_operador}
Let $\Delta_{*}^{\mathcal{GC}}$ be the GS-boundary map associated to $\varphi_{X}$, where $X\in \mathfrak{X}_{\mathcal{GC}}(M)$. Then  $\Delta_{k-1}^{\mathcal{GC}}\circ \Delta_{k}^{\mathcal{GC}} = 0$, for all $k\in\mathbb{Z}$.
\end{theorem}

\begin{proof}
Given a  singularity of repeller nature $x$  and  a singularity of attractor nature $z$, consider  
$${\mathcal B}^{1}_{xz}:=\{ (u,v) \ | \ u\in \mathcal{M}^x_y, v \in \mathcal{M}^y_z, \ \text{for } \ y\in Sing(X) \ \text{of saddle nature}\}.$$
With this notation, one can write the composition  $\Delta_{k-1}^{\mathcal{GC}}\circ \Delta_{k}^{\mathcal{GC}}$ as follows:
 \begin{eqnarray}
 \Delta^{\mathcal{GC}}_{k-1}\circ \Delta^{\mathcal{GC}}_{k}(x) & = &\displaystyle\sum_{z\in Sing(X)}\Bigg(\displaystyle\sum_{y\in
        Sing(X)}n(x,y)n(y,z)\Bigg)z  \nonumber \\
        &= &\displaystyle\sum_{z\in Sing(X)}\Bigg(\displaystyle\sum_{y\in
        Sing(X)}\displaystyle\sum_{u\in\mathcal M^x_y}\displaystyle\sum_{v\in\mathcal M^y_z}n_{u}n_{v} \Bigg)z \nonumber  \\
        &=&\displaystyle\sum_{z\in Sing(X)}\Bigg(\displaystyle\sum_{(u,v)\in {\mathcal B}_{xz}^{1}}n_{u}n_{v}\Bigg)z\nonumber \\
      &=& \displaystyle\sum_{z\in Sing(X)}\left(  \displaystyle\sum(n_{u_{i}}n_{v_{i}}+n_{{u}_{j}}n_{{v}_{j}}) \right)z\nonumber
       \end{eqnarray}
where the sum in the last equality is over the ends of the connected components of $\mathcal{M}_{z}^x$, for all $z\in Sing(X)$ of attractor  nature.

In terms of  the graph of the map $G(\Delta^{\mathcal{GC}})$, fixing $x$ and $z$, each term  $n_{u_{i}}n_{v_{i}}+n_{{u}_{j}}n_{{v}_{j}}$ of the last sum corresponds to a 
cycle in $G({\Delta^\mathcal{GC}})$ connecting the vertices  $v_x$ and $v_z$. By Lemma~\ref{lemmagrafoGS}, no cycle in $G({\Delta^\mathcal{GC}})$ contains a vertex $v_y$, where $y$ is a cone singularity of saddle nature. Therefore, the cycles in the graph $G({\Delta^\mathcal{GC}})$ are also cycles in the graph $G(\widetilde{\Delta})$, where $\widetilde{\Delta}$ is the Morsified boundary map. Since the cycles we are considering correspond to  the ends of noncompact connected components of the moduli space $\mathcal{M}_{x}^z$ of order 2, then    $n_{u_{i}}n_{v_{i}}+n_{{u}_{j}}n_{{v}_{j}}=0$.
It follows that $\Delta^{\mathcal{GC}}_{k-1}\circ\Delta^{\mathcal{GC}}_k =0$.
\end{proof}

We have shown that the pair  $(C^{\mathcal{GC}}_*(M,X),\Delta_*^{\mathcal{GC}})$ is in fact a chain complex whenever $X \in \mathfrak{X}_{{\mathcal{GC}}}(M)$ and $M\in\mathfrak{M}(\mathcal{GC})$.

\begin{example}
 Consider a GS-flow $\varphi_{X}$  defined on a singular manifold $M$  and its Morsification $(\widetilde{M},\varphi_{\widetilde{X}})$ as in Figure~\ref{figura_ex_cone_man}, where the characteristic sign transfer process is illustrated. Consider as well, the choice of orientations on the unstable manifolds  of the critical points of $\widetilde{M}$.  The GS-characteristic signs on the orbits of $\varphi_{X}$  are obtained from this choice as shown in Figure~\ref{figura_ex_cone_man}.  



\begin{figure}[H]
    \centering
        \includegraphics[width=0.8\textwidth]{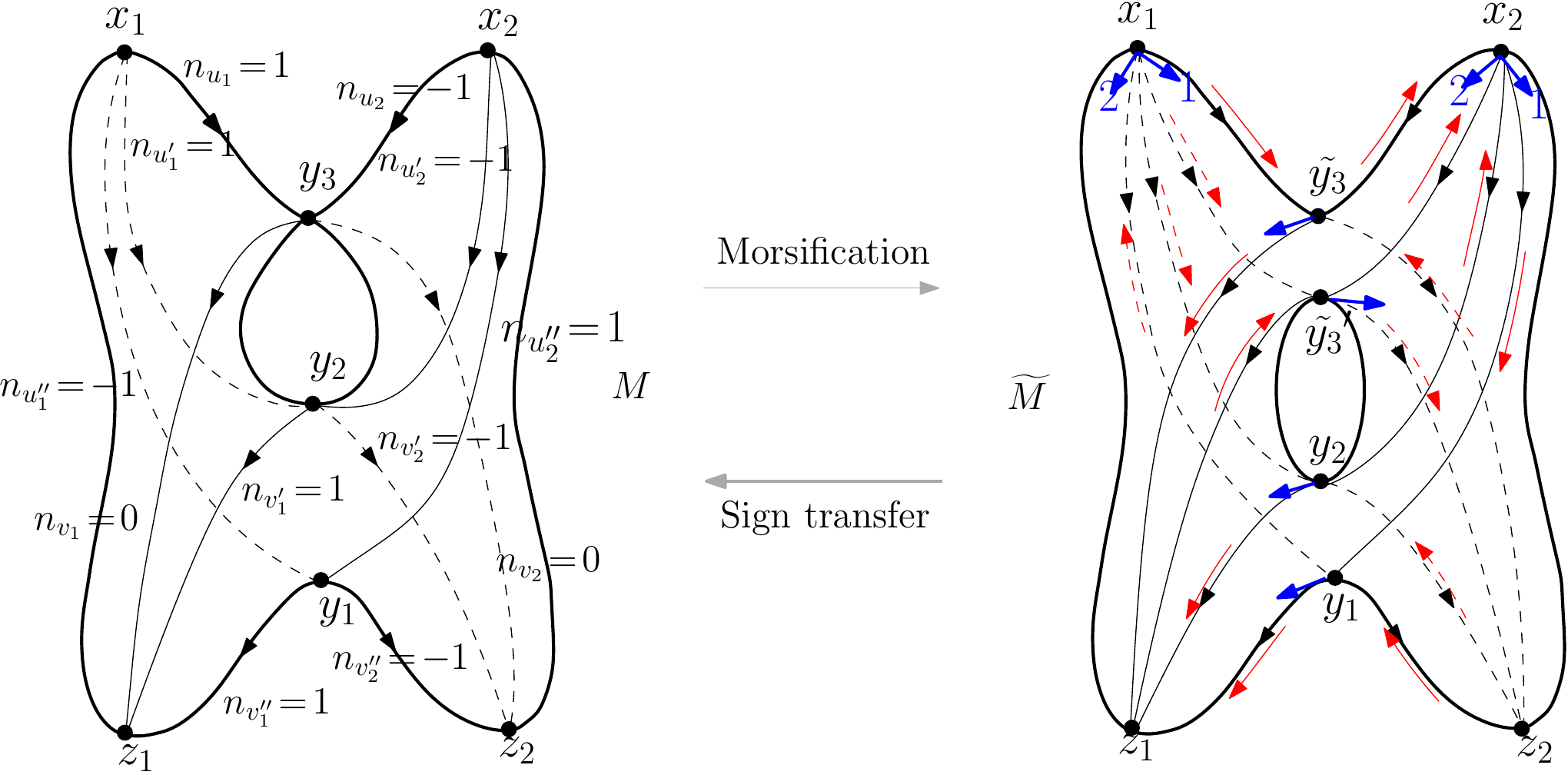}
   \caption{A GS-flow on a pinched torus with cone singularities and its Morsification.}\label{figura_ex_cone_man}
\end{figure}


Let us examine this example in more detail. For instance, in the sign transfer process, consider  the connecting manifold $\widetilde{\mathcal{M}}_{x_1z_2}$ with its ends given by the broken flow lines $(\tilde{u_1},\tilde{v_2})$ and $(\tilde{\tilde{u_1}},\tilde{\tilde{v_2}})$. 
One has that  $n_{\tilde{u_1}}=1=n_{\tilde{\tilde{u_1}}}$. On the other hand, since  $n_{u_1}=1$ and
$n_{\tilde{v_2}}=-1\neq1=n_{\tilde{\tilde{v_2}}}$, one has that $n_{v_2}=0$.  See Figure~\ref{figura_ex_characteristic2}.
Analogously, the same analysis holds for $\mathcal{M}_{x_2z_1}$, obtaining
$n_{u_2}=-1$ and $n_{v_1}=0$.  

\begin{figure}[H]
    \centering
        \includegraphics[width=0.450\textwidth]{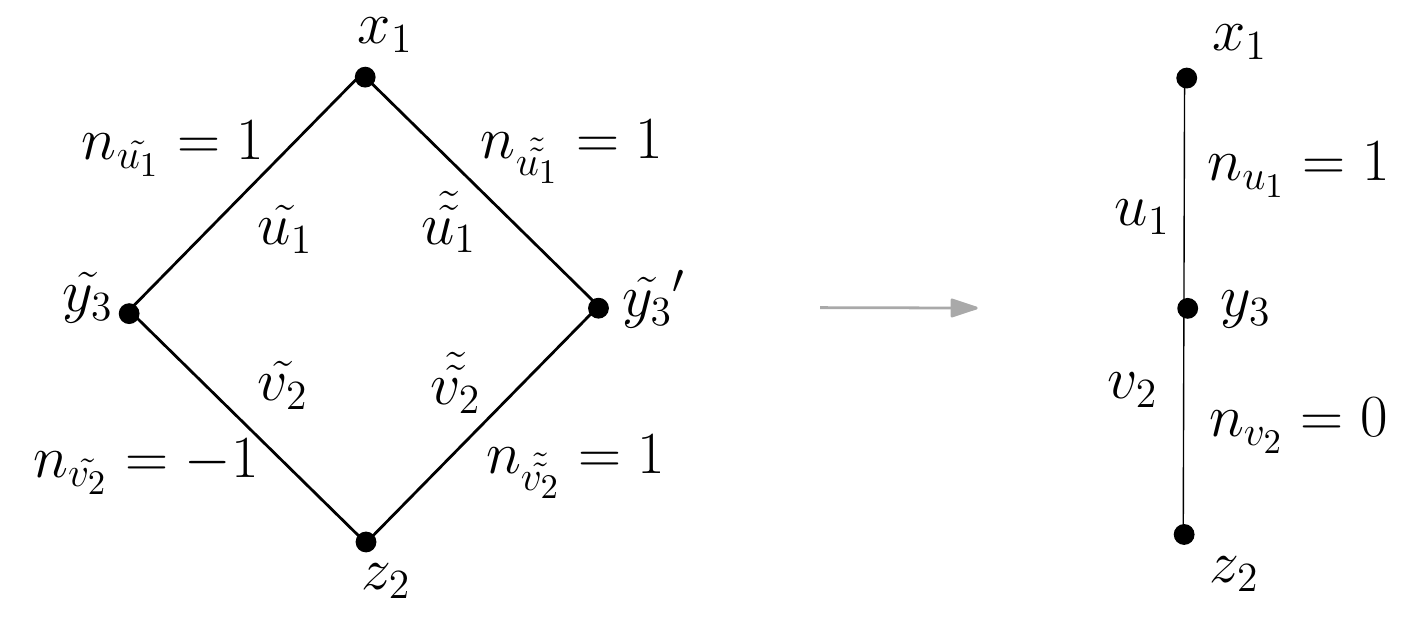}
   \caption{Characteristic sign tranfer.}\label{figura_ex_characteristic2}
\end{figure}

The remaining orbits inherit the same characteristic signs as in $\widetilde{M}$, since all but  $y_3$ are regular singularities.

%

The GS-chain groups are:
$ C_2^{\mathcal{GC}}(M,X)=\mathbb{Z} \langle x_1\rangle \oplus \mathbb{Z} \langle x_2\rangle$,
$C_1^{\mathcal{GC}}(M,X)=\mathbb{Z} \langle y_1\rangle \oplus \mathbb{Z} \langle y_2\rangle \oplus\mathbb{Z} \langle y_3\rangle $,
$C_0^{\mathcal{GC}}(M,X)=\mathbb{Z} \langle z_1\rangle \oplus \mathbb{Z} \langle z_2\rangle$
and $C_k^{\mathcal{GC}}(M)=0, k\neq 0,1,2$. The GS-intersection numbers are:
$n(x_1,y_1)=n_{u_1''}=-1, \ n(x_2,y_1)=n_{u_2''}=1, \ 
n(x_1,y_2)=n_{u_1'}=1, \ n(x_2,y_2)=n_{u_2'}=-1, \ 
n(x_1,y_3)=n_{u_1}=1, \ n(x_2,y_3)=n_{u_2}=-1, \ 
n(y_1,z_1)=n_{v_1''}=1, \ n(y_1,z_2)=n_{v_2''}=-1, \ 
n(y_2,z_1)=n_{v_1'}=1, \ n(y_2,z_2)=n_{v_2'}=-1, \ 
n(y_3,z_1)=n_{v_1}=0$ and $n(y_3,z_2)=n_{v_2}=0$. Therefore, the  GS-boundary maps  $\Delta^{\mathcal{GC}}_2:C_2\rightarrow C_1$, $\Delta^{\mathcal{GC}}_1:C_1\rightarrow C_0$ and $\Delta_0^{\mathcal{GC}}:C_0\rightarrow 0$ are given by:
$$ \Delta^{\mathcal{GSC}}_2(x_1)=-\langle y_1\rangle + \langle y_2 \rangle + \langle y_3 \rangle, \ 
\Delta_2^{\mathcal{GC}}(x_2)=\langle y_1\rangle - \langle y_2 \rangle - \langle y_3 \rangle $$
$$\Delta^{\mathcal{GC}}_1(y_1)=\langle z_1\rangle - \langle z_2\rangle, \ 
\Delta^{\mathcal{GC}}_1(y_2)=\langle z_1\rangle - \langle z_2\rangle, \ 
\Delta^{\mathcal{GC}}_1(y_3)=0$$
$$\Delta^{\mathcal{GC}}_0(z_1)=0=\Delta_0(z_2),$$
respectively. Hence, the matrix associated to the GS-boundary map $\Delta^{\mathcal{GC}}$  and its  graph $G_{\mathcal{GC}}$ are as in Figure~\ref{figura_matriz_graph2}.

\begin{figure}[H]
    \centering
        \includegraphics[width=0.7\textwidth]{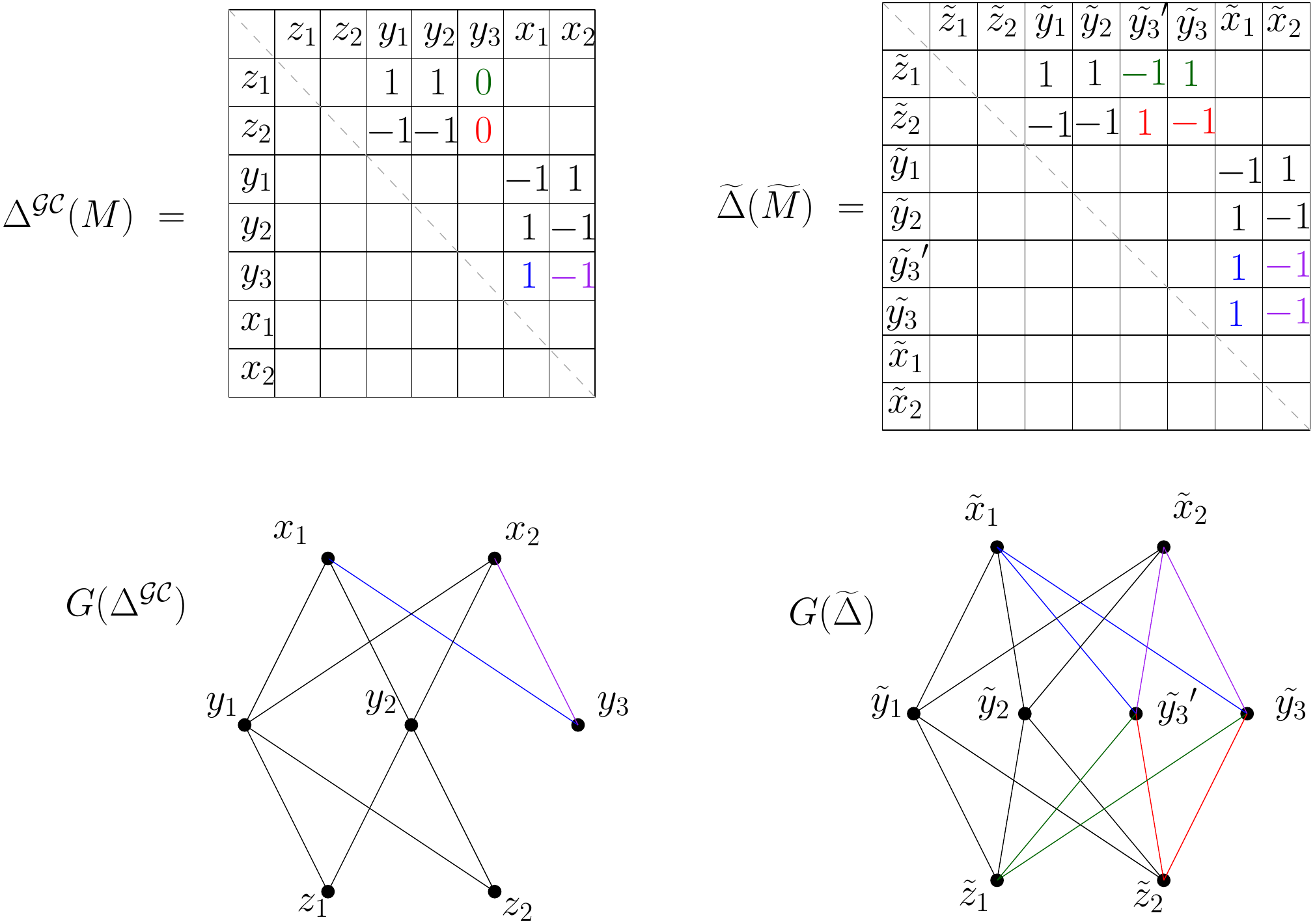}
   \caption{GS-boundary map  and its matrix graph and its associated Morsification.}\label{figura_matriz_graph2}
\end{figure}

It is also interesting to verify that Lemma~\ref{lemmagrafoGS} holds in this example. Note that the orbits connecting $\tilde{y}_3$ and $\tilde{z}_1$, $\tilde{y}_3'$ and $\tilde{z}_1$  have opposite characteristic signs. Hence the characteristic sign of the orbit connecting $y_3$ and $z_1$ is zero. The same holds for the orbit that connects $y_3$ and $z_2$.
Note that  $\widetilde{G}$, the graph associated to  $\widetilde{\Delta}$ has  cycles containing the vertices   $v_{\tilde{y_3}}$ and $v_{\tilde{y_3}'}$. However, the graph $G_{\mathcal{GC}}$ of the boundary map  $\Delta^{\mathcal{GC}}$ has no cycle containing $v_{y_3}$. This is always the case, as it was proven in Lemma~\ref{lemmagrafoGS}, which asserts that vertices associated to saddle cone singularities have only either  two positively  incident or two negatively incident edges. Hence, the cycles in  $G_{\mathcal{GC}}$ are also in  $\widetilde{G}$ and inherit the  null cycle condition.
\hfill $\bigtriangleup$
\end{example}

\subsection{Gutierrez-Sotomayor complex for Whitney, double and triple crossing singularities}

Given a singular flow $\varphi_{X}$ on  a singular 2-manifold $M \in \mathfrak{M}(\mathcal{GS})$ associated to $X\in \mathfrak{X}_{\mathcal{GS}}(M)$, where $\mathcal{S}=\mathcal{W},\mathcal{D}$ or $\mathcal{T}$, the Morsification process gives us 
a  Morsified  flow $\varphi_{\widetilde{X}}$ on  a smooth manifold $\widetilde{M} \in \mathfrak{M}(\mathcal{GC})$ associated to $\widetilde{X}\in \mathfrak{X}(\mathcal{GC})$.
In what follows, we define characteristic signs for the flow lines of $\varphi_{X}$ by means of the  transfer process  of the characteristic signs of the orbits of the Morse flow $\varphi_{\widetilde{X}}$. 

For the case of a vector field  $X \in \mathfrak{X}_{\mathcal{GW}}(M)$, note that  each Whitney singularity $x$ of $X$  has an associated   regular singularity $\widetilde{x}$ in $\widetilde{X}\in \mathfrak{X}(\mathcal{GC})$, which makes the sign transfer process  straightforward. 

\begin{definition}[Characteristic signs of flows lines of  $ \mathfrak{X}_{\mathcal{GW}}(M)$]\label{characteristicGS_w}
Let $x,y \in Sing(X)$ be singularities of consecutive natures and  $X\in \mathfrak{X}_{\mathcal{GW}}(M)$.
The  \textbf{GS-characteristic sign} $n_{u}$ of a flow line $u \in \mathcal{M}_{xy} $ is defined as follows:
\begin{enumerate}
\item If  $u$ does not belong to the singular part of $M$, define $n_u = n_{\tilde{u}}$;
\item If $u$ belongs to the singular part of $M$, define $n_u =0$.
\end{enumerate}
\end{definition}

The numbers $n_u$ are  well defined, since the orbits of $\varphi_{X}$ in the  regular part of $M$ has exactly one corresponding orbit in the Morsified flow $\varphi_{\widetilde{X}}$.

In the context of flows $\varphi_{X}$  where $X\in \mathfrak{X}_{\mathcal{GW}}(M)$, the GS-intersection number between consecutive singularities $x$ and $y$ is defined as $n(x,y):=\sum_{u\in\mathcal{M}^x_y} n_u$.

Now, consider the case of crossing singularities. 
If $x$ is a double or a triple crossing singularity then it is associated to  at least two singularities in the Morsified flow.  In fact,
if $x$ is a $n$-sheet double crossing singularity of attractor (resp., repeller) nature, then there are $n$ regular singularities  $h_0^1(x),\ldots , h_0^n (x)\in Sing(\widetilde{X})$  (resp., $h_2^1(x),\ldots , h_2^n (x)\in Sing(\widetilde{X})$) of  attractor  (resp., repeller)  nature associated to $x$ via the Morsification  process. Throughout this paper, assume that $h_0^1(x)$ (resp., $h_2^1(x)$) corresponds  to the singularity of the external manifold and $h_0^2(x), \dots ,h_0^n(x)$  (resp., $h_2^2(x), \dots ,h_2^n(x)$) correspond to the singularities of the internal manifolds.
If $x$ is a double crossing singularity of $sa$-nature (resp., $sr$-nature), then there exist exactly two singularities associated to $x$ via the Morsification process, namely, 
a regular saddle or saddle cone singularity $h_1^1(x)$  and a  regular attracting singularity (resp., repelling) $h_0^1(x)$. If $x$ is a double crossing singularity of $ss_s$- or $ss_u$-nature, then there are two saddle cone or regular saddle  singularities $h_1^1(x)$ and $h_1^2(x)$  associated to $x$ via the Morsification process. 

Note that, given  $u\in\mathcal{M}_{xy}$, if   $x$ and $y$ are both double crossing singularities of consecutive natures, then $u$ is a singular orbit.  Hence, there are exactly two flow lines  in $\varphi_{\widetilde{X}}$ which project to $u$ by the Morsification process, namely, $\tilde{u}^e $ and $ \tilde{u}^i,$ the first one is in the external manifold and the second one is in the internal manifold. Although $x$ and $y$ are considered to be consecutive singularities, one of  the flow lines  $\tilde{u}^e $ and $ \tilde{u}^i$ maybe be connecting non consecutive singularities  in the Morsified flow in $\varphi_{\widetilde{X}}$. In this sense, to simplify notation we will consider $n_{\tilde{u}}=0$ whenever $u$ is a flow line between non consecutive points. For example, this is the case of a double crossing singularity of $sr$-nature and a double crossing singularity of $a^2$-nature.

%
%

\begin{definition}[Characteristic signs of flows lines of  $ \mathfrak{X}_{\mathcal{GD}}(M)$]\label{characteristicGS_D} Let $x,y \in Sing(X)$ be singularities of consecutive natures and  $X\in \mathfrak{X}_{\mathcal{GD}}(M)$.
The  \textbf{GS-characteristic sign} $n_{u}$ of a flow line $u \in \mathcal{M}_{xy} $ is defined as follows:
\begin{enumerate}
\item If $u$ does not belong to the singular part of $M$, define $n_u = n_{\tilde{u}}$;
\item If $u$ belongs to the singular part of $M$, define $n_u$ to be the pair  $n_u =(n_u^e,n_u^i):=(n_{\tilde{u}^e},n_{\tilde{u}^i})$, where $n_{\tilde{u}^e}$ and  $n_{\tilde{u}^i}$ are the characteristic signs of the flow lines $\tilde{u}^e$ and $\tilde{u}^i$, respectively.
\end{enumerate}
\end{definition}

Note that, for vector fields having double crossing singularities, the characteristic sign of an orbit is defined in terms of the signs of the orbits of their corresponding singularities  through the Morsified process. Hence, it is natural to define the intersection number between $h_k^j(x)$ and $h_{k-1}^{\ell}(y)$ by 
$$n(h_k^j(x),h_{k-1}^{\ell}(y))=\sum n_{\tilde{u}},$$
where the sum is over the flow lines $\tilde{u}\in\widetilde{\mathcal{M}}^{h_k^j(x)}_{h_{k-1}^{\ell}(y)}$, for  $k=1,2$, $j=1,\dots ,\eta_k(x)$ and $\ell = 1,\dots ,\eta_{k-1}(y)$. 

Now, consider the case of triple crossing singularities. 
If  $x$ is an $(2n+1)$-sheet triple crossing singularity of $a^{2n+1}$-nature (resp., $r^{2n+1}$-nature), then there are $2n+1$ regular attracting  (resp., repelling) singularities $h_0^1(x),\ldots , h_0^{2n+1} (x)\in Sing(\widetilde{X})$ 
(resp., $h_2^1(x),\ldots , h_2^{2n+1} (x)\in Sing(\widetilde{X})$)
  associated  to $x$ via the Morsification process. Assume that 
 $h_0^1(x)$ (resp., $h_2^1(x)$) corresponds to the singularity of the external manifold and 
$h_0^2(x), \dots ,h_0^{2n+1}(x)$  (resp., $h_2^2(x), \dots ,h_2^{2n+1}(x)$) correspond to the singularities of the middle and internal manifolds.  More specifically, the singularities  $h_0^2(x), \dots ,h_0^{n+1}(x)$ (resp., $h_2^2(x), \dots ,h_2^{n+1}(x)$) belong to the middle manifolds and  $h_0^{n+2}(x), \dots ,h_0^{2n+1}(x)$ (resp., $h_2^{n+2}(x), \dots ,h_2^{2n+1}(x)$) belong to the internal manifolds.
If $x$ is a triple crossing singularity of  $ssa$-nature (resp.,  $ssr$-nature), then there are  regular saddle or  saddle cone singularities $h_1^1(x),$ $ h_1^2(x)$ and a regular attracting (resp., repelling) singularity $h_0^1(x)$ (resp., $h_2^1(x)$) associated to $x$ via the Morsification process.

Note that, given a orbit $u\in\mathcal{M}_{xy}\subset M$ where $x$ and $y$ are consecutive triple crossing singularities, then $u$ is a singular orbit. Hence, there are  at most three orbits in $\varphi_{\widetilde{X}}$ which projects to $u$ by the Morsfication process, namely, 
$$\tilde{u}^e\in\widetilde{\mathcal{M}}_{h_k^1(x)h_{k-1}^1(y)},  
\tilde{u}^m\in\widetilde{\mathcal{M}}_{h_k^{j_1}(x)h_{k-1}^{\ell_1}(y)}\ \text{e} \ \tilde{u}^i\in\widetilde{\mathcal{M}}_{h_k^{j_2}(x)h_{k-1}^{\ell_2}(y)},$$
where $k=1,2$, $j_1,j_2\in \{2,\dots ,\eta_k(x)\}$ and $\ell_1,\ell_2\in\{2,\dots ,\eta_{k-1}(y)\}$. One has that $\tilde{u}^e$ belongs to the external manifold, $\tilde{u}^m$ belongs to the middle manifold and $\tilde{u}^i$ belongs to the internal manifold.

In other to simplify the notation, we consider  $n_{\tilde{u}^i}=0$ when $h_k^j(x)$ and $h_{k-2}^{\ell}(y)$ are not consecutive generators in the Morsified flow.

\begin{definition}[Characteristic signs of flows lines of  $ \mathfrak{X}_{\mathcal{GT}}(M)$]\label{characteristicGS_T}
 Let $x,y \in Sing(X)$ be singularities of consecutive natures and  $X\in \mathfrak{X}_{\mathcal{GT}}(M)$.
The  \textbf{characteristic sign} $n_{u}$ of a flow line $u \in \mathcal{M}_{xy} $ is defined as follows:
\begin{itemize}
\item If $u$ does not belong to the singular part of $M$, define $n_u = n_{\tilde{u}}$;
\item If $u$ belongs to the singular part of $M$, define $n_u$ to be the triple  $n_u =(n_u^e,n_u^m,n_u^i):=(n_{\tilde{u}^e},n_{\tilde{u}^m},n_{\tilde{u}^i})$, where $n_{\tilde{u}^e}$, $n_{\tilde{u}^m}$ and  $n_{\tilde{u}^i}$ are the characteristic signs of the flow lines $\tilde{u}^e$, $\tilde{u}^m$ and $\tilde{u}^i$, respectively.
\end{itemize}
\end{definition}

Note that, for vector fields having triple crossing singularities, the characteristic sign of a flow line is defined in terms of the characteristic sign of the orbits of  their corresponding singularities  through the Morsification process. Hence, it is natural to define the intersection number between $h_k^j(x)$ and $h_{k-1}^{\ell}(y)$ as
$$n(h_k^j(x),h_{k-1}^{\ell}(y))=\sum n_{\tilde{u}},$$
where the sum is over flow lines $\tilde{u}\in\widetilde{\mathcal{M}}^{h_k^j(x)}_{h_{k-1}^{\ell}(y)}$, for  $k=1,2$, $j=1,\dots ,\eta_k(x)$ and $\ell = 1,\dots ,\eta_{k-1}(y)$. 


Now it remains to prove that the GS-boundary map $
 \Delta_*^{\mathcal{GS}}$, described  in  Definition~\ref{def_complex_GS}, is well defined for GS-flows  on  
$M\in\mathfrak{M}(\mathcal{GS})$, $\mathcal{S} = \mathcal{W},\mathcal{D},\mathcal{T}$.

\begin{proposition}\label{prop_matriz} 
Consider a flow $\varphi_X$ in  $M\in\mathfrak{M}(\mathcal{GS})$ associated to a vector field $X\in\mathfrak{X}_{\mathcal{GS}}(M)$, where $\mathcal{S} = \mathcal{W}, \mathcal{D}$ or $ \mathcal{T}$. Let   $\widetilde{\varphi}_{\widetilde{X}}$ be the flow in $\widetilde{M}\in\mathfrak{M}(\mathcal{GS})$  obtained by the Morsification process of  ${X}	\in\mathfrak{X}_{\mathcal{GS}}(M)$, then the boundary map   $\Delta^{\mathcal{GS}}$ associated to $\varphi_X$  is equal to  the boundary map $\widetilde{\Delta}^{\mathcal{GC}}$ associated to $\widetilde{\varphi}_{\widetilde{X}}$.
\end{proposition}

\begin{proof}
Given a flow line $u$ of $\varphi_X$ that does not belong to the singular part of $M$, the characteristic sign $n_u$ of $u$  is equal to the characteristic sign of its unique corresponding flow line $\tilde{u}$ in $\varphi_{\widetilde{X}}$. Therefore the intersection number between $x, y \in Sing(X)$ is equal to the intersection number  $\tilde{x},\tilde{y} \in Sing(\widetilde{X})$ whenever the connecting manifold $\mathcal{M}_{xy}$ is contained in the regular part of $M$. 

Consider orbits in a connecting manifold $\mathcal{M}_{xy}$ which belong to the singular part of $M$. This means that  both $x$ and $y$ are not regular singularities. The proof is done case by case  according to the type of vector field $X \in \mathfrak{X}_{\mathcal{GS}}(M)$ where $\mathcal{S}=\mathcal{W},\mathcal{D},\mathcal{T}$. 

\begin{itemize}
    \item Let $X \in \mathfrak{X}_{\mathcal{GW}}(M)$ and $x,y\in Sing(X)$ be Whitney singularities then $u\in\mathcal{M}_{xy}$ belongs to the singular part of $M$. Denote by $\tilde{x},\tilde{y}\in Sing(\widetilde{X})$ the Morsified singularities associated to $x$ and $y$, respectively.  The Morsified process matches $u$ to two flow lines $\tilde{u}_1,\tilde{u}_2\in\widetilde{M}^{\tilde{x}}_{\tilde{y}}$ with
$n_{\tilde{u}_1}\neq n_{\tilde{u}_2}$. Hence, $n(\tilde{x},\tilde{y})=0=n(x,y)$.

    \item Let $X \in \mathfrak{X}_{\mathcal{GD}}(M)$ and $x,y\in Sing(X)$ be double crossing singularities then $u\in\mathcal{M}_{xy}$ belongs to the singular part of $M$. 
Consider the flow lines of $\widetilde{X}$,
$$\tilde{u}^e\in\widetilde{\mathcal{M}}_{h_k^1(x)h_{k-1}^1(y)} \ \text{and} \ \tilde{u}^i\in\widetilde{\mathcal{M}}_{h_k^j(x)h_{k-1}^{\ell}(y)},$$
with $k=1,2$, $j\in \{2,\dots ,\eta_k(x)\}$ and $\ell\in\{2,\dots ,\eta_{k-1}(y)\}$. 
By the definition of transfer of signs, $n_u^e=n_{\tilde{u}^e}$ 
and $n_u^i=n_{\tilde{u}^i}$. Thus,
$$\qquad n(x,y)=\Bigg(\sum_{u\in\mathcal{M}_{xy}}n_{u}^e\ ,\sum_{u\in\mathcal{M}_{xy}}n_{u}^i\Bigg)
=\Bigg(\sum_{\tilde{u}^e}n_{\tilde{u}^e}\ ,\sum_{\tilde{u}^i}n_{\tilde{u}^i}\Bigg)
=\Big(n(h_k^1(x)h_{k-1}^1(y))\ ,\ n(h_k^j(x)h_{k-1}^{\ell}(y)\Big).$$

    \item Let $X \in \mathfrak{X}_{\mathcal{GT}}(M)$ and $x,y\in Sing(X)$ be triple crossing singularities, then $u\in\mathcal{M}_{xy}$ belongs to the singular part of $M$. Consider the flow lines of $\widetilde{X}$,
$$\tilde{u}^e\in\widetilde{\mathcal{M}}_{h_k^1(x)h_{k-1}^1(y)}, \tilde{u}^m\in\widetilde{\mathcal{M}}_{h_k^{j_1}(x)h_{k-1}^{\ell_1}(y)} \ \text{and} \ \tilde{u}^i\in\widetilde{\mathcal{M}}_{h_k^{j_2}(x)h_{k-1}^{\ell_2}(y)},$$
with $k=1,2$, $j_1,j_2\in \{2,\dots ,\eta_k(x)\}$ and $\ell_1,\ell_2\in\{2,\dots ,\eta_{k-1}(y)\}$. 
By the definition of sign transfer, one has that $n_u^e=n_{\tilde{u}^e},n_u^m=n_{\tilde{u}^m}$ 
and $n_u^i=n_{\tilde{u}^i}$ . Therefore,
$$n(x,y)=\Bigg(\sum_{u\in\mathcal{M}_{xy}}n_{u}^e,\sum_{u\in\mathcal{M}_{xy}}n_{u}^m,\sum_{u\in\mathcal{M}_{xy}}n_{u}^i\Bigg)
=\Bigg(\sum_{\tilde{u}^e}n_{\tilde{u}^e},\sum_{\tilde{u}^m}n_{\tilde{u}^m},\sum_{\tilde{u}^i}n_{\tilde{u}^i}\Bigg)$$
$$=\Big(n(h_k^1(x)h_{k-1}^1(y)),n(h_k^{j_1}(x)h_{k-1}^{\ell_1}(y),n(h_k^{j_2}(x)h_{k-1}^{\ell_2}(y))\Big).$$
\end{itemize}
 In any case, the proposition follows.
\end{proof}

\begin{corollary}
Let $\Delta_{*}^{\mathcal{GS}}$ be the GS-boundary map associated to $\varphi_{X}$ with $X\in \mathfrak{X}_{\mathcal{GS}}(M)$, where  $\mathcal{S} = \mathcal{W}, \mathcal{D}$ or $\mathcal{T}$. Then  $\Delta_{k-1}^{\mathcal{GS}}\circ \Delta_{k}^{\mathcal{GS}} = 0$, for all $k\in\mathbb{Z}$.
\end{corollary}

\begin{proof}
It follows directly from Theorem~\ref{prop_operador} and Proposition~\ref{prop_matriz}. 
\end{proof}

Hence, the pair  $(C^{\mathcal{GS}}_*(M,X),\Delta^{\mathcal{GS}}_*)$ is indeed a chain complex,  whenever  $M\in\mathfrak{M}(\mathcal{GS})$,  $X\in\mathfrak{X}_{\mathcal{GS}}(M)$ and  $\mathcal{S} = \mathcal{W}, \mathcal{D}$ or $ \mathcal{T}$.

\subsubsection{Examples}

\begin{example}\label{ex_w3}

Consider a GS-flow $\varphi_{X}$  defined on a singular manifold $M \in \mathfrak{M}(\mathcal{GW})$  and its Morsification $(\widetilde{M},\varphi_{\widetilde{X}})$, as in Figure~\ref{figura_ex_whitney_man}.
   Considering a choice of orientations on the unstable manifolds  of the critical points of $\widetilde{M}$,  the GS-characteristic signs on the orbits of  $\varphi_{X}$   are obtained, by the Definition~\ref{characteristicGS_w}.
   
 \begin{figure}[H]
    \centering
            \includegraphics[width=0.80\textwidth]{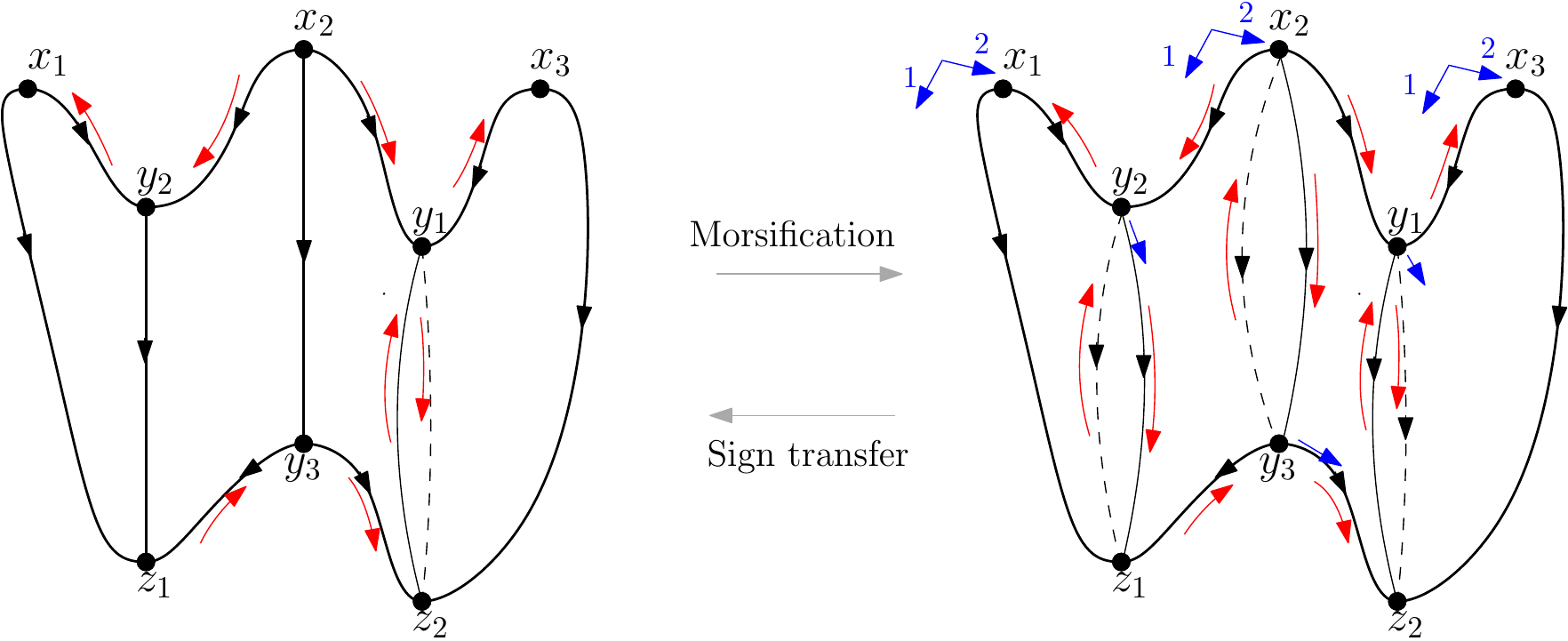}             
   \caption{A GS-flow with Whitney singularities and its Morsification.}\label{figura_ex_whitney_man}
\end{figure}

The GS-chain groups are:
$ C_2^{\mathcal{GW}}(M)=\mathbb{Z} \langle x_1\rangle \oplus \mathbb{Z} \langle x_2\rangle \oplus \mathbb{Z} \langle x_3\rangle,$
$C_1^{\mathcal{GW}}(M)=\mathbb{Z} \langle y_1\rangle \oplus \mathbb{Z} \langle y_2\rangle \oplus\mathbb{Z} \langle y_3\rangle, $
$C_0^{\mathcal{GW}}(M)=\mathbb{Z} \langle z_1\rangle \oplus \mathbb{Z} \langle z_2\rangle $
and $C_k^{\mathcal{GW}}(M)=0, k\neq 0,1,2$.  The GS- intersection numbers are:
$n(x_1,y_2)=-1, \ n(x_2,y_1)=1, \ n(x_2,y_2)=1, \ 
n(x_2,y_1)=1, \ n(x_3,y_1)=0, \ 
n(x_2,y_3)=0,  \ n(y_1,z_2)=0, \ n(y_2,z_1)=0, \  n(y_3,z_2)=1, \ n(y_3,z_1)-1.   $
The GS-boundary maps $\Delta^{\mathcal{GW}}_2:C_2\rightarrow C_1$, $\Delta^{\mathcal{GW}}_1:C_1\rightarrow C_0$ and $\Delta^{\mathcal{GW}}_0:C_0\rightarrow 0$  are given by:
$$ \Delta^{\mathcal{GW}}_2(x_1)=-\langle y_2\rangle , \ 
\Delta^{\mathcal{GW}}_2(x_2)=\langle y_1\rangle + \langle y_2 \rangle , \ 
\Delta^{\mathcal{GW}}_2(x_3)=- \langle y_1\rangle, $$
$$\Delta^{\mathcal{GW}}_1(y_1)=0, \ 
\Delta^{\mathcal{GW}}_1(y_2)=0, \ 
\Delta^{\mathcal{GW}}_1(y_3)=\langle z_2\rangle - \langle z_1\rangle.$$
Therefore, the matrix of the GS-boundary map  $\Delta^{\mathcal{GW}}$ is as in Figure~\ref{figura_ex_whitney_matriz-new}.

 \begin{figure}[H]
    \centering
                \includegraphics[width=0.40\textwidth]{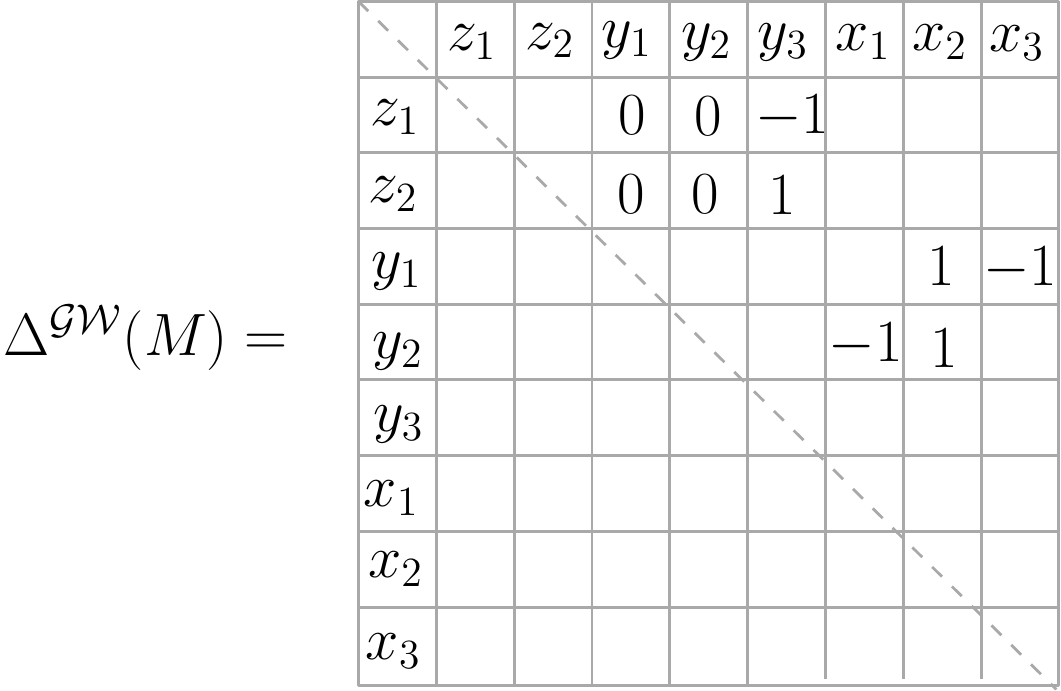}
   \caption{GS-boundary operator $\Delta^{\mathcal{GW}}(M)$ .}\label{figura_ex_whitney_matriz-new}
\end{figure}\hfill $\bigtriangleup$

\end{example}

\begin{example}\label{ex_d}
 Consider a GS-flow $\varphi_{X}$  defined on a singular manifold $M$  and its  Morsification $(\widetilde{M},\varphi_{\widetilde{X}})$  as in Figure~\ref{fig_caracteristico_d1}.  Considering a choice of orientations on the unstable manifolds  of the critical points of $\widetilde{M}$,  the GS-characteristic signs on the orbits of $(\widetilde{M},\varphi_{\widetilde{X}})$   are obtained. By  Definition~\ref{characteristicGS_D}, one gets the  GS-characteristic signs of the orbits in $\varphi_{X}$. 

\begin{figure}[H]
    \centering
        \includegraphics[width=0.75\textwidth]{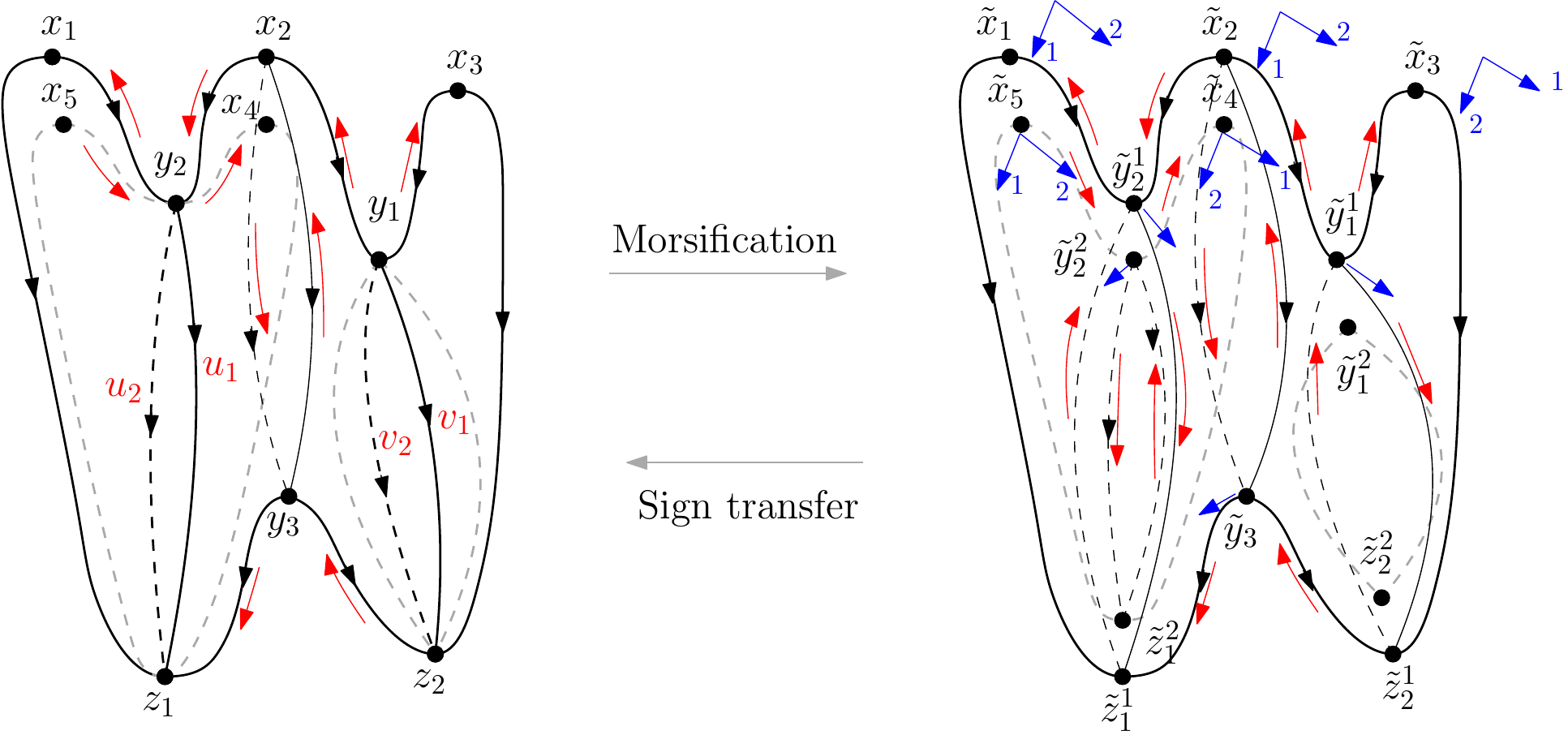}
   \caption{A GS-flow with double crossing singularities and its Morsification.}\label{fig_caracteristico_d1}
\end{figure}

Since the  orbits $u_1, u_2, v_1$ and $v_2$  are singular orbits, hence the GS-characteristic sign is given by a pair:
$n_{u_1}=(n_{u_1^1},$ $n_{u_1^2})=(1,-1), $   $ n_{u_2}=(n_{u_2^1},$ $ n_{u_2^2})=(-1,1),$
$n_{v_1}=(n_{v_1^1},$ $n_{v_1^2})=(1,0),  n_{v_2}=(n_{v_2^1},$ $n_{v_2^2})=(-1,0).$
The groups of the GS-chain complex are given by
$$ C_2^{\mathcal{GD}}(M)=\bigoplus_{i=1}^5 \mathbb{Z} \langle x_i\rangle \oplus \mathbb{Z} \langle y^2_1 \rangle$$
$$C_1^{\mathcal{GD}}(M)=\mathbb{Z} \langle y^1_1\rangle \oplus \mathbb{Z} \langle y_2^1\rangle \oplus\mathbb{Z} \langle y_2^2\rangle \oplus\mathbb{Z} \langle y_3\rangle, \qquad C_0^{\mathcal{GD}}(M)=\mathbb{Z} \langle z_1^1\rangle \oplus \mathbb{Z} \langle z_1^2\rangle \oplus \mathbb{Z} \langle z_2^1\rangle \oplus \mathbb{Z} \langle z_2^2\rangle$$
and $C_k^{\mathcal{GD}}(M)=0, k\neq 0,1,2$. 
Note that $y_1$ is a double crossing singularity of saddle-repelling  nature ($sr$), hence it is associated to a generator in  $C_2^{\mathcal{GD}}(M)$ and a generator in $C_1^{\mathcal{GD}}(M)$. Analogously, the double crossing singularity $y_3$ has saddle-saddle nature ($ss_s$), hence it is associated to two generators in  $C_1^{\mathcal{GD}}(M)$. Finally, each attracting double crossing singularity  $z_1$ and $z_2$ is associated to two generators in $C_0^{\mathcal{GD}}(M)$. The GS-intersection numbers are:
$n(x_1,y_2^1)=-1, \ $ $
n(x_2,y_2^1)=1,$ $ \ n(x_2,y_1^1)=-1, $
$n(x_2,y_3)=0, \ $ $ n(x_3,y_1^1)=1, \ $ $n(x_4,y_2^2)=-1, \ $ $ n(x_5,y_2^2)=1, $ 
$n(y_3,z_1^1)=1, \ $ $ n(y_3,z_2^1)=-1, $
$n(y_1,z_2)=n_{v_1}+n_{v_2}=(1,0)+(-1,0)=(0,0), $
$n(y_2,z_1)=n_{u_1}+n_{u_2}=(1,-1)+(-1,1)=(0,0).   $
Therefore, the GS-boundary operator $\Delta_2^{\mathcal{GD}}:C_2\rightarrow C_1$, $\Delta_1^{\mathcal{GD}}:C_1\rightarrow C_0$ and $\Delta_0^{\mathcal{GD}}:C_0\rightarrow 0$ are given by:
$$ \Delta_2^{\mathcal{GD}}(x_1)=-\langle y_2^1\rangle, \ 
\Delta^{\mathcal{GD}}_2(x_2)=-\langle y_1^1 \rangle + \langle y_2^1\rangle, \ 
\Delta^{\mathcal{GD}}_2(x_3)=-\langle y_1^1\rangle, \
\Delta^{\mathcal{GD}}_2(x_4)=-\langle y_2^2\rangle, \
\Delta^{\mathcal{GD}}_2(x_5)=-\langle y_2^2\rangle, $$
$$\Delta^{\mathcal{GD}}_1(y_3)=\langle z_1^1\rangle - \langle z_2^1\rangle, \ 
\Delta^{\mathcal{GD}}_1(y_1^1)=0, \ 
\Delta^{\mathcal{GD}}_1(y_2^1)=0, \ 
\Delta^{\mathcal{GD}}_1(y_2^2)=0, \ \Delta^{\mathcal{GD}}_0(z_1)=\Delta^{\mathcal{GD}}_0(z_2)=0.$$
The matrix  $\Delta^{\mathcal{GD}}$ of the GS-boundary operator is shown in Figure~\ref{fig_flow_d}.

\begin{figure}[H]
    \centering
        \includegraphics[width=0.5\textwidth]{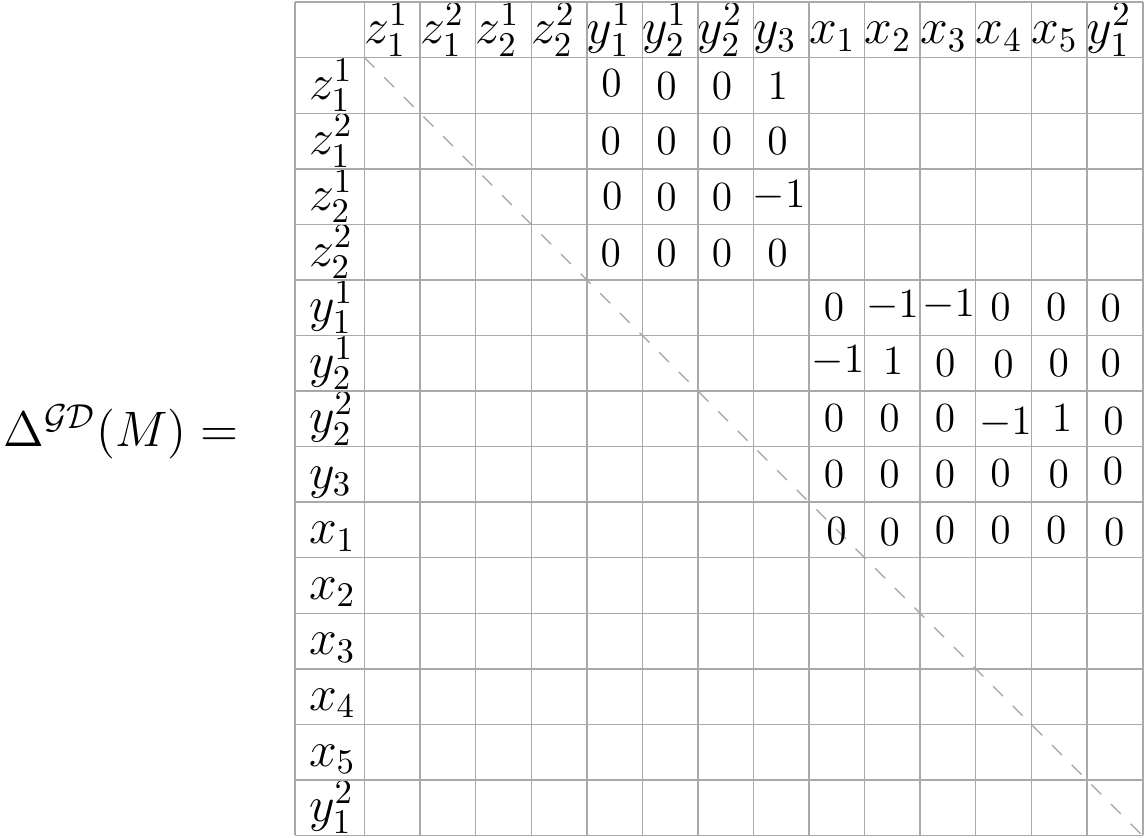}
   \caption{GS-boundary operator $\Delta^{\mathcal{GD}}(M)$.}\label{fig_flow_d}
\end{figure}\hfill $\bigtriangleup$

\end{example}

\section{Dynamical Homotopical  Cancellation Theorem for GS-flows }\label{DCTGSF}

In this section we prove a   homotopical cancellation theorem for consecutive singularities  of a Gutierrez-Sotomayor vector field.

In the  homotopical cancellation process of singularities in  $\mathcal\mathcal{S} = \mathcal{C}$, $\mathcal{W}$, $\mathcal{D}$ or $\mathcal{T}$,
 three singularities in $\mathcal\mathcal{S}$, one of saddle nature $y$ and two of attracting (resp., repelling) nature $z_1,z_2$ (resp., $x_1,x_2$),  give rise to a new singularity of attracting  nature $\bar{z}$  (resp., repelling nature  $\bar{x}$). Droplets or folds associated to these singularities are topological invariants  and are registered in the  singularity  type number, see definition in Section \ref{sec:GSVF}. 
After the   homotopical cancellation of $y$ and $z_i$ (resp., $x_i$), $i=1$ or $2$, $\bar{z}$ is the new singularity (resp., $\bar{x}$) and the type number of $\bar{z}$ is related to the types numbers $m(y)$, $m(z_1)$ and $m(z_2)$  (resp. $m(y)$, $m(x_1)$ and $m(x_2)$) of $y$, $z_1$ and $z_2$ (resp., $y$, $x_1$ and $x_2$). 
We say that  $\bar{z}$ (resp., $\bar{x}$) \textbf{\emph{inherits the type numbers}} $m(y)$, $m(z_1)$ and $m(z_2)$ as follows:
$$
m(\bar{z}) =
m(y) + m(z_1) + m(z_2)  \ \ \  (\text{resp.,} \ m(\bar{x}) =
m(y) + m(x_1) + m(x_2) )
$$
for $\mathcal\mathcal{S} = \mathcal{C}$, $\mathcal{W}$, $\mathcal{D}$ or $\mathcal{T}$.  Recall, that the type number of a regular singularity is always zero. Hence, as a consequence, whenever $z_1$ and $z_2$ are regular $\bar{z}$ will inherit a type number equal to zero.

For example, in Figure~\ref{fig:exemplo_cone_intro2}, the three singularities: $z_1$ a regular type singularity of attracting nature, 
$y_2$ a cone type singularity of saddle nature and $z_2$ a $2$-sheet cone type singularity of attracting nature, are involved in the  homotopical cancellation process. More specifically, they give rise to $\bar{z}_{1}$ which 
 inherits the types of $z_1$, $z_2$ and $y_2$, i.e., $z_1$ contributes zero, $y_2$ contributes one and $z_2$ contributes 2, hence $\bar{z}_1$ is a 3-sheet cone type singularity of attracting nature.
For more examples, see Subsection \ref{subsec:6.3}.

\begin{definition}
 Let  $X\in \mathfrak{X}_{\mathcal{GS}}(M)$ be a Gutierrez-Sotomayor vector field on $M\in\mathfrak{M}(\mathcal{GS})$, for  $\mathcal{S}=\mathcal{C}$ or $\mathcal{W}$. Let $p, q \in Sing(X)$ be consecutive singularities of $k$ and $k\! -\!1$ nature numbers, respectively.
One says that $p$ and $q$  are \textbf{dynamically homotopically cancelled}  and that together with $q'$ give rise to  $\bar{q}'$ if 
 there is a singular $2$-manifold $M'$ of the same homotopy type as $M$ and 
 there exists a vector field  $X'\in  \mathfrak{X}_{\mathcal{GS}}(M')$
 which is topologically equivalent to $X$ outside of a neighborhood $V$ of the flow lines $u_1$ and $u_2$ which is given  as follows:

 \begin{enumerate}
     \item If $k=1$, then ${\mathcal{M}}^{p}_{q}=\{u_1\}$ and ${\mathcal{M}}^{p}_{q'}=\{u_2\}$, where $q'$ is the other singularity of attracting nature connecting with $p$.    
     Moreover, the vector field $X'\in  \mathfrak{X}_{\mathcal{GS}}(M')$ on $V$ contains only one $n$-sheet $\mathcal{S}$-singularity $\bar{q}'\in \mathfrak{X}_{\mathcal{GS}}(M')$ of attracting nature which inherits the type numbers of $p$, $q$ and  $q'$.  In this case,   $p$ and $q$   together with $q'$ give rise to  $\bar{q}'$.

     \item  If $k=2$, then  ${\mathcal{M}}^{p}_{q}=\{u_1\}$ and ${\mathcal{M}}^{p'}_{q}=\{u_2\}$, where $p'$ is the other singularity of repelling nature connecting with $q$.
       Moreover, the vector field $X'\in  \mathfrak{X}_{\mathcal{GS}}(M')$ on $V$ contains only one singularity $\overline{p}'$ of repelling nature.  In this case, $p$ and $q$   together with ${p}'$ give rise to an $n$-sheet $\mathcal{S}$-singularity $\bar{p}' \in \mathfrak{X}_{\mathcal{GS}}(M')$ which inherits the type numbers of
$p$, $p'$ and $q$.
     
 \end{enumerate}
\label{def:dynhotcan}
\end{definition}

Note that if  $X\in\mathfrak{X}_{\mathcal{GS}}(M)$ is  a GS-vector field on $M\in\mathfrak{M}(\mathcal{GS})$ with only regular singularities, then the GS  homotopical cancellation of consecutive singularities coincides with the notion of cancellation of critical points established by Smale in the Morse setting. This follows, since  Morse critical points have a unique generator and hence $q'= \overline{q}'$ and $p' =\overline{p}'$ in Definition~\ref{def:dynhotcan}.

Consider $M\in\mathfrak{M}(\mathcal{GS})$ a closed 2-manifold and a GS-flow $\varphi_{X}$ on $M$ associated to a GS-vector field
$X\in\mathfrak{X}_{\mathcal{GS}}(M)$, where $\mathcal{S}=\mathcal{C}, \mathcal{W}, \mathcal{D} $ or $\mathcal{T}$. 
Given consecutive singularities $x$ and $y$, suppose that:
\begin{itemize}
\item the GS-intersection number $n(x,y)$ is $\pm 1$, when $\mathcal{S} = \mathcal{C}$ or $\mathcal{W}$;
\item one of the coordinates of the GS-intersection number  $n(x,y)$ is equal to $\pm 1$, when $\mathcal{S} = \mathcal{D}$ or $ \mathcal{T}$.
\end{itemize}

 The theorems in this section  guarantee that under these conditions,  it is always possible to dynamically homotopically cancel the singularities with GS-intersection number equal to $\pm 1$.

\begin{theorem}[Dynamical   Homotopical Cancellation Theorem  for GS-flows - Cases $\mathcal{C}$ and $\mathcal{W}$]\label{teo_cancelamento}
 Let  $X\in \mathfrak{X}_{\mathcal{GS}}(M)$ be a Gutierrez-Sotomayor vector field on $M\in\mathfrak{M}(\mathcal{GS})$, for  $\mathcal{S}=\mathcal{C}$ or $\mathcal{W}$. Let $p, q \in Sing(X)$ be consecutive singularities of $k$ and $k\! -\!1$ nature numbers, respectively.
If $n(p,q)$ is non zero, then
 $p$ and $q$  dynamically homotopically cancelled.
\end{theorem}

\begin{proof}

Consider $k=1$, i.e.  $p$ is a saddle and $q$ is an attracting singularity. The assumption that  $n(p,q)$  is non zero guarantees that there exists a singularity $q'$  of attracting  nature, distinct from $p$ and connecting with $q$. 
 Let $C(p,q,q')$ be the set composed by the singularities $p$, $q$ and $q'$ and the flow lines ${\mathcal{M}}^{p}_{q}=\{u_1\}$ and ${\mathcal{M}}^{p}_{q'}=\{u_2\}$ connecting them. $C(p,q,q')$  is an isolated invariant set which is an attractor. Hence,  choose an isolating block $V$ of $C(p,q,q')$ such that the vector field $X$ is transversal to the boundary of $V$, $\partial \overline{V}$.

 Given a point $x \in M$, denote by $\gamma(x)$ the orbit through the point $x$.
Clearly $\gamma(p)=p$ and $\gamma(q)=q$ and $\gamma(q')=q'$. Consider a path $\delta$  which has image coinciding with the  juxtaposition of the  paths
 $$\gamma(q') * \gamma(u_2)* \gamma(p)* \gamma^{-1}(u_1)* \gamma(q),$$
i.e., $\delta$ is the path going from $q$ to $q'$ through the orbits of $u_1$ and $u_2$. Since $\delta$ is not a closed path, one can contract it to a point $\overline{q}'$,  obtaining a  new topological space $\overline{V}$ preserving the boundary $\partial \overline{V} $ of the same homotopy type as $V$, where the type number of $\overline{q}'$ is equal to $m(\overline{q}') =
m(q) + m(q') + m(p) $. 
Now consider an attracting singular vector field on $\overline{V}$ such that $\overline{q}'$  is a singularity of attracting nature and the vector field is transverse to the boundary of $\overline{V}$.  

The result follows by gluing  $M \setminus V$ and $\overline{V}$ together with their  respective vector fields.

 If $k=2$, the proof is analogous by considering the reverse flow. 
\end{proof}

\begin{definition}

 Let  $X\in \mathfrak{X}_{\mathcal{GS}}(M)$ be a Gutierrez-Sotomayor vector field on $M\in\mathfrak{M}(\mathcal{GS})$, for  $\mathcal{S}=\mathcal{D}$ or $\mathcal{T}$. Let $p, q \in Sing(X)$ be consecutive singularities,  $h_k^j(p)$ and $h_{k-1}^{\ell}(q)$ be the respective  consecutive generators of their natures, where $j\in\{1,\dots ,\eta_k(p)\}$ and $\ell\in\{1,\dots ,\eta_{k-1}(q)\}$.
One says that $p$ and $q$ are \textbf{dynamically homotopically  cancelled} if 
 there exists a vector field  $X'\in  \mathfrak{X}_{\mathcal{GS}}(M')$  which is topologically equivalent to $X$ outside of a neighborhood $V$
 of the flow lines $u_1$ and $u_2$ which are  given as follows:
 \begin{enumerate}
     \item If $k=1$, then $\widetilde{\mathcal{M}}^{h^{j}_{k}(p)}_{h_ {k-1}^{\ell}(q)}=\{u_1\}$ and $\widetilde{\mathcal{M}}^{h^{j}_{k}(p)}_{h_ {k-1}^{i}(q')}=\{u_2\}$, where $q'$ is the other singularity of attracting nature connecting with $p$  such that $h_ {k-1}^{i}(q')$ and $h_ {k-1}^{\ell}(q)$ belong to the same connected component of the Morsification $\widetilde{M}$.    
     Moreover, the vector field $X'\in \mathfrak{X}_{\mathcal{GS}}(M')$ on $V$ contains only one singularity $\overline{q}'$  whose  generators  match the union of all  the generators of  $p, q$ and $q'$  excluding the cancelled pair of  generators $h_k^j(p)$ and $ h_{k-1}^{\ell}(q)$.   In this case, $p$ and $q$  together with  ${q}'$
 give rise to an $n$-sheet $\mathcal{S}$-singularity      $\overline{q}' \in \mathfrak{X}_{\mathcal{GS}}(M')$
   which    inherits the type numbers of $p$, $q$ and $q'$.
     
     \item  If $k=2$, then $\widetilde{\mathcal{M}}^{h^{j}_{k}(p)}_{h_ {k-1}^{\ell}(q)}=\{u_1\}$ and $\widetilde{\mathcal{M}}^{h^{i}_{k}(p')}_{h^ {\ell}_{k-1}(q)}=\{u_2\}$, where $p'$ is the other singularity of repelling nature connecting with $q$  such that $h_{k}^{i}(p')$ and $h_ {k}^{j}(p)$ belong to the same connected component of  the Morsification $\widetilde{M}$.
       Moreover, the vector field $X'\in \mathfrak{X}_{\mathcal{GS}}(M')$ on $V$ contains only one singularity $\overline{p}'$ whose  generators    match the union of all the generators of $p, q$ and $p'$ excluding the cancelled pair of  generators $h_k^j(p)$ and $ h_{k-1}^{\ell}(q)$.  In this case, $p$ and $q$   together with  ${p}'$ give rise to an $n$-sheet $\mathcal{S}$-singularity  $\overline{p}'\in \mathfrak{X}_{\mathcal{GS}}(M')$  which    inherits the type numbers of $p$, $q$ and $p'$.
 \end{enumerate}
\end{definition}

For example, in Figure~\ref{fig:ex_duplo-cancel}, the three singularities: $z_1$ and $z_2$ are $2$-sheet double crossing type singularities of attracting nature and  
$y$ is a regular  type singularity of saddle nature; are involved in the  dynamical homotopical cancellation process. More specifically, they give rise to $\bar{z}_{2}$ which 
 inherits the types of $z_1$, $z_2$ and $y$, i.e., $z_1$ contributes 1, $y$ contributes zero and $z_2$ contributes 1, hence $\bar{z}_2$ is a 3-sheet  double crossing type singularity of attracting nature.
For more examples, see Subsection \ref{subsec:6.3}.

\begin{figure}[h!]
    \centering
        \includegraphics[width=0.8\textwidth]{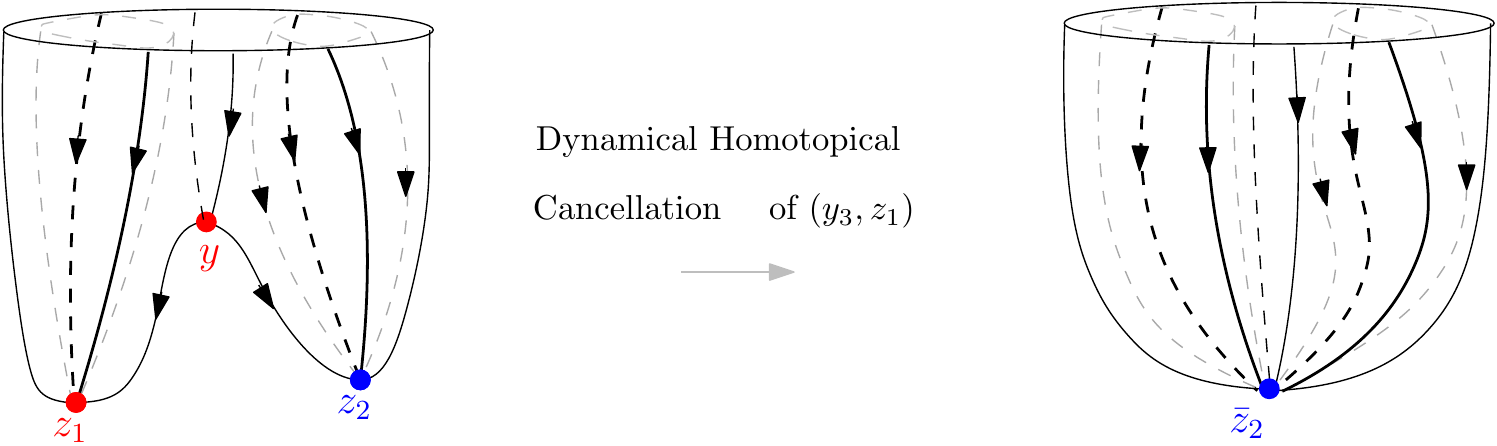}
    \caption{Dynamical Homotopical Cancellation of the singularities $y$ and $z_1$.}\label{fig:ex_duplo-cancel}
\end{figure}

\begin{theorem}[Dynamical   Homotopical Cancellation Theorem  for GS-flows - Cases $\mathcal{D}$ and $\mathcal{T}$]\label{teo_cancelamento2}
 Let  $X\in \mathfrak{X}_{\mathcal{GS}}(M)$ be a Gutierrez-Sotomayor vector field on $M\in\mathfrak{M}(\mathcal{GS})$, for  $\mathcal{S}=\mathcal{D}$ or $\mathcal{T}$. Let $p, q \in Sing(X)$ be consecutive singularities,  $h_k^j(p)$ and $h_{k-1}^{\ell}(q)$ be the respective  consecutive generators of their natures, where $j\in\{1,\dots ,\eta_k(p)\}$ and $\ell\in\{1,\dots ,\eta_{k-1}(q)\}$.
If $n(h_k^j(p), h_{k-1}^{\ell}(q))$ is non zero for some $j$ and $\ell$, then 
 $p$ and $q$ are dynamically homotopically  cancelled.
\end{theorem}

\begin{proof}
 Firstly, one considers the Morsified manifold  $\widetilde{M} \in \mathfrak{M}(\mathcal{GC})$  and the Morsified  flow $\varphi_{\widetilde{X}}$ on  $\widetilde{M}$ associated to $\widetilde{X}\in \mathfrak{X}(\mathcal{GC})$. One also needs to be aware that, in the Morsification process, the generators of  double or  triple crossing points will give rise to singularities in  $\mathfrak{X}_{\mathcal{GS}}(\widetilde{M})$  which may be in  distinct connected components of $\widetilde{M}$.
 
Consider $k=1$, i. e. suppose that  $q$  is a singularity having at least one generator of attracting nature and   $p$ is a singularity having at least one generator of saddle nature.   Since $n(h_k^j(p), h_{k-1}^{\ell}(q))$ is non zero for some $j$ and $\ell$, there is another singularity $q'$ of attracting nature  connecting with $p$. In other words, $n(h_k^j(p), h_{k-1}^{\ell'}(q'))$  is non zero, for some $\ell'$. Moreover, the generator $h_k^j(p)$ corresponds to a saddle singularity  and $ h_{k-1}^{\ell}(q)$ and $h_{k-1}^{\ell'}(q')$ correspond to   attracting singularities in the Morsified flow. Since they  belong to the same connected component of $\widetilde{X}$, one can apply the Smale's cancellation theorem   and cancel  $h_k^j(p)$  and $ h_{k-1}^{\ell}(q)$,  as in the Morse case. After the cancellation, one obtains a  Morsified flow $\varphi_{\widetilde{X'}}$ on  $\widetilde{M}$ which coincides with $\varphi_{\widetilde{X}}$ outside a neighborhood $V$ of the  set $C(h_k^j(p),h_{k-1}^{\ell}(q),h_{k-1}^{\ell'}(q'))$  composed by the singularities $h_k^j(p)$, $h_{k-1}^{\ell}(q)$ and $h_{k-1}^{\ell'}(q')$ and the flow lines  connecting them. 

An orbit $\gamma$ of the initial singular flow $\varphi_{X}$, which belongs to the singular part of $M$ and has $\omega$-limit set equal to  $p$, $q$ or $q'$, is a fold. Moreover, $\gamma$ admits a duplication $\gamma^1$ and $\gamma^2$ in $\varphi_{\widetilde{X}}$. Only one of these orbits has $\omega$-limit set equal to   $h_k^j(p)$, $ h_{k-1}^{\ell}(q)$ or $h_{k-1}^{\ell'}(q')$, and after the cancellation, this orbit will have $\omega$-limit set equal to the singularity  $h_{k-1}^{\ell'}(q')$.   Now considering all the connected components of $\widetilde{M}$  with the respective new Morse vector field $\widetilde{X}'$, one can identify  all the orbits $\gamma$ corresponding to singular orbits in the initial flow.  As a result, one obtains a new singular manifold $M'$ with a flow $\varphi_{X'}$ associated to 
 a vector field  $X'\in\mathfrak{X}_{\mathcal{GS}}(M')$
 which coincides with $X$, up to topological equivalence, outside of a neighborhood $V$ of the flow lines $u_1$ and $u_2$, as in the statement of the theorem.
 
 If $k=2$, the proof is analogous by considering the reverse flow. 
\end{proof}
%

\section{Detecting Dynamical  Homotopical Cancellations through Spectral Sequences }\label{DDCTSS}

 The use of algebraic tools  to extract dynamical information has 
 been explored in highly influential work, see~\cite{Franks2, M, M2}. Particularly, the use of spectral sequences has been used in  dynamical systems  (see~\cite{BLMdRS1, BLMdRS2, LMdRS}) as well as in computational topology (see~\cite{Edels, Z}). 

In this section, our main motivation is to establish how the algebraic cancellations in a spectral sequence of a filtered GS-chain complex affects dynamical homotopical cancellations within a GS-flow. 
In order to keep track of the changes of the differentials $d^r$ and the generators of the modules $E^r$  on each page  $(E^r,d^r)$ of the spectral sequence,  the Spectral Sequence Sweeping Algorithm (SSSA) was developed in~\cite{CdRS} and is presented in Section~\ref{subsec:appl1}.  It provides all of the algebraic cancellations that occur on each page, as well as, the new generators of the modules of the following page. With the complete information in hand of the algebraic cancellations in Section~\ref{subsec:appl2}, we make use of the Row Cancellation Algorithm (RCA),  which is the dynamical counterpart of the SSSA.
 We refer the reader to~\cite{CdRS,MdRS} for more details on RCA and SSSA.

In Section \ref{subsec:appl1}, we give a brief overview of basic definitions for spectral sequences and state the SSSA. In Section \ref{subsec:appl2}, we present the main homotopical cancellation results, Theorem \ref{cancel1} and Theorem \ref{cancel2}. In Section \ref{subsec:6.3}  examples of these theorems for flows on  $M\in\mathfrak{M}(\mathcal{GS})$, where  $\mathcal{S}=\mathcal{C},\mathcal{W}$ or $\mathcal{D}$ are proved.

\subsection{Spectral Sequence Sweeping Algorithm} \label{subsec:appl1}


Let $R$ be a principal ideal domain. A $k$-\textbf{\emph{spectral sequence}} $E$ over $R$ is a sequence $\{E^{r},\partial^{ r}\}_{r \geq k }$, such that 
\begin{enumerate}
\item[(1)] $E^{r}$ is a bigraded module over $R$, i.e., an indexed  collection of $R$-modules $E^{r}_{p,q}$, for all $p,q\in\mathbb{Z}$;
\item[(2)] $d^{r}$ is a differential of bidegree $(-r,r-1)$ on $E^{r}$, i.e., an indexed collection of homomorphisms  $d^{r}:E^{r}_{p,q}\rightarrow E^{r}_{p-r,q+r-1}$, for all $p,q\in\mathbb{Z}$, and $(d^{r})^{2}=0$;
\item[(3)] for all $r\geq k$, there exists an isomorphism 
$H(E^r)\approx E^{r+1}$, where
$$H_{p,q}(E^r)=\frac{\mbox{Ker}
d^r:E^{r}_{p,q}\to E^r_{p-r,q+r-1}}{\mbox{Im} d^r:E^{r}_{p+r,
q-r+1}\to E^r_{p,q}}. $$
\end{enumerate}

Let  $Z^k_{p,q}=Ker (d^k_{p,q}: E^k_{p,q} \rightarrow
E^k_{p,q-1})$ and $B^k_{p,q}=Im (d^k_{p,q+1}:E^k_{p,q+1} \rightarrow
E^k_{p,q})$, then $E^{k+1}=Z^k/B^k$ and
$$B^k\subseteq B^{k+1}\subseteq \ldots \subseteq B^r \subseteq\ldots \subseteq Z^r\subseteq \ldots\subseteq Z^{k+1}\subseteq Z^k. $$

Consider the bigraded modules
$Z^{\infty}=\cap_rZ^r, B^{\infty}=\cup_rB^r$ and  $E^{\infty}=Z^{\infty}/B^{\infty}$. The latter module is called the \textbf{\emph{limit of the spectral sequence}}.
A spectral sequence $E=\{E^{r},\partial^{r}\}$ is \textbf{\emph{convergent}} if given $p,q$  there is  $r(p,q)\geq k$ such that for all  $r\geq r(p,q)$, $d^r_{p,q}: E^{r}_{p,q} \rightarrow E^{r}_{p-r,q+r-1}$ is trivial.
A spectral sequence $E=\{E^{r},\partial^{r}\}$  is \textbf{\emph{convergent in the strong sense}} if given $p,q \in \mathbb{Z}$ there is  $r(p,q)\geq k$ such that $E^{r}_{p,q}\approx E^{\infty}_{p,q}$, for all $r\geq r(p,q)$.

Let $(C,\partial)$ be a chain complex. An \textbf{\emph{increasing filtration}} $F$ on $(C,\partial)$ is a sequence of submodules $F_pC$ of $C$ such that
 $F_pC\subset F_{p+1}C$, for all integer  $p$, and 
 the filtration is compatible with the gradation of $C$, i.e. $F_pC$ is a chain subcomplex of $C$ consisting of $\{F_{p}C_{q}\}$.
A filtration $F$ on $C$ is called \textbf{\emph{convergent}} if $\cap_{p}F_{p}C=0$ and $\cup_{p}F_{p}C = C$.  It is called {\it finite} if there are $p,p' \in \mathbb{Z}$ such that $F_{p}C=0$ and $F_{p'}C=C$. Also, it is said to be \textbf{\emph{bounded below}} if for any $q$ there is $p(q)$ such that $F_{p(q)}C_{q}=0$.

Given a filtration on $C$, the associated  bigraded module $G(C)$ is defined as
$$G(C)_{p,q}=\dfrac{F_{p}C_{p+q}}{F_{p-1}C_{p+q}}.$$

A filtration $F$ on $C$ induces a filtration $F$ on $H_{\ast}(C)$ defined by $F_{p}H_{\ast}(C) = \text{Im} \ [H_{\ast}(F_{p}C)\rightarrow H_{\ast}(C)].$ If the filtration $F$ on $C$ is convergent and bounded below then the same holds for the induced filtration on $H_{\ast}(C)$.

The following theorem (see~\cite{Sp}) shows that one can associate a spectral sequence to a filtered chain complex whenever the filtration is convergent and bounded below.

\begin{theorem}[Spanier,~\cite{Sp}]
Let $F$ be a convergent and bounded below filtration on a chain complex $C$. There is a convergent  spectral sequence with 
$$E^{0}_{p,q} = \dfrac{F_{p}C_{p+q}}{F_{p-1}C_{p+q}} =G(C)_{p,q}  \qquad \text{and} \qquad E^{1}_{p,q} \approx H_{p+q}\left(\dfrac{F_{p}C_{p+q}}{F_{p-1}C_{p+q}}\right)$$
and $E^{\infty}$ is isomorphic to the bigraded module $GH_{\ast}(C)$ associated to the induced filtration on $H_ {\ast}(C)$. 
\end{theorem}

This theorem is proved by expliciting  algebraic formulas for the modules $E^{r}$, which are given by
$$E^{r}_{p,q}=\dfrac{Z^{r}_{p,q}}{Z^{r-1}_{p-1,q+1}+\partial
Z^{r-1}_{p+r-1,q-r+2}},$$
where  $Z^r_{p,q}=\{c\in F_pC_{p+q}\,|\,\,\partial c\in F_{p-r}C_{p+q-1}\}.$

Despite the fact that $E^{\infty}$ does not determine $H_*(C)$ completely, it determines the bigraded module $GH_*(C)$, i.e.
$ E^{\infty}_{p,q}\approx GH_*(C)_{p,q}.$
Moreover, it is a well known fact (see~\cite{D}) that, whenever $GH_*(C)_{p,q}$ is free and the filtration is bounded, then  \begin{equation}\label{eq:seqspec}\displaystyle\bigoplus_{p+q=k}GH_*(C)_{p,q}\approx H_{p+q}(C).\end{equation}

The Spectral Sequence Sweeping Algorithm (SSSA)  was introduced in~\cite{CdRS}  in order to recover the modules and differentials of  a spectral sequence associated to a finite chain complex over $\mathbb{Z}$ with a finest filtration. More specifically, let $(C, \partial)$ be a finite chain complex such that each module $C_{k}$ is finitely generated.  Denote the generators of the $C_{k}$ chain module by  $h_{k}^{1}, \cdots, h_{k}^{c_{k}}$. One can reorder the set of the generators  of $C_{\ast}$ as 
$\{h_{0}^{1},\cdots, h_{0}^{\ell_{0}}, h_{1}^{\ell_{0}+ 1},\cdots, h_{1}^{\ell_{1}}, \cdots, h_{k}^{\ell_{k-1}+1},\cdots, h_{k}^{\ell_{k}}, \cdots \},$
where $\ell_{k}=c_{0}+\cdots + c_{k}$ {\footnote{In order to simplify notation, we use the index $f_{k}$ to denote the first column of $\Delta$ associated to a $k$-chain. Hence $f_{k}=\ell_{k-1}+1$. Moreover, $\ell_{k}$ denotes the last column of $\Delta$ associated to a $k$-chain. }}.  Let  $F$ be a finest filtration on $C$ defined by 
$$ F_{p}C_{k} = \displaystyle\bigoplus_{h^{\ell}_{k}, \ \ell \leq p+1} \mathbb{Z} \langle h^{\ell}_{k} \rangle , $$
for $p \in \mathbb{N}$.
The spectral sequence associated to $(C_{\ast},\partial_{\ast})$ with this finest filtration has a special property: the only $q$ for which $E^{r}_{p,q}$ is non-zero is $q=k-p$, where $k$ is the index of the chain in $F_{p}C \setminus F_{p-1}C$. Hence, in this case, we  omit reference to $q$. It is understood that $E^{r}_{p}$ is in fact $E^{r}_{p,k-p}$.  The SSSA presented below, provides an alternative way to obtain such modules as well as the differentials $d^{r}$'s.

We  can view the differential boundary map $\partial$  as the matrix $\Delta$ where the  column $j$ of $\Delta$ corresponds to the generator $h^{j}_k$ of $C_{\ast}$, and the submatrix $\Delta_{k}$ corresponds to the $k$-th boundary map $\partial_k$.
From now on, the boundary operator $\partial$ and the matrix $\Delta$  will be used interchangeably. 

For  completeness sake we give a summarized description of the Spectral Sequence Sweeping Algorithm below. For more details 
see~\cite{CdRS,MdRS}.

\vspace{0.5cm}

\noindent {\bf  Spectral Sequence Sweeping Algorithm - SSSA}

For a fixed diagonal $r$ parallel and to the right of the main diagonal, the method described below must be applied  simultaneously for all $k$. 
\vspace{0.3cm}

\noindent {\bf Initialization Step:}
\vspace{-0.15cm}
\begin{enumerate}
\item[(1)] Let $\xi_{1}$ be the first diagonal of $\Delta$ that contains non-zero entries $\Delta_{{i,j}}$ in $\Delta_{k}$, which will be called index $k$ primary pivots. Define $\Delta^{\xi_{1}}$ to be $\Delta$ with the $k$- index  {\it primary pivots} marked on the $\xi_{1}$-th diagonal.
\item[(2)] Consider the matrix $\Delta^{\xi_{1}}$. Let $\xi_{2}$ be the first diagonal  greater than $\xi_{1}$ which contains non-zero entries $\Delta^{\xi_{1}}_{{i,j}}$. The construction of $\Delta^{\xi_{2}}$ follows the procedure below. Given a non-zero entry  $\Delta^{\xi_{1}}_{{i,j}}$ on the $\xi_{2}$-th diagonal of $\Delta^{\xi_{1}}$:\\
\textbf{If}  $\Delta^{\xi_{1}}_{{s,j}}$  contains an index $k$ primary pivot for $s>i$, \textbf{then}  the numerical value of the given entry remains the same, $\Delta^{\xi_{2}}_{{i,j}}=\Delta^{\xi_{1}}_{{i,j}}$, and the entry is left unmarked.\\
\textbf{If} $\Delta^{\xi_{1}}_{{s,j}}$ does not contain a primary pivot for $s>i$:\\
\rule{0.5cm}{0pt}\textbf{then if} $\Delta^{\xi_{1}}_{{i,t}}$ contains a primary pivot, for $t< j$, \\
\rule{1cm}{0pt} \textbf{then}  define $\Delta^{\xi_{2}}_{{i,j}}=\Delta^{\xi_{1}}_{{i,j}}$ and mark the  entry $\Delta^{\xi_{2}}_{{i,j}}$ as a {\it change-of-basis pivot}.  \\
\rule{0.5cm}{0pt} \textbf{Else}, define $\Delta^{\xi_{2}}_{{i,j}}=\Delta^{\xi_{1}}_{{i,j}}$ and  permanently mark $\Delta^{\xi_{2}}_{{i,j}}$ as an  index $k$ {\it primary pivot}. \\
\end{enumerate}
\vspace{-0.5cm}

\noindent {\bf Iterative Step: }
Suppose by induction  that $\Delta^{\xi}$ is defined for all $\xi\leq r$ with the primary and change-of-basis pivots marked on the diagonals smaller or equal to $\xi$. 
Without loss of generality, one can assume that there is at least one change-of-basis pivot on the $r$-th diagonal of $\Delta^{r}$. Otherwise, define $\Delta^{r+1}=\Delta^{r}$ with primary pivots and change-of-basis pivots marked as in step $(2)$ below.
\begin{enumerate}
\item[(1)] {\bf Change of basis.}
Let $\Delta^{r}_{{i,j}}$ be a change-of-basis pivot in $\Delta^{r}_{k}$.  Denote by $\Delta^{r}_{i,t}$ the primary pivot in the $i$-th row, with $t<j$. Perform a change of basis  on $\Delta^{r}$ by adding or subtracting the column $t$ to the column $j$ of $\Delta^{r}$, in order to zero out the entry $\Delta^{r}_{{i,j}}$ without introducing non-zero entries in $\Delta^{r}_{{s,j}}$ for $s>i$. 

Define $T^{r}$ as the matrix which performs all the  change of basis on all of the $r$-th diagonal. Define $\Delta^{r+1} = (T^{r})^{-1}\Delta^{r} T^{r}$.

\item[(2)]  {\bf Markup. }
Given a non-zero entry  $\Delta^{r+1}_{{i,j}}$ on the $(r+1)$-th diagonal of $\Delta^{r+1}_{k}$:\\
\textbf{If}  $\Delta^{r+1}_{{s,j}}$  contains a primary pivot for $s>i$, \textbf{then} leave the entry $\Delta^{r+1}_{{i,j}}$  unmarked.\\
\textbf{If} $\Delta^{r+1}_{{s,j}}$ does not contain a primary pivot for $s>i$:\\
\rule{0.5cm}{0pt}\textbf{then if} $\Delta^r_{{i,t}}$ contains a primary pivot, for $t< j$, \\
\rule{0.5cm}{0pt} \textbf{then} mark $\Delta^r_{{i,j}}$ as 
a change-of-basis pivot.\\
\rule{0.5cm}{0pt} \textbf{Else} permanently mark $\Delta^r_{{i,j}}$ as a primary pivot.\\
\end{enumerate}
\vspace{-0.5cm}

\noindent {\bf Final Step:}\\
Repeat the above procedure until all diagonals have been considered.

\vspace{0.5cm}

According to the algorithm, if $\Delta_{{i,j}}^{r}$ is a change-of-basis pivot on the $r$-th diagonal of $\Delta^{r}_{k}$, then  once the corresponding change of basis has been performed, one obtains a new $k$-chain associated to  column $j$ of $\Delta^{r+1}$,  which will be  denoted by $\sigma_{k}^{j,r+1}$. Observe that $\sigma_{k}^{j,r+1}$ is a linear combination over $\mathbb{Z}$ of  columns $t$ and $j$ of $\Delta^{r}$, i.e., $\sigma_{k}^{j,r+1}$ is a linear combination over $\mathbb{Z}$ of $\sigma_{k}^{t,r}$and $ \sigma_{k}^{j,r}$.  Hence,
\begin{eqnarray}\label{eq:change-of-basis}
\sigma^{j,{r+1}}_{k} & = & \underbrace{\sum_{\ell = f_{k}}^{j}c^{j,r}_{\ell} h_k^{\ell}}_{\sigma^{j,r}_{k}} \ \ \  \pm \ \ \  \underbrace{\sum_{\ell = f_{k}}^{j-1}c^{j-1,r}_{\ell} h_k^{\ell}}_{\sigma^{j-1,r}_{k}} \ \  = \ \ \mathsf{c^{j,r+1}_{j}} h_{k}^{j}   + \mathsf{c^{j-1,r+1}_{j-1}} h_{k}^{j-1} + \cdots + \mathsf{c^{f_{k},r+1}_{f_{k}}} h_{k}^{f_{k}}  \nonumber
\end{eqnarray}
where $\mathsf{c^{\ell,r+1}_{k}} \in \mathbb{Z}$, for $\ell = f_{k}, \cdots, j$.  If $\Delta^{r}$ contains an index $k$ primary pivot in the entry $\Delta^{r}_{{s,\bar{\ell}}}$ with $s>i$ and $\bar{\ell}<j$, then $q_{\bar{\ell}}=0$.
Of course, the first column of any $\Delta_{k}$ cannot undergo changes of basis, since there is no column  to its left associated to a $k$-chain. 

The family of matrices $\{\Delta^{r}\}$ produced by the SSSA has several properties, such as: there is at most  one primary pivot in a fixed row or column; if the entry $\Delta^{r}_{j-r,j}$ is a primary pivot or a change-of-basis pivot, then $\Delta^{r}_{s,j} = 0$ for all $s>j-r$; $\Delta^{r}$ is a strictly  upper triangular boundary map, for each $r$.

In~\cite{CdRS} it is proved that the SSSA provides a system that spans the modules $E^{r}$ in terms of the original basis of $C_{\ast}$ and identifies all differentials $d^{r}_{p}:E^{r}_{p}\rightarrow E^{r}_{p-r}$ with primary and change-of -basis pivots on the $r$-th diagonal. 
A formula for the module $Z^{r}_{p,k-p}$ in terms of the chains $\sigma_{k}$'s  is 
\begin{equation}\label{eq:formuladoZ}
Z^{r}_{p,k-p} =  \mathbb{Z} \biggr[ \mu^{p+1,r}\sigma^{p+1,r}_{k} , \mu^{p,r-1}\sigma^{p,r-1}_{k}, \cdots, \mu^{f_{k},r-p-1+f_{k}}\sigma^{f_{k},r-p-1+f_{k}}_{k}  \biggr],
\end{equation}
where $f_{k}$ is the first column of $\Delta$ associated to a $k$-chain, and $\mu^{j,\xi}=0$ whenever there is a primary pivot on column $j$ below  row  $(p-r+1)$ and $\mu^{j,\xi}=1$ otherwise. 
Moreover, if $E^{r}_{p}$ and $E^{r}_{p-r}$ are both non-zero, then the differential $d^{r}:E^{r}_{p}\rightarrow E^{r}_{p-r}$ is induced by  multiplication by $\Delta^{r}_{p-r+1,p+1}$, whenever  this entry is either a primary pivot, change-of-basis pivot or a zero with a column of zero entries below it.

\subsection{Algebraic Cancellation and  Dynamical  Homotopical Cancellation  } \label{subsec:appl2}

Our goal in this subsection is to establish a global homotopical  cancellation result for GS-flows which follows closely the unfolding of its spectral sequence. In order to achieve this, we make use of the Row Cancellation Algorithm (RCA), which reflects more closely the effect  of dynamical homotopical cancellations on the modules $E^r$, while at the same time retaining the relevant  information given by the primary and change of bases pivots of the SSSA.

\begin{theorem}[Algebraic Cancellation and Dynamical Homotopical Cancellation  via Spectral Sequence]\label{cancel1}Let $(C^{\mathcal{GS}}(M,X),\Delta^{\mathcal{GS}})$ be the GS-chain complex associated to a GS-flow $\varphi_{X}$ on a singular 2-manifold $M\in\mathfrak{M}(\mathcal{GS})$, where $X\in\mathfrak{X}_{\mathcal{GS}}(M)$ and $\mathcal{S}=\mathcal{C},\mathcal{W},\mathcal{D}$ or $\mathcal{T}$. Let $(E^r,d^r)$ be the associated spectral sequence for a finest filtration  $F=\{F_pC^{\mathcal{GS}}\}$ on the chain complex.
\begin{enumerate}
    \item If $X\in\mathfrak{X}_{\mathcal{GC}}(M)$ or $X\in\mathfrak{X}_{\mathcal{GW}}(M)$, then the algebraic cancellations of the modules $E^r$ of the spectral sequence are in one-to-one correspondence with dynamical  homotopical cancellations of the singularities  of $X$.
    \item If $X\in\mathfrak{X}_{\mathcal{GD}}(M)$ or $X\in\mathfrak{X}_{\mathcal{GT}}(M)$,  then the algebraic cancellations of the modules $E^r$ of the spectral sequence are in one-to-one correspondence with dynamical homotopical  cancellations of the natures of the singularities  of $X$.
\end{enumerate}
Moreover, the order of homotopical  cancellation occurs as the gap $r$ increases with respect to the filtration $F$.
\end{theorem}

We  want to  associate the data of the spectral sequence $(E^r,d^r)$  with a  dynamical continuation of the initial flow, by means of homotopical  cancellations of the singularities and using as guide the family of matrices $\{\Delta^r\}_{\mathcal{GS}}$ produced by the SSSA, which codifies all data related to the modules and differentials of $(E^r,d^r)$. However, it is easy to see that the matrices $\{\Delta^r\}_{\mathcal{GS}}$ are not in general realized as the GS-boundary operator associated to a GS-flow. Moreover, the changes of basis caused by pivots in row $j-r$ reflect all the changes in connecting orbits caused by the cancellation of the consecutive generators $h_k^{j}$ and $h_{k-1}^{j-r}$. When we cancel the pair of generators $h_k^{j}$ and $h_{k-1}^{j-r}$, then
\begin{itemize}
    \item all the flow lines between the corresponding singularities associated to generators of $k$-nature  and $h_{k-1}^{j-r}$ are immediately removed and new connections are born; 
    \item also all the flow lines between $h_{k}^{j}$ and singularities associated to generators of  $(k-1)$-nature are immediately removed.
\end{itemize}
In this sense, in order to interpret  the SSSA as a dynamical homotopical cancellation, we have to perform the changes of basis that occur therein in a different order to reflect the death and birth of connections. More specifically, if $\Delta^{r}_{j-r,j}=\pm 1$ is a primary pivot marked in step $r$ of the SSSA
 all changes of basis caused by $\Delta^{r}_{j-r,j}$ must be performed in  step $r+1$. The  algorithm which reflects it is called the {\it Row cancellation  Algorithm} (RCA) and it was first introduced in~\cite{BLMdRS1,BLMdRS2}.
One emphasizes that whenever a primary pivot is marked, all the changes of basis caused by this pivot are performed in the next step.

\vspace{0.5cm}

\begin{bf}Row Cancellation Algorithm - RCA\end{bf}
\begin{description}
\item[\textbf{Initialization Step:}]\mbox{}\\
 $\begin{array}{@{}l}
\left[\begin{tabular}{l}
 $r=0$\\
 $\tilde{\Delta}^r=\Delta$\\
 $\tilde{T}^r=I$ ($m\times m$ identity matrix) \end{tabular}\right.\end{array}$

\item[\textbf{Iterative Step:}] (Repeated until all diagonals parallel and to the right of the main diagonal have been swept)\\
$\begin{array}{@{}l}
\left[\begin{tabular}{l}
\textbf{Matrix $\tilde{\Delta}$ update}\\
\begin{tabular}{@{\hspace{.5cm}}l}
$r\leftarrow r+1$ \\
$\tilde{\Delta}^r = (\tilde{T}^{r-1})^{-1} \tilde{\Delta}^{r-1}\tilde{T}^{r-1}$\\
\end{tabular}\end{tabular}\right.\\
\end{array}$
\\[5pt]
$\begin{array}{@{}l}
\left[\begin{tabular}{l}
\textbf{Markup}\\
\begin{tabular}{@{\hspace{.5cm}}l}
Sweep entries of $\tilde{\Delta}^r$ in the $r$-th diagonal: \\
\textbf{\textsf{If}} $\tilde{\Delta}^r_{j-r,j}\neq 0$ \textbf{\textsf{and}} $\tilde{\Delta}^r_{\textbf{\large .},j}$ does not contain a primary pivot\\
\rule{.5cm}{0pt}\textbf{\textsf{Then}} permanently mark $\tilde{\Delta}^r_{j-r,j}$ as a primary pivot\\
\end{tabular}\end{tabular}\right.\\
\end{array}$
\\[5pt]
$\begin{array}{@{}l}
\left[\begin{tabular}{l}
\textbf{Matrix $\tilde{T}^r$ construction}\\
\begin{tabular}{@{\hspace{.5cm}}l}
$\tilde{T}^r \leftarrow I$\\
\textbf{\textsf{For each}} primary pivot $\tilde{\Delta}^r_{p-r,p}$ such that $j<m$,
change the $p$-th row of $\tilde{T}^r$ as follows\\
\begin{tabular}{@{\hspace{.5cm}}l}
$\tilde{T}^r_{p,\ell}\leftarrow -(1/\tilde{\Delta}^r_{p-r,p})\tilde{\Delta}^r_{p-r,\ell}$, for $\ell=p+1,\dots,m$
\end{tabular}
\end{tabular}\end{tabular}\right.\\
\end{array}$

\item[\textbf{Final Step:}]\mbox{} \\
$\begin{array}{@{}l}\left[
\begin{tabular}{l}
\textbf{Matrix $\tilde{\Delta}$ update}\\
\begin{tabular}{@{\hspace{.5cm}}l}
$r\leftarrow r+1$ \\
$\tilde{\Delta}^r = (T^{r-1})^{-1} \tilde{\Delta}^{r-1} T^{r-1}$\\
\end{tabular}
\end{tabular}\right.
\end{array}$

\end{description}

\vspace{0.3cm}

In~\cite{BLMdRS2}, it was proved the Primary Pivots Equality Theorem  which states  that  the primary pivots on the $r$-th diagonal of $\widetilde{\Delta}^{r}$ marked in the $r$-th step of the RCA coincide with the ones on the $r$-th diagonal of $\Delta^r$ marked in the $r$-th step of the SSSA.
More details on this algorithm can be found in~\cite{BLMdRS2}. 


\begin{theorem}[Family of GS-Flows via Spectral Sequences]\label{cancel2} Let $(C^{\mathcal{GS}}(M,X),\Delta^{\mathcal{GS}})$ be the GS-chain complex associated to a GS-flow $\varphi_{X}$ on a singular 2-manifold $M\in\mathfrak{M}(\mathcal{GS})$, where $X\in\mathfrak{X}_{\mathcal{GS}}(M)$ and $\mathcal{S}=\mathcal{C},\mathcal{W},\mathcal{D}$ or $\mathcal{T}$. The RCA for the GS-boundary map $\Delta^{\mathcal{GS}}$ produces a family of GS-flows $\{\varphi^1=\varphi_X,\varphi^2,\ldots ,\varphi^{\omega}\}$ where $\varphi^r$ continues to $\varphi^{r+1}$ by cancelling  pairs of singularities  of $\varphi^{r}$ having gap $r$ with respect to the filtration $F$. Moreover, the flow $\varphi^{\omega}$ is a minimal GS-flow in the sense that there is no more possible  homotopical cancellations to be done.
\end{theorem}

\begin{proof}
In order to prove the theorem, firstly one analyzes  the  local and global effects a homotopical cancellation of a pair of consecutive singularities has on the GS-boundary map $\Delta^{\mathcal{GS}}(M,\varphi_{X^\prime})$ of the new flow $\varphi_{X^\prime}$.  Secondly, one constructs a family of GS-flows  $\{\varphi^1=\varphi_X,\varphi^2,\ldots ,\varphi^{\omega}\}$ via the RCA in such way that the connections of the flow $\varphi^r$ are codified in the $r$-th matrix produced by the RCA.


To simplify the exposition, one considers first the cases where  $X \in \mathfrak{X}_{\mathcal{GS}}(M)$ for $\mathcal{S} = \mathcal{C}$ or $\mathcal{W}$, since in these cases each singularity posseses only one nature, hence there is only one  generator corresponding to  the nature of the singularity.

Without loss of generality, the set of orientations  of the unstable manifolds for the Morsified flow $\varphi_{\widetilde{X}}$ will be considered, within this proof,  as the one where all  orientations  of   the unstable manifolds of repeller singularities are the same. This assumption  guarantees that, whenever  $h_1$ is a saddle singularity,  the  flow lines  in $W^{s}(h_{1}) \backslash\{h_{1}\}$ either  have opposite characteristic signs or they are null. On the other hand, by definition, the flow lines  in $W ^{u}(h_{1})\backslash\{h_{1}\}$ always  have opposite  characteristic signs, if they are not null.

Throughout the proof, denote by $n(h_{k},h_{k-1},\varphi)$ the intersection number of $h_{k}$ and $h_{k-1}$ with respect to the flow  $\varphi$.

Let $h^{j}_{k}$ and $h^{j-r}_{k-1}$ be consecutive singularities of the vector field $X$.  By the Dynamical Homotopical Cancellation Theorem for GS-Flows (Theorem~\ref{teo_cancelamento}), if  $n(h^{j}_{k},h^{j-r}_ {k-1},\varphi_X)=\pm 1$ then  these singularities can be cancelled, i.e. there is a GS-flow $\varphi_{X^\prime}$ which coincides with $\varphi_X$ outside a neighborhood of $\{h^{j}_{k}, h_{k-1}^i,h^{j-r}_{k-1}\} \cup \mathcal{O}(u_1) \cup \mathcal{O}(u_2)$, up to homotopy, where $\mathcal{M}^{h^{j}_{k}}_{h_ {k-1}^{j-r}}=\{u_1\}$ and $\mathcal{M}^{h^{j}_{k}}_{h_ {k-1}^{i}}=\{u_2\}$.
For $k=1$ (resp., $k=2$), let $h^{j}_{1}$ (resp., $h^{j-r}_{1}$) be a saddle singularity   that  connects with the  attracting (resp., repelling)  singularities $h_{0}^{j-r}$ and $h_0^i$ (resp., $h_2^j$ and $h_2^p$). If $h^{j}_{1}$ cancels with $h^{j-r}_{0}$ (resp., $h^{j}_{2}$ cancels with $h_{1}^{j-r}$), then each saddle $h^{p}_{1}$ (resp., $h^i_1$) which connects with $h^{j-r}_{0}$ (resp., $h^{j}_{2}$) in $\varphi_{X}$ will connect with $h^{i}_{0}$ (resp., $h^{p}_{2}$) in $\varphi_{X^{\prime}}$. Since the old and new connections have the same characteristic signs, then
$$n(h^{p}_{1},h^{i}_{0},\varphi_{X^{\prime}}) = n(h^{p}_{1},h^{j-r}_{0},\varphi_X) + n(h^{p}_{1},h^{i}_{0},\varphi_X)$$
(resp., $n(h^{p}_{2},h^{i}_{1},\varphi_{X^{\prime}}) = n(h^{p}_{2},h^{j-r}_{1},\varphi_X) + n(h^{p}_{2},h^{i}_{1},\varphi_X)$).

\begin{figure}[h!]
    \centering
        \includegraphics[width=0.71\textwidth]{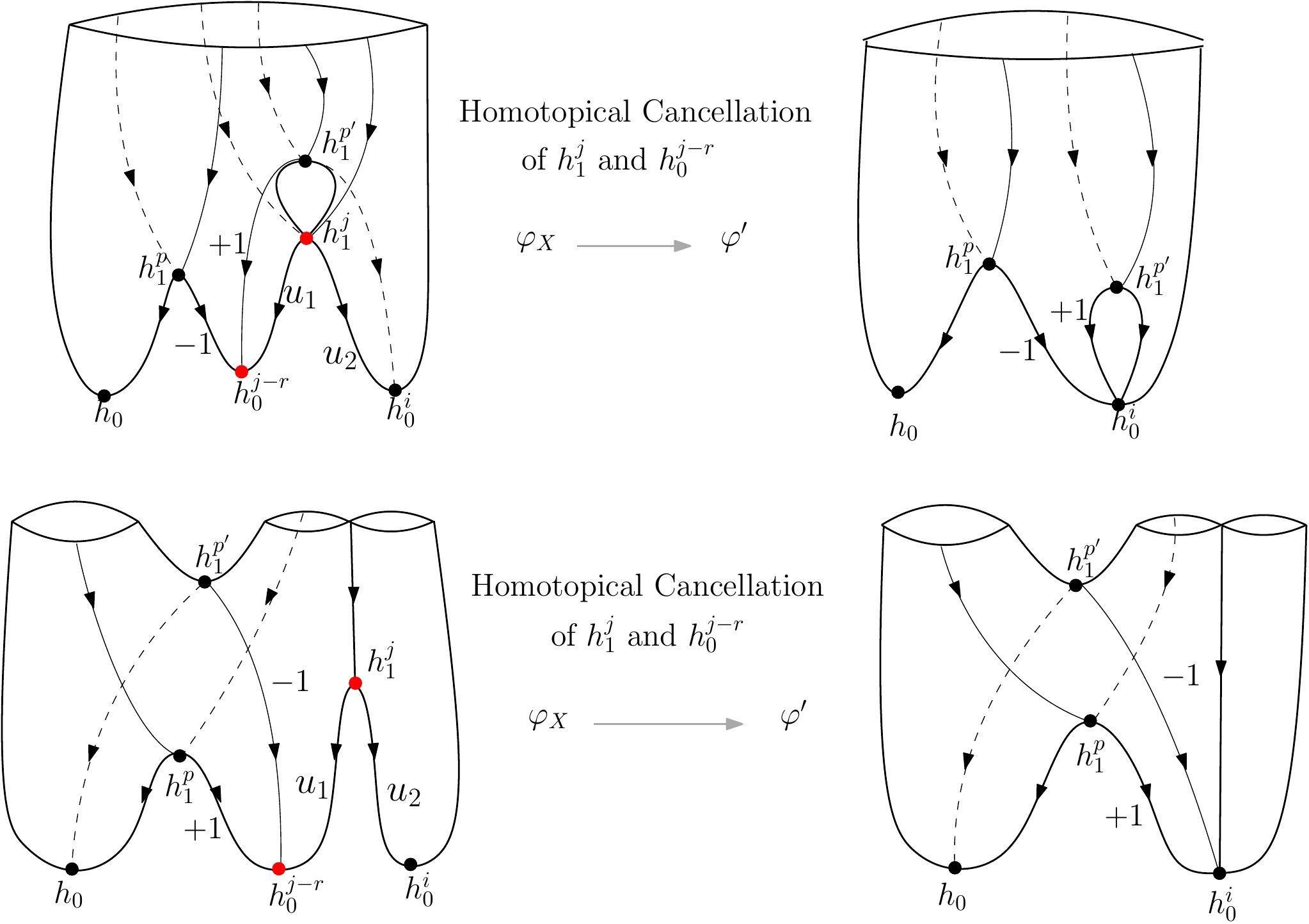}
    \caption{Birth and death of connections - characteristic signs.}\label{cancel_prova}
\end{figure}

Since the flow $\varphi_{X^{\prime}}$ coincides with the flow $\varphi_X$ outside a neighborhood $U$ of $\{h^{j}_{k}, h^{i}_{k-1}, h^{j-r}_{k-1}\} \cup \mathcal{O}(u_1) \cup \mathcal{O}(u_2)$, up to homotopy, the only intersection numbers that are modified after a homotopical cancellation are those $n(h^{p}_{1},h^{i}_{0})$, where $h^{p}_{1}$ is such that $\mathcal{M}^{h^{p}_{1}}_{h^{j-r}_{0}} \neq \emptyset$  in the case of saddle-sink  homotopical cancellation, and those $n(h^{p}_{2},h^{i}_{1})$, where $h^{p}_{2}$ is such that $\mathcal{M}^{h^{p}_{2}}_{h^{j-r}_{1}} \neq \emptyset$, in the case of source-saddle homotopical cancellation.

The GS-boundary map  $\Delta^{\mathcal{GS}}(M, X^{\prime})$ can be obtained  from $\Delta^{\mathcal{GS}}(M,X)$  in the following way:
\begin{itemize}
\item If a saddle singularity $h^{j}_{1}$ is cancelled with an attracting singularity $h_{0}^{j-r}$, then define the matrix $\widetilde{\Delta}$ to be the matrix obtained from  $\Delta^{\mathcal{GS}}(M,X)$ by replacing row $i$  by the sum of row $(j-r)$ to row $i$. Then  $\Delta^{\mathcal{GS}}(M, X^{\prime})$  is the submatrix of $\widetilde{\Delta}$ which does not contain rows  $j-r$, $j$ and neither columns $j-r$, $j$.
\item If a repelling singularity $h^{j}_{2}$ is cancelled with a saddle singularity   $h_{1}^{j-r}$, then define the matrix $\widetilde{\Delta}$ to be the matrix obtained from $\Delta^{\mathcal{GS}}(M,X)$ by replacing column $p$  by the sum of column $j$ to column $p$. Then $\Delta^{\mathcal{GS}}(M, X^{\prime})$ is the submatrix of $\widetilde{\Delta}$ which does not contain rows $j-r$, $j$ rows nor columns $j-r$, $j$.
\end{itemize}
It is  straightforward to see that this corresponds to the row operations performed by the RCA.

Consider the matrices $\{\widetilde{\Delta}^r\}$  produced by the RCA when applied to  $\Delta^{\mathcal{GS}}(M,X)$. Define $\varphi^{1} = \varphi_X$ and $\varphi^{r+1}$ to be a flow  obtained from $\varphi^{r}$ by cancelling all pairs of consecutive singularities corresponding to primary pivots on the $r$-th diagonal of  $\widetilde{\Delta}^{r}$. In order to show that these flows are well defined, we have to prove that whenever a primary pivot $\Delta^r_{j-r,j}$ on the $r$-th diagonal of $\widetilde{\Delta}^r$ is marked, it is actually an intersection number between two consecutive singularities  $h^{j}_{k}$ and $h^{j-r}_{k-1}$ of the flow $\varphi^r$ and hence they can be cancelled by the Dynamical Homotopical Cancellation Theorem (Theorem~\ref{teo_cancelamento}).

Since $\varphi^{1} = \varphi_X$, the GS-boundary map $\Delta^{\mathcal{GS}}(M,{\varphi^{1}})$ is $\widetilde{\Delta}^{1}$. Let $\Delta_{j-1,j}^{1}=\pm 1$ be a primary pivot on the first diagonal of $\widetilde{\Delta}^{1}$. By definition,  this primary pivot represents the  intersection number between two singularities of the flow $\varphi^{1}$, namely  $h^{j}_{k}$ and $h^{j-1}_{k-1}$,  which are consecutive since the gap between them is one. 
Using the Dynamical Homotopical Cancellation Theorem (Theorem~\ref{teo_cancelamento}), we can define a flow $\varphi^{2}$ by cancelling all pairs of consecutive singularities corresponding  to primary pivots on the first diagonal of  $\widetilde{\Delta}^{1}$. Moreover,   the GS-boundary map $\Delta^{\mathcal{GS}}(M,{\varphi^{2}})$ is the submatrix obtained from $\widetilde{\Delta}^{2}$ which does not contain the columns and rows corresponding to the cancelled singularities. Because of this and the fact that all non-zero entries of $\widetilde{\Delta}^{2}$ belong to $\Delta^{\mathcal{GS}}(M,{\varphi^{2}})$,  each non-zero entry of $\widetilde{\Delta}^{2}$ represents an intersection number between two singularities of $\varphi^{2}$. Observe that two singularities $h^{j}_{k}$ and $h_{k-1}^{j-2}$ of $\varphi^{2}$ with gap two in the filtration $F$ are consecutive in the flow $\varphi^{2}$ since all the gap 1 singularities have been cancelled in the previous stage.

Suppose that $\varphi^{r}$ is well defined, that is, each  primary pivot $\Delta^{r-1}_{j-(r-1),j}$ on the diagonal $(r-1)$ of $\widetilde{\Delta}^{r-1}$ corresponds to the intersection number of consecutive singularities $h^{j}_{k}$ and $h^{j-(r-1)}_{k-1}$ of $\varphi^{r-1}$ and the GS-boundary map $\Delta^{\mathcal{GS}}(M,{\varphi^{r}})$ is a submatrix of $\widetilde{\Delta}^{r}$ which does not contain  columns and rows of  $\widetilde{\Delta}^{r}$ corresponding to all
primary pivots marked until the diagonal $r-1$.  These correspond to all singularities of $\varphi$ of gap less than or equal to $r-1$.
Under these hypothesis  singularities $h^{i}_{k}$ and $h^{i-r}_{k-1}$ of $\varphi^{r}$ with gap $r$ with respect to the filtration $F$ are consecutive in the flow  $\varphi^{r}$. Hence two singularities corresponding to a primary pivot on the diagonal $r$ of $\widetilde{\Delta}^{r}$ can be cancelled, by the Dynamical Homotopical  Cancellation Theorem (Theorem~\ref{teo_cancelamento}). Therefore,  $\varphi^{r+1}$ is a well defined flow obtained from $\varphi^{r}$ by cancelling all pairs of critical points corresponding to primary pivots on the diagonal $r$ of  $\widetilde{\Delta}^{r}$. Moreover, the GS-boundary map $\Delta^{\mathcal{GS}}(M,{\varphi^{r+1}})$ is a submatrix of $\widetilde{\Delta}^{r+1}$ which does not contain  columns and rows of  $\widetilde{\Delta}^{r+1}$ corresponding to all primary pivots marked until step $r$.
The flow $\varphi_X$ continues to $\varphi^r$ for all $r$.

Assume that $X \in \mathfrak{X}_{\mathcal{GS}}(M)$ for $\mathcal{S} = \mathcal{D}$ or $\mathcal{T}$. The proof follows the same steps as above by considering the homotopical cancellation of consecutive generators of the natures of the singularities.
\end{proof}

\begin{proof}[Proof of Theorem~\ref{cancel1}] 
By the Primary Pivots for Orientable Surfaces Theorem, see~\cite{BLMdRS1,BLMdRS2},  the primary pivots are always equal to $\pm 1$ when working on orientable surfaces. 
Thus the differentials $d^{r}_p:E^r_{p}\to E^r_{p-r}$  induced by the primary pivots are isomorphisms and the ones associated to change of basis pivots always correspond to zero maps.
Consequently, if $d^{r}_p$ is non-zero differential, then,  at the next stage of the spectral sequence, the 
algebraic cancellation $E^{r+1}_{p}=E^{r+1}_{p-r}=0$ occurs. 

An algebraic cancellation
$E^{r+1}_{p}=E^{r+1}_{p-r}=0$ is associated to a primary pivot
$\Delta^{r}_{p-r+1,p+1}=\pm 1$ on the $r$-th diagonal of
$\Delta^{r}$ produced by the $r$-th step of the SSSA. The row $p-r+1$ is associated to 
$h_{k-1}^{p-r+1}\in F_{p-r}C^{\mathcal{GS}}_{k-1}\setminus F_{p-r-1}C^{\mathcal{GS}}_{k-1}$ and
the column $p+1$ is associated to $h_{k}^{p+1}\in F_pC^{\mathcal{GS}}_k\setminus
F_{p-1}C^{\mathcal{GS}}_k$ in a gradient flow $\varphi$ associated to $f$. By the Primary Pivots Equality Theorem in~\cite{BLMdRS2}, the primary pivot $\Delta^{r}_{p-r+1,p+1}=\pm 1$ is also a primary pivot $\widetilde{\Delta}^{r}_{p-r+1,p+1}=\pm 1$ of the RCA. As it was shown in the proof of Theorem~\ref{cancel2}, the primary pivot $\Delta^{r}_{p-r+1,p+1}$ corresponds  to the  intersection number of
two consecutive generators of natures of the singularities $h_{k}^{p+1}$ and $h_{k-1}^{p-r+1}$ of the flow $\varphi^r$. This pair can be homotopical cancelled by the Dynamical Homotopical Cancellation Theorem (Theorem~\ref{teo_cancelamento}).

Moreover,  $E^{r}_p$ and $E_{p-r}^r$ correspond to generators  of saddle and attractor (or repeller and saddle) natures, respectively,  with gap $r$ with respect to the filtration $F$. Therefore, the  dynamical and algebraic cancellations occur with increasing gap.
\end{proof}

\subsection{Examples}\label{subsec:6.3}

In this subsection we present some examples where we explore the algebraic cancellations of the modules of the spectral sequence  and their corresponding dynamical  homotopical cancellations.

Throughout this section, the primary pivots are the entries indicated by   darker edge and the change of basis pivots are indicated  by  dashed edges, null entries are left blank and the diagonal being swept is indicated with a gray line.  

\begin{example}
Consider the singular manifold  $M\in\mathfrak{M}(\mathcal{GC})$  and a GS-flow $\varphi_{X}$ associated to a vector field $X\in\mathfrak{X}_{\mathcal{GC}}(M)$  as in Figure~\ref{ex_cone1}.  Consider as well a choice of orientations on the unstable manifolds of the critical points of a Morsification  $\widetilde{M}$. Then we are able to determine the GS-characteristic signs of the orbits of $\varphi_{X}$ , as it is shown in Figure~\ref{ex_cone1}.

\begin{figure}[H]
    \centering
        \includegraphics[width=0.85\textwidth]{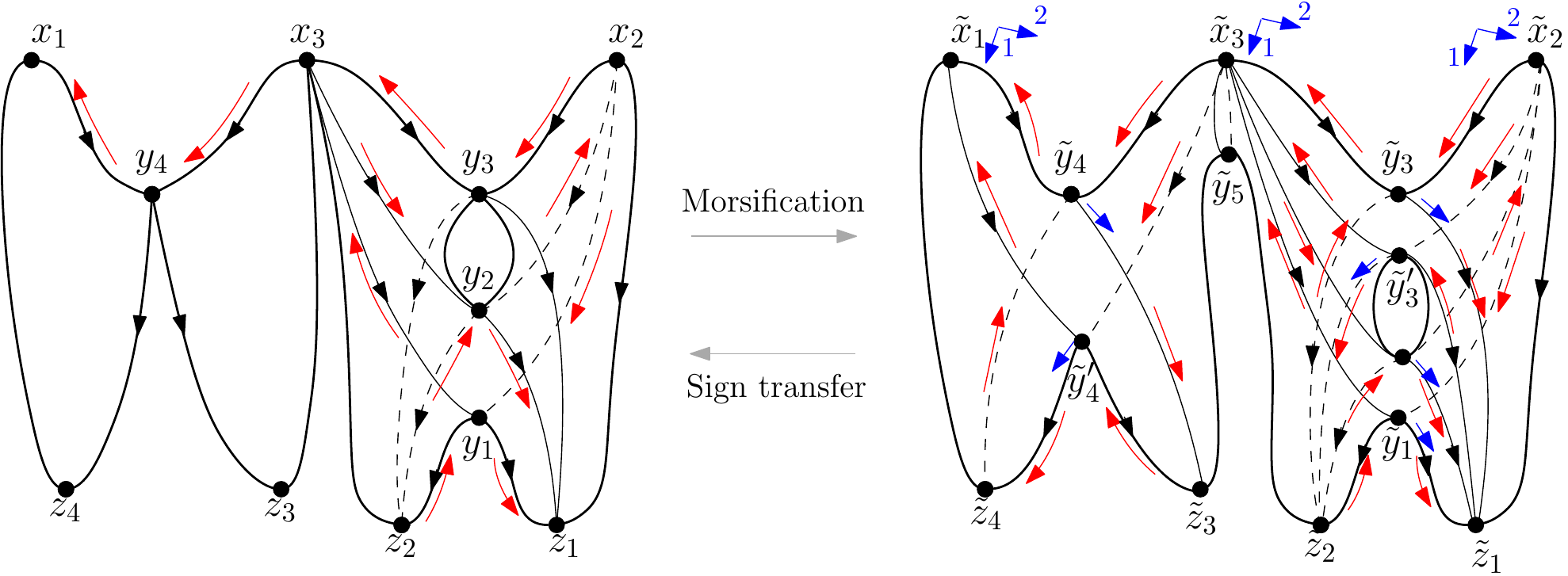}
    \caption{A GS-flow with cone singularities, its Morsification and sign tranfers.}\label{ex_cone1}
\end{figure}

The GS-chain groups are:
$C_0(M,X)^{\mathcal{GC}}=\mathbb{Z} [\langle z_1\rangle,\langle z_2 \rangle, \langle z_3\rangle, \langle z_4\rangle ]$, 
$C_1(M,X)^{\mathcal{GC}}=\mathbb{Z} [\langle y_1\rangle, \langle y_2\rangle, \langle y_3\rangle, \langle y_4\rangle ] $, $ C_2(M,X)^{\mathcal{GC}}=\mathbb{Z} [\langle x_1\rangle,\langle x_2\rangle,\langle x_3\rangle]$,
and $C_k(M)=0, k\neq 0,1,2$.
The GS-boundary operator $\Delta^{\mathcal{GS}}_{\ast}$ is  given by the matrix in Figure~\ref{cone_matriz0}.  

Consider a finest  filtration on the GS-chain complex $(C_{\ast}^{\mathcal{GC}}(M,X),\Delta^{\mathcal{GC}}_{\ast})$, namely, $F_0C^{\mathcal{GC}}  = \mathbb{Z}[z_1]$, $F_1C^{\mathcal{GC}}  = \mathbb{Z}[z_1, z_2]$, $F_2C^{\mathcal{GS}}  = \mathbb{Z}[z_1, z_2,z_3]$, $F_3C^{\mathcal{GS}}  = \mathbb{Z}[z_1, z_2,z_3,z_4]$, $F_4C^{\mathcal{GC}}  = \mathbb{Z}[z_1, z_2,z_3,z_4,y_1]$, $F_5C^{\mathcal{GC}}  = \mathbb{Z}[z_1, z_2,z_3,z_4,$ $y_1,y_2]$, $F_6C^{\mathcal{GC}}  = \mathbb{Z}[z_1, z_2,z_3,z_4,y_1,y_2,y_3]$,
$F_7C^{\mathcal{GC}}  = \mathbb{Z}[z_1, z_2,z_3,z_4,y_1,y_2,y_3,y_4]$,
$F_8C^{\mathcal{GC}}  = \mathbb{Z}[z_1, z_2,z_3,z_4,y_1,y_2,$ $y_3,y_4,x_1]$,
$F_9C^{\mathcal{GC}}  = \mathbb{Z}[z_1, z_2,z_3,z_4,y_1,y_2,y_3,y_4,x_1,x_2]$ and 
$F_{10}C^{\mathcal{GC}}  = \mathbb{Z}[z_1, z_2,z_3,z_4,y_1,y_2,y_3,y_4,x_1,x_2,x_3]$. The  spectral sequence associated to  $(C_{\ast}^{\mathcal{GC}}(M,X),\Delta^{\mathcal{GC}}_{\ast})$ enriched with the filtration $F$ is shown in Figure~\ref{spec}.

\begin{figure}[H]
      \centering
        \includegraphics[width=0.9\textwidth]{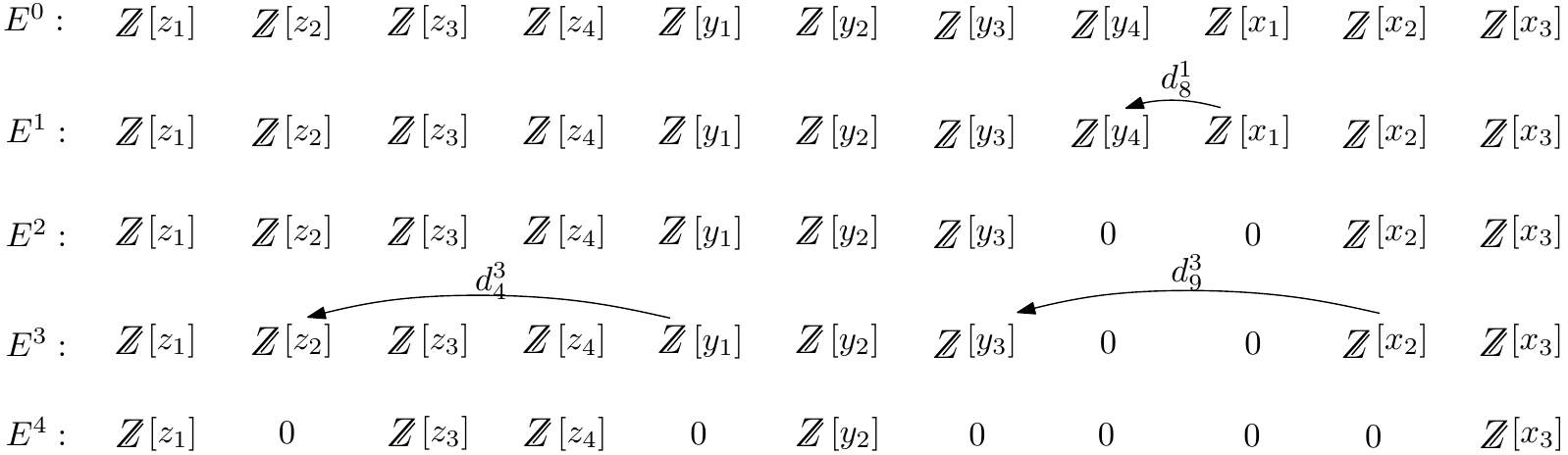}
    \caption{The spectral sequence for $(C_{\ast}^{\mathcal{GC}}(M,X),\Delta^{\mathcal{GC}}_{\ast})$ with filtration $F$.}\label{spec}
\end{figure}

Applying the  SSSA to the GS-boundary differential $\Delta^{\mathcal{GC}}$, one obtains the sequence of matrices  $\Delta^{1},\cdots,\Delta^{5}$ as in  Figures 
~\ref{cone_matriz1},$\cdots$,~\ref{cone_matriz5}, respectively, where the singularities are identified by  $h_0^i=z_i$, $h_1^{i+4}=y_i$, for $i=1\dots 4$ and $h_2^{i+8}=z_i$, for $i=1,\dots, 3$.

\begin{figure}[H]
\centering
\subfloat[$\Delta^0$, the GS-boundary operator.]{
\includegraphics[height=6.5cm]{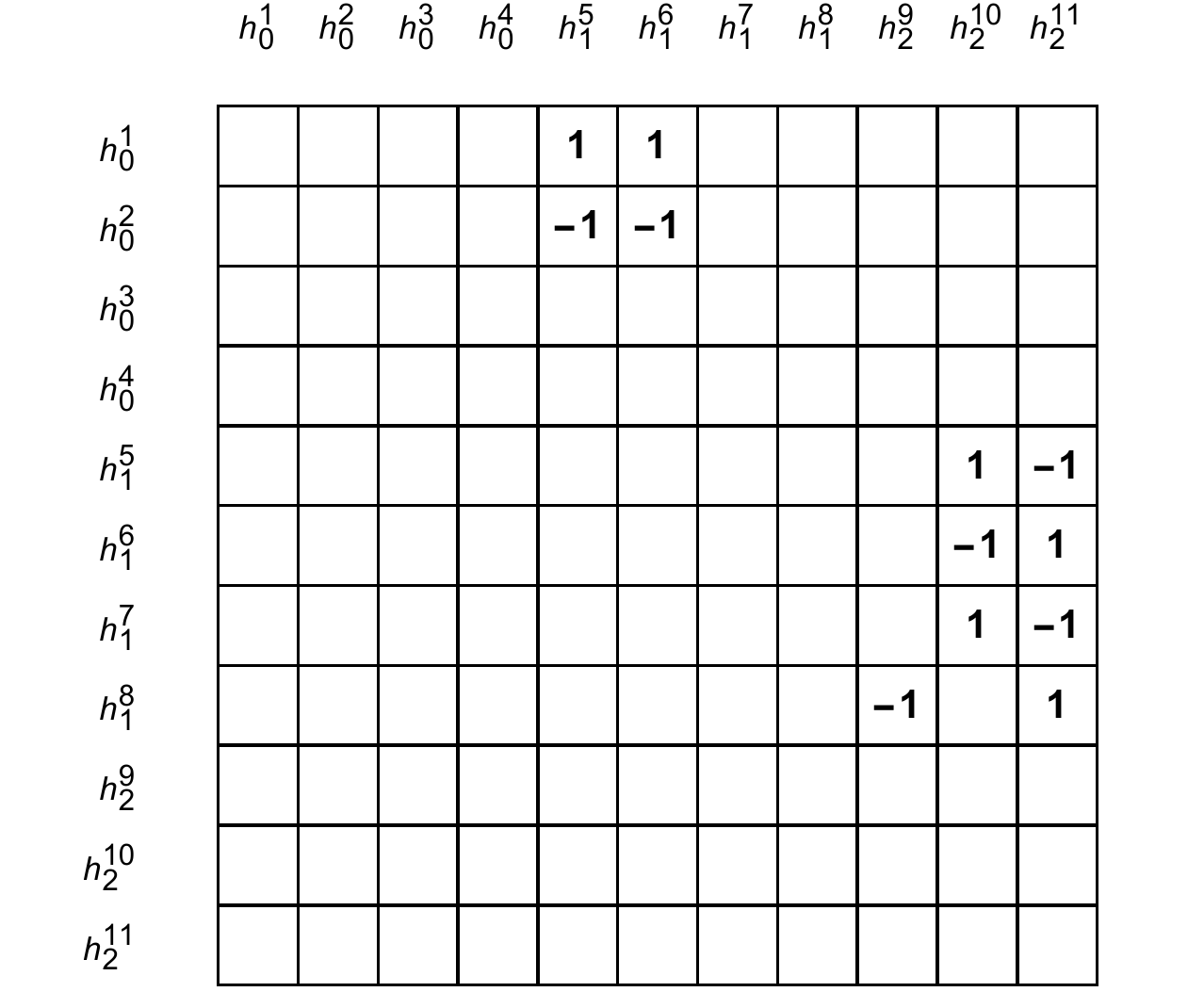}
\label{cone_matriz0}
}
\quad 
\subfloat[$\Delta^1$, sweeping 1-st diagonal.]{
\includegraphics[height=6.5cm]{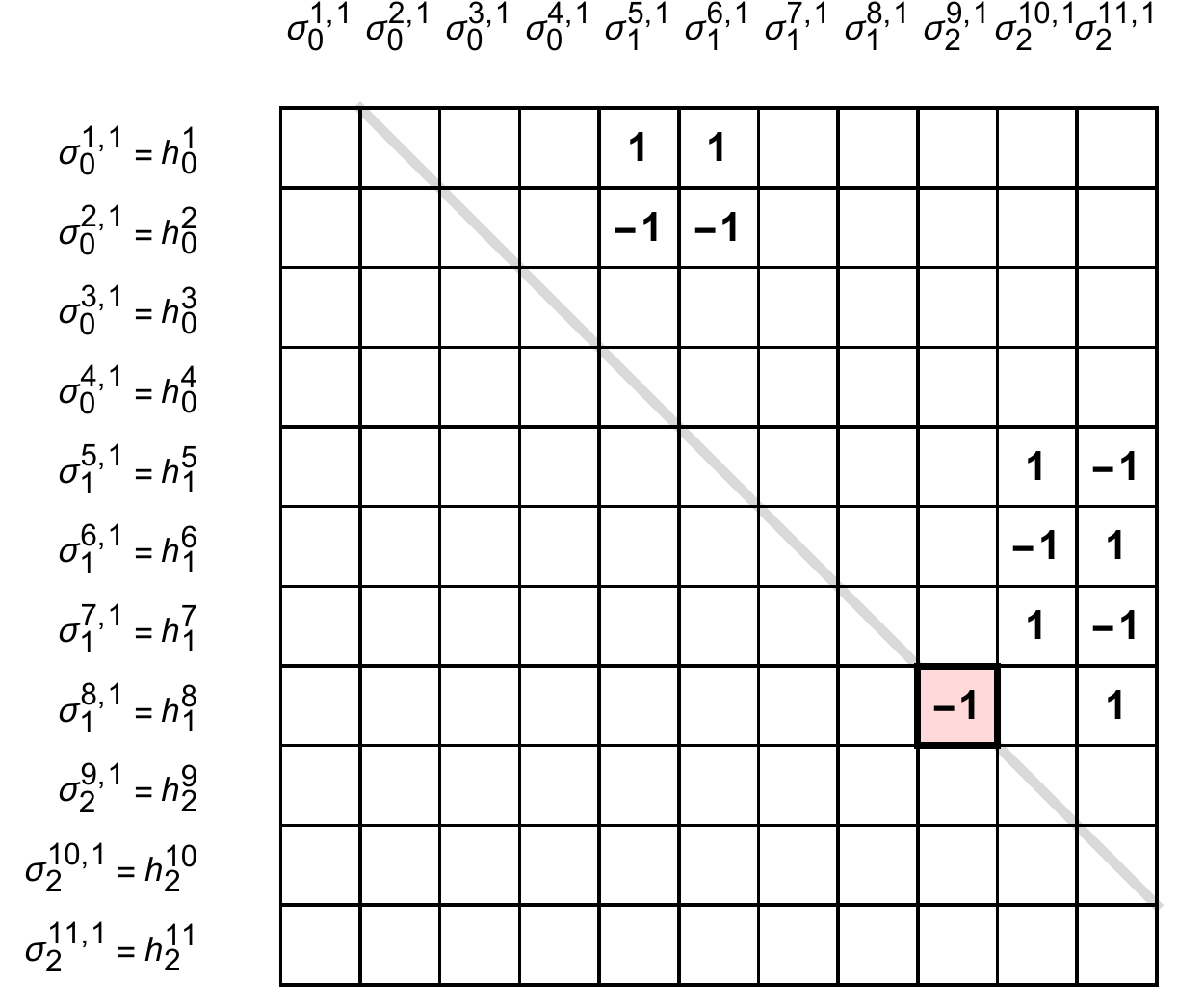}
\label{cone_matriz1}
}\\
\subfloat[$\Delta^2$, sweeping 2-nd diagonal.]{
\includegraphics[height=6.5cm]{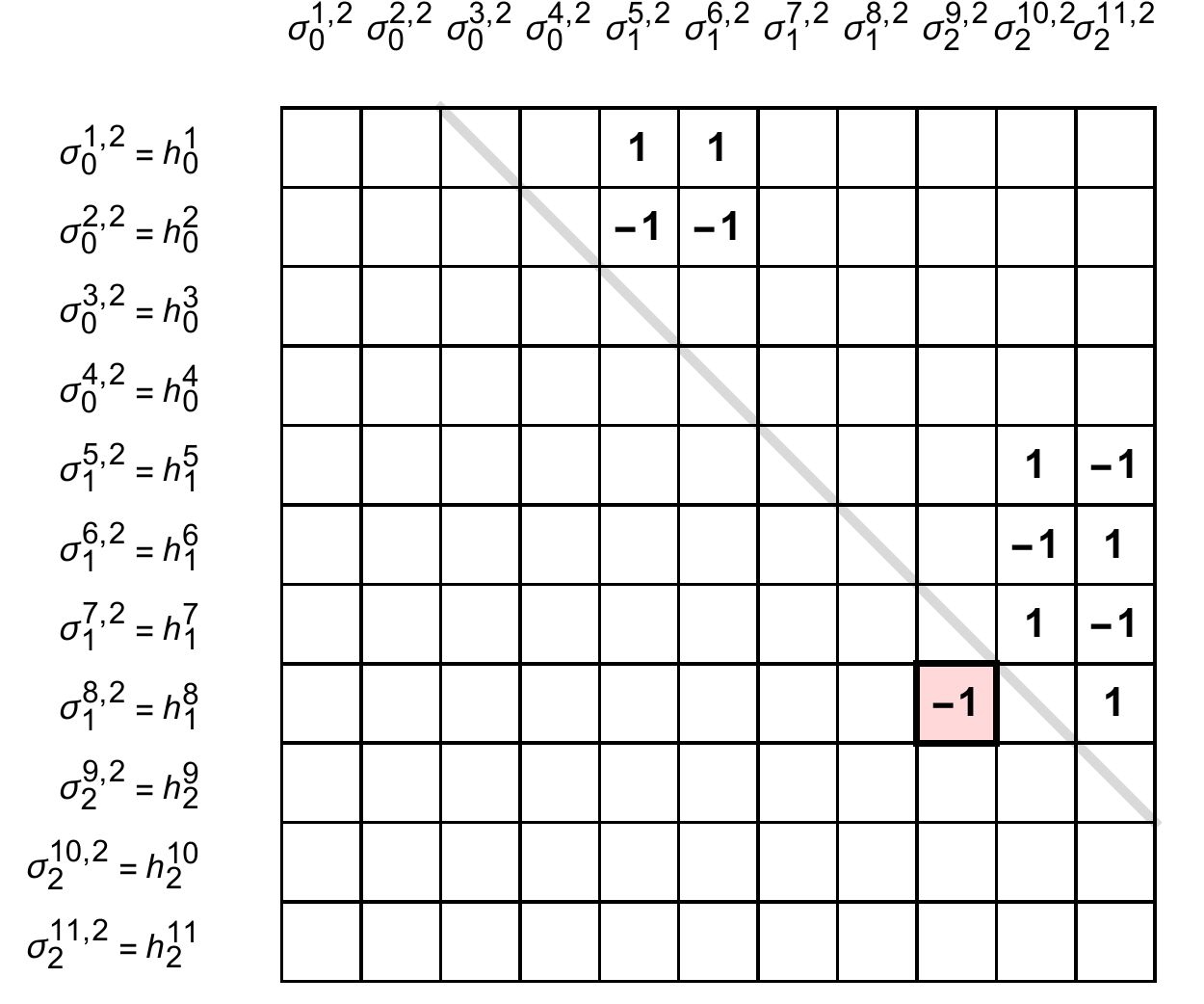}
\label{cone_matriz2}
}
\quad 
\subfloat[$\Delta^3$, sweeping 3-rd diagonal.]{
\includegraphics[height=6.5cm]{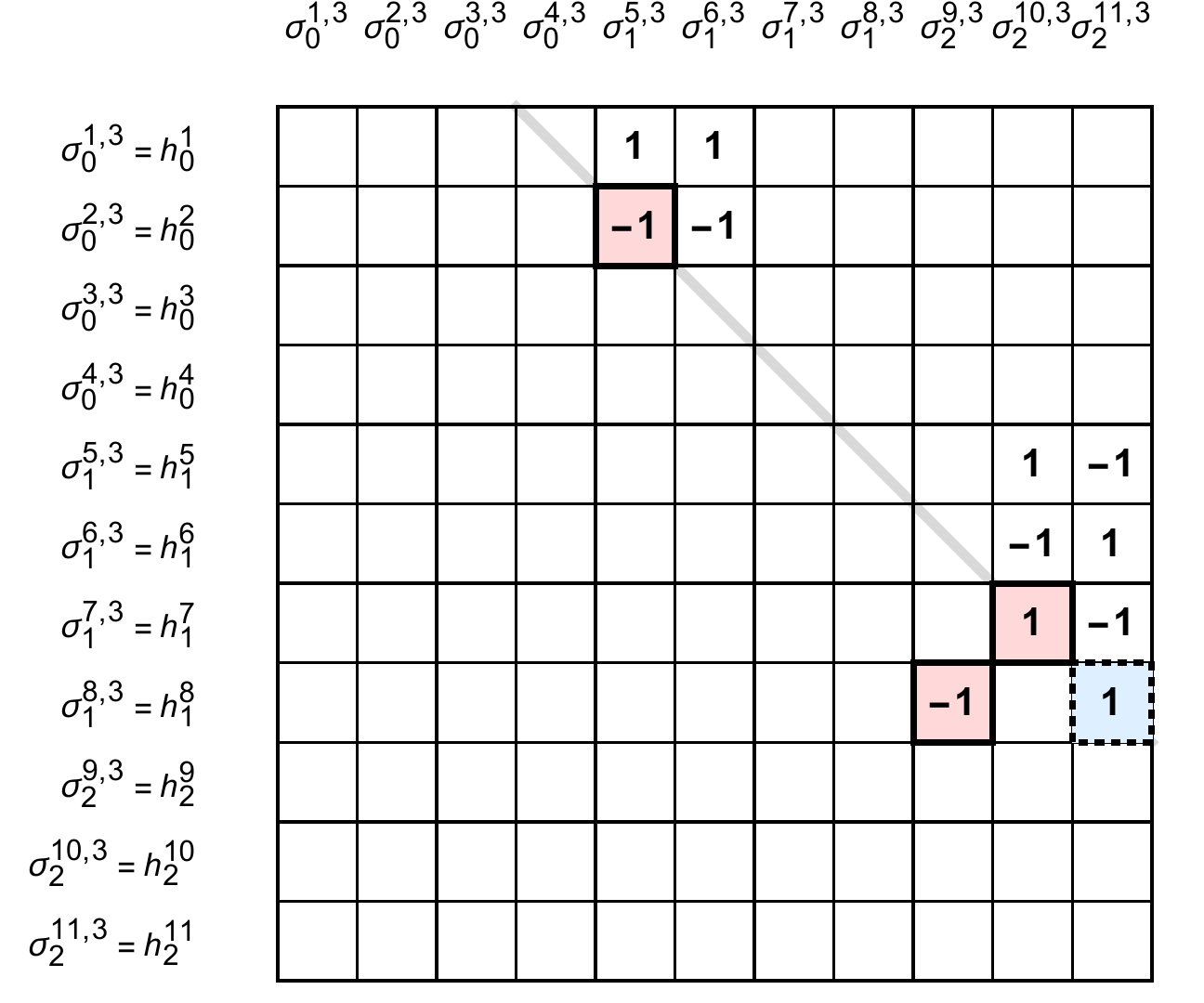}
\label{cone_matriz3}
}\\
\subfloat[$\Delta^4$, sweeping 4-th diagonal.]{
\includegraphics[height=6.5cm]{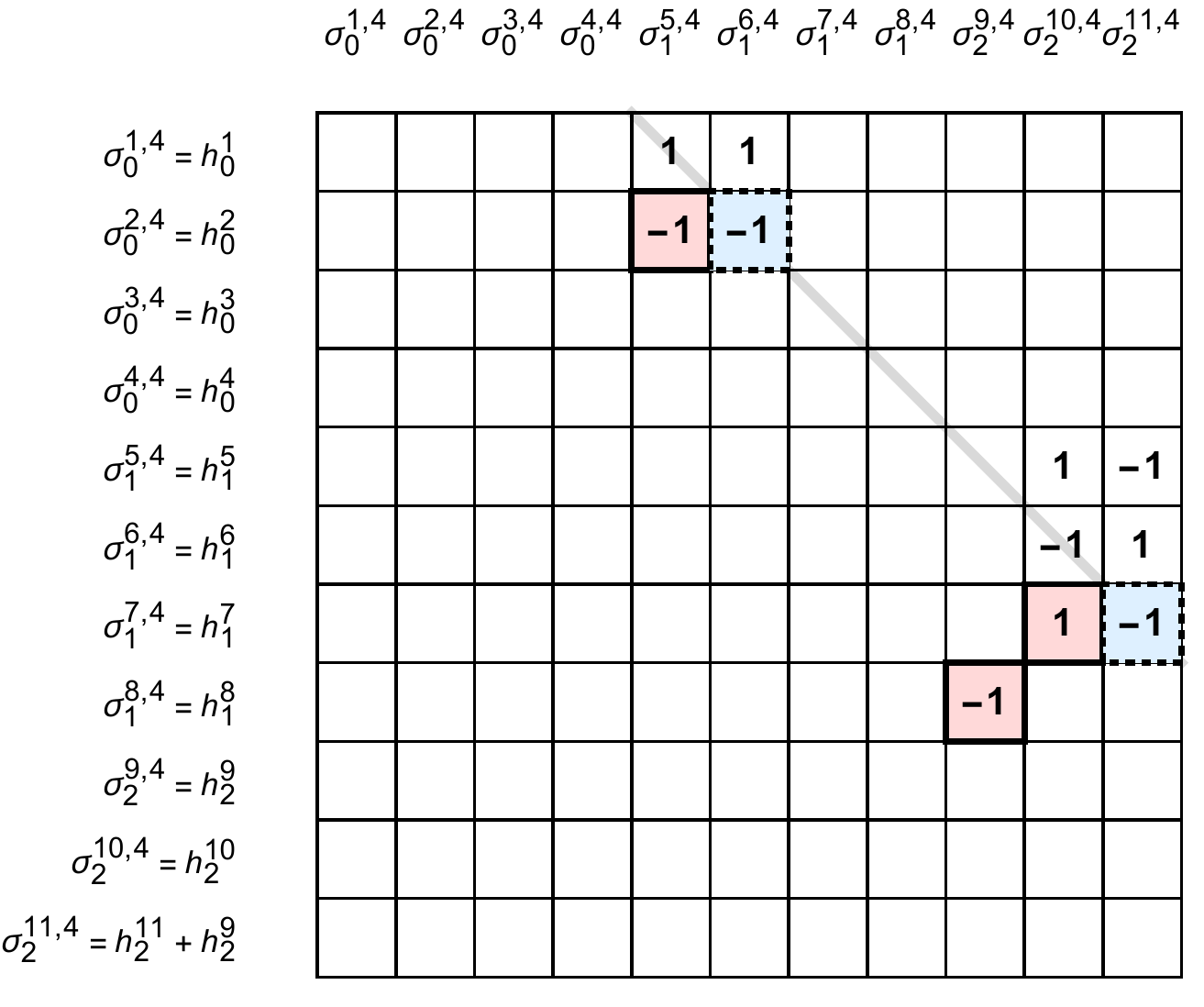}
\label{cone_matriz4}
}
\quad 
\subfloat[$\Delta^5$, sweeping 5-th diagonal.]{
\includegraphics[height=6.5cm]{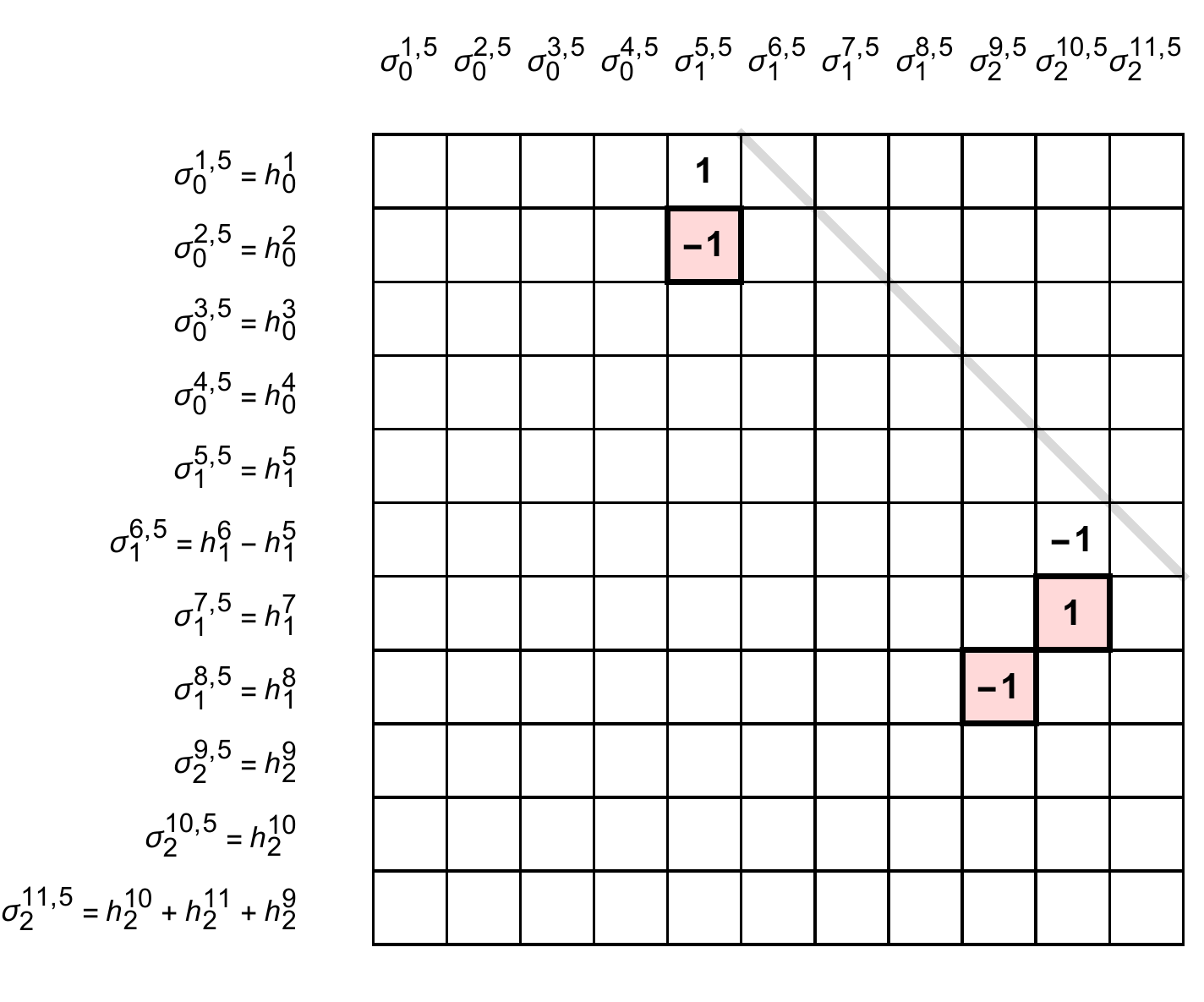}
\label{cone_matriz5}
}
\caption{Sequence of matrices produced by the SSSA.}
\label{fig001}
\end{figure}

%
%

As proven in Theorem~\ref{cancel1}, the primary pivots detect  algebraic cancellations of the modules of the spectral sequence. More specifically, 
\begin{itemize}
\item the primary pivot $\Delta^{1}_{8,9}$ detects the algebraic cancellation of the modules $E^{1}_{8}$ and $E^{1}_{7}$;
\item the primary pivot $\Delta^{2}_{7,10}$ detects the algebraic cancellation of the modules $E^{2}_{9}$ and $E^{2}_{1}$;
\item the primary pivot $\Delta^{2}_{2,5}$ detects the algebraic cancellation of the modules $E^{2}_{4}$ and $E^{2}_{1}$.
\end{itemize}
On the other hand, these algebraic cancellations are associated to dynamical cancellations by Theorem~\ref{cancel1}, namely:
\begin{itemize}
\item  the algebraic cancellation of  $E^{1}_{8}$ and $E^{1}_{7}$ determines the dynamical homotopical cancellation of the singularities $(x_1,y_4)$.
\item  the algebraic cancellation of  $E^{2}_{9}$ and $E^{2}_{1}$ determines the dynamical homotopical cancellation of the singularities $(x_2,y_3)$.
\item the algebraic cancellation of  $E^{2}_{4}$ and $E^{2}_{1}$ determines the dynamical homotopical cancellation of the singularities $(y_1,z_2)$.
\end{itemize}

Figure \ref{ex_cone1-newN000}  shows the dynamical homotopical cancellations of the pair of  singularities  $(x_1,y_4)$, $(x_2,y_3)$ and  $(y_1,z_2)$, respectively. 

\begin{figure}[H]
    \centering
        \includegraphics[width=0.8\textwidth]{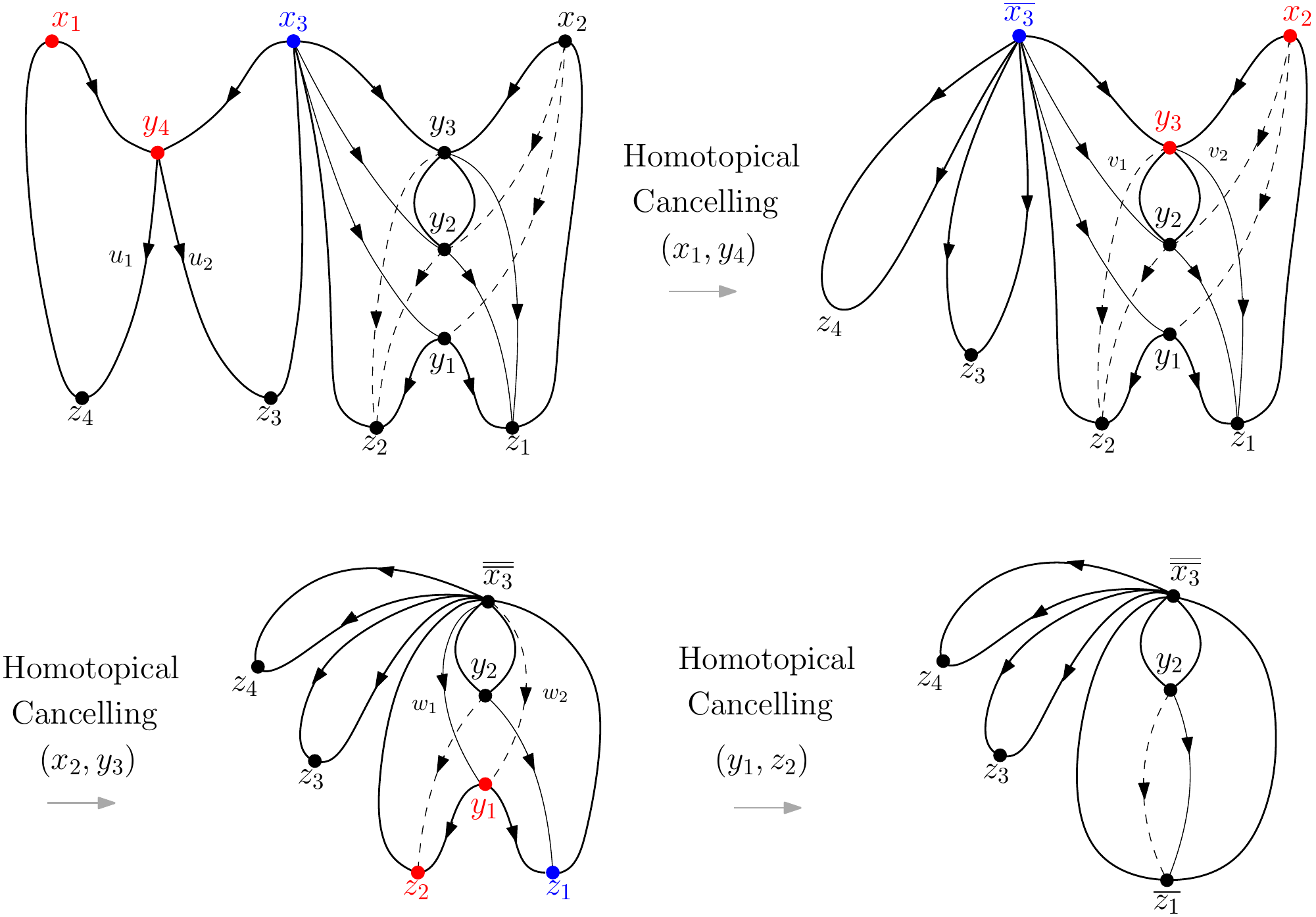}
    \caption{Homotopical Cancellation the pair of singularities $(x_1,y_4)$, $(x_2,y_3)$ and  $(y_1,z_2)$, sucessively.}\label{ex_cone1-newN000}
\end{figure}

%
%

\end{example}

\begin{example}
Consider the singular manifold 
$M\in\mathfrak{M}(\mathcal{GW})$ and the GS-flow $\varphi_{X}$ associated to a vector field   $X\in\mathfrak{X}_{\mathcal{GW}}(M)$ as in  Figure~\ref{figura_ex_whitney_man}.
The GS-chain complex associated to this flow is presented in Example~\ref{ex_w3}.
Consider a finest  filtration on  $(C_{\ast}^{\mathcal{GW}}(M,X),\Delta^{\mathcal{GW}}_{\ast})$, namely, $F_0C^{\mathcal{GW}}  = \mathbb{Z}[z_1]$, $F_1C^{\mathcal{GW}}  = \mathbb{Z}[z_1, z_2]$,   $F_2C^{\mathcal{GW}}  = \mathbb{Z}[z_1, z_2,y_1]$, $F_3C^{\mathcal{GW}}  = \mathbb{Z}[z_1, z_2,y_1,y_2]$, $F_4C^{\mathcal{GW}}  = \mathbb{Z}[z_1, z_2,y_1,y_2,y_3]$,
$F_5C^{\mathcal{GW}}  = \mathbb{Z}[z_1, z_2,y_1,y_2,y_3,x_1]$,
$F_6C^{\mathcal{GW}}  = \mathbb{Z}[z_1, z_2,y_1,y_2,y_3,x_1,x_2]$ and 
$F_{7}C^{\mathcal{GW}}  = \mathbb{Z}[z_1, z_2,y_1,y_2,y_3,x_1,x_2,x_3]$. The  spectral sequence associated to  $(C_{\ast}^{\mathcal{GW}}(M,X),\Delta^{\mathcal{GW}}_{\ast})$ enriched with the filtration $F$ is shown in Figure~\ref{specW}.

\begin{figure}[H]
      \centering
        \includegraphics[width=0.67\textwidth]{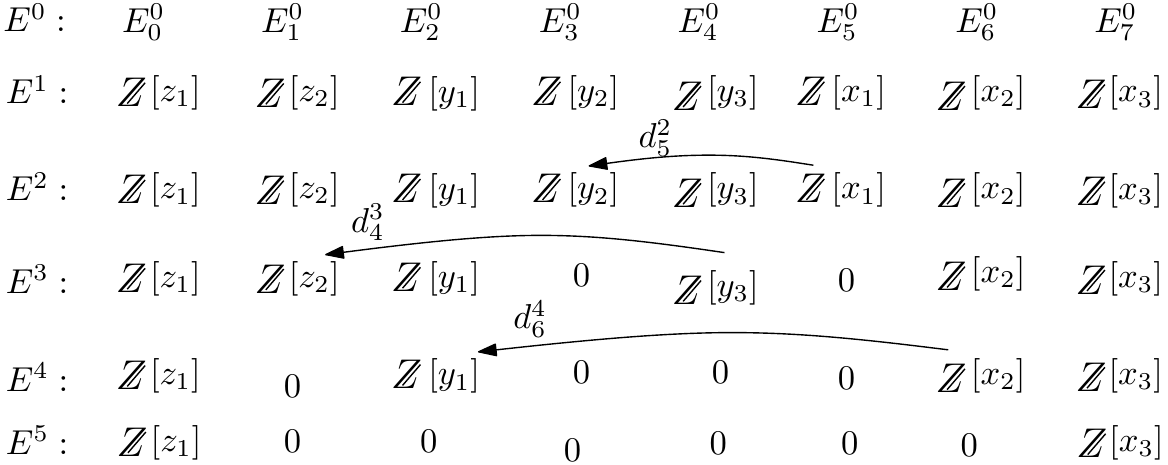}
    \caption{The spectral sequence for $(C_{\ast}^{\mathcal{GW}}(M,X),\Delta^{\mathcal{GW}}_{\ast})$ with filtration $F$.}\label{specW}
\end{figure}

Applying the  SSSA to the GS-boundary differential $\Delta^{\mathcal{GW}}$, one obtains the sequence of matrices  $\Delta^{1},\cdots,\Delta^{6}$ as in  Figures 
~\ref{whitney_matriz0},$\cdots$,~\ref{whitney_matriz5}, respectively, where the singularities are identified by   $h_0^i =z_i$, for $i=1,2$, $h_1^{i+2}=y_i$, for $i=1\dots 3$ and $h_2^{i+5}=x_i$, for $i=1\dots 3$. 


\begin{figure}[H]
\centering
\subfloat[$\Delta^1$, sweeping 1-st diagonal.]{
\includegraphics[height=6cm]{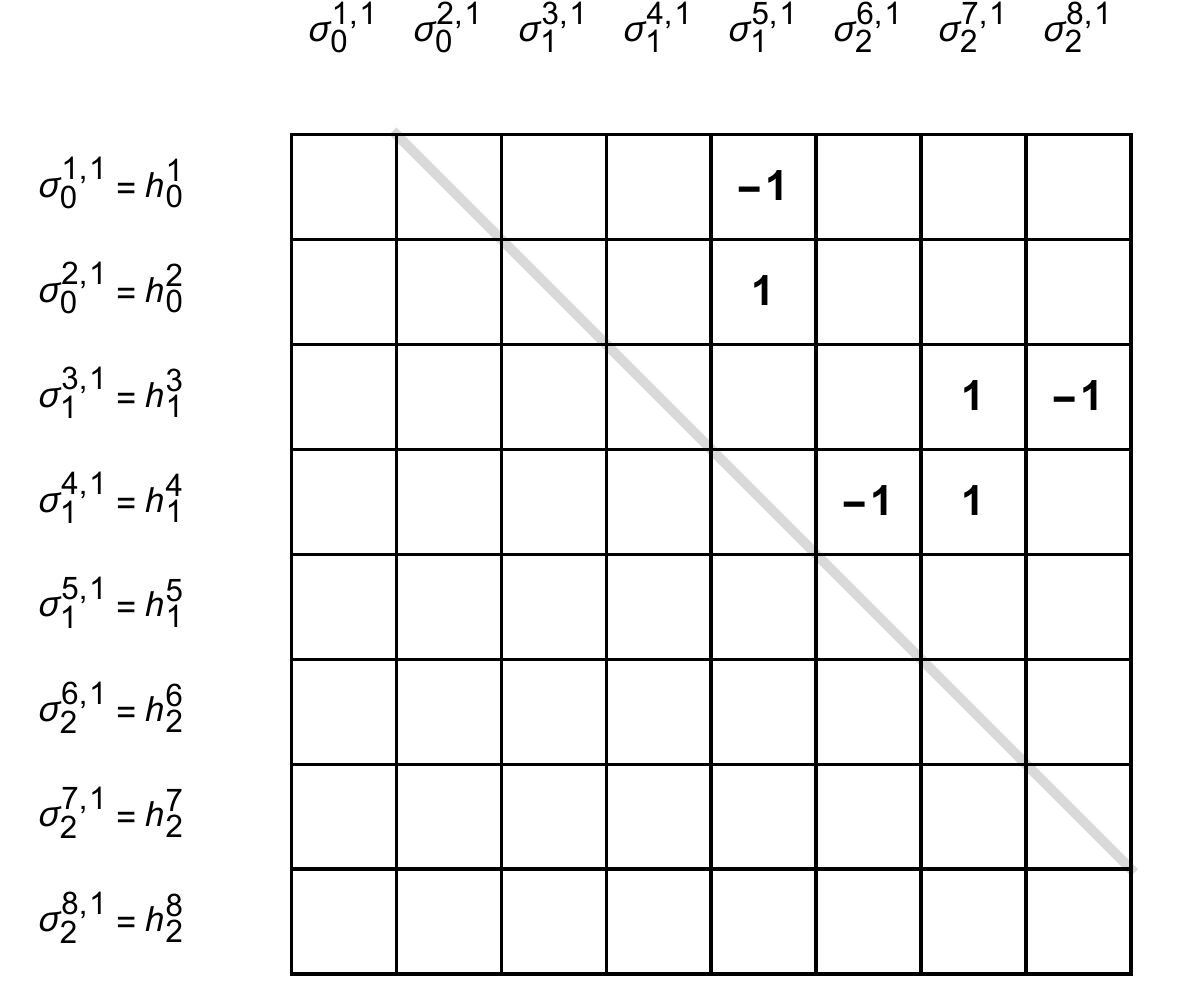}
\label{whitney_matriz0}
}
\quad 
\subfloat[$\Delta^2$, sweeping 2-nd diagonal.]{
\includegraphics[height=6cm]{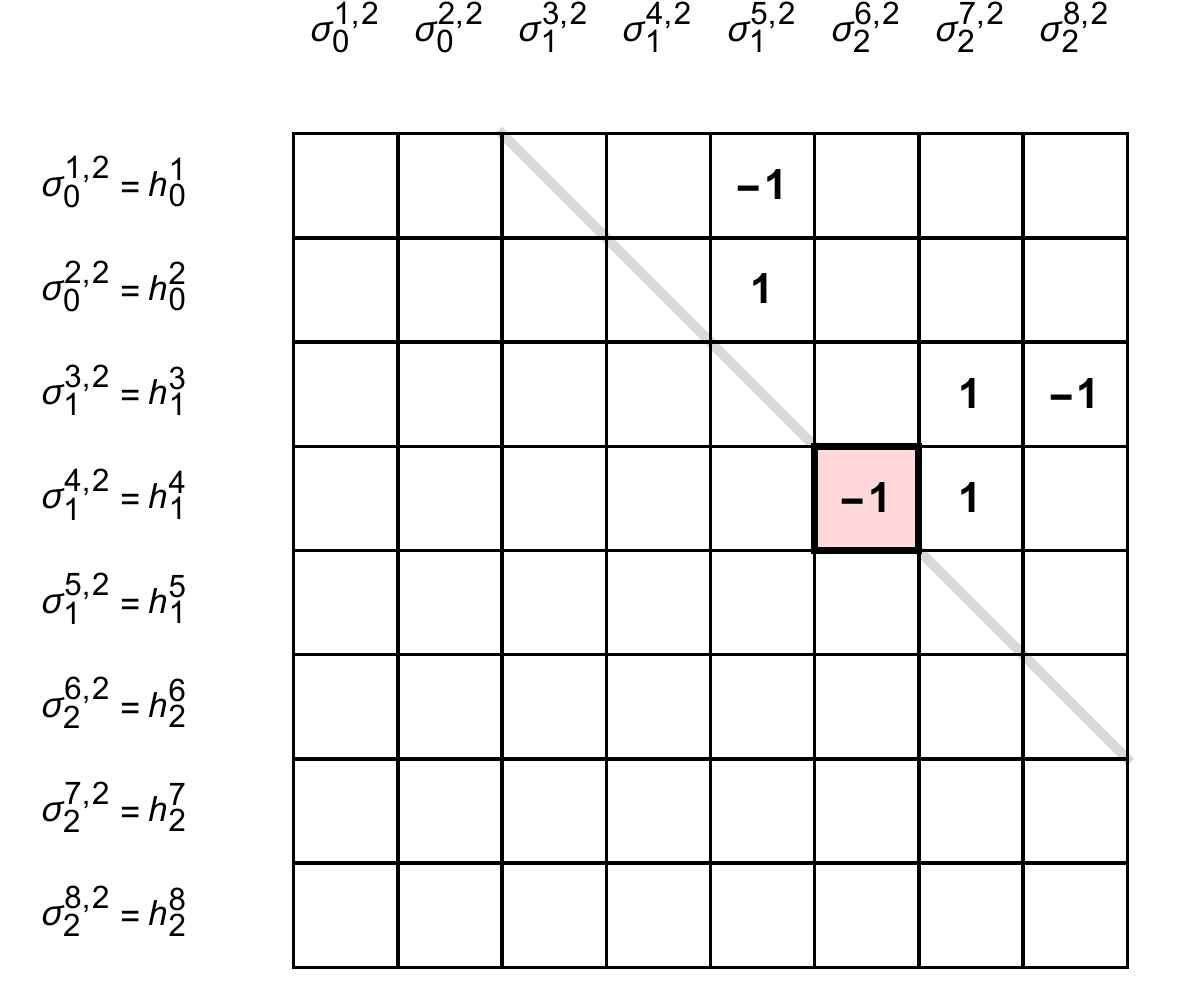}
\label{whitney_matriz1}
}\\
\subfloat[$\Delta^3$, sweeping 3-rd diagonal.]{
\includegraphics[height=6cm]{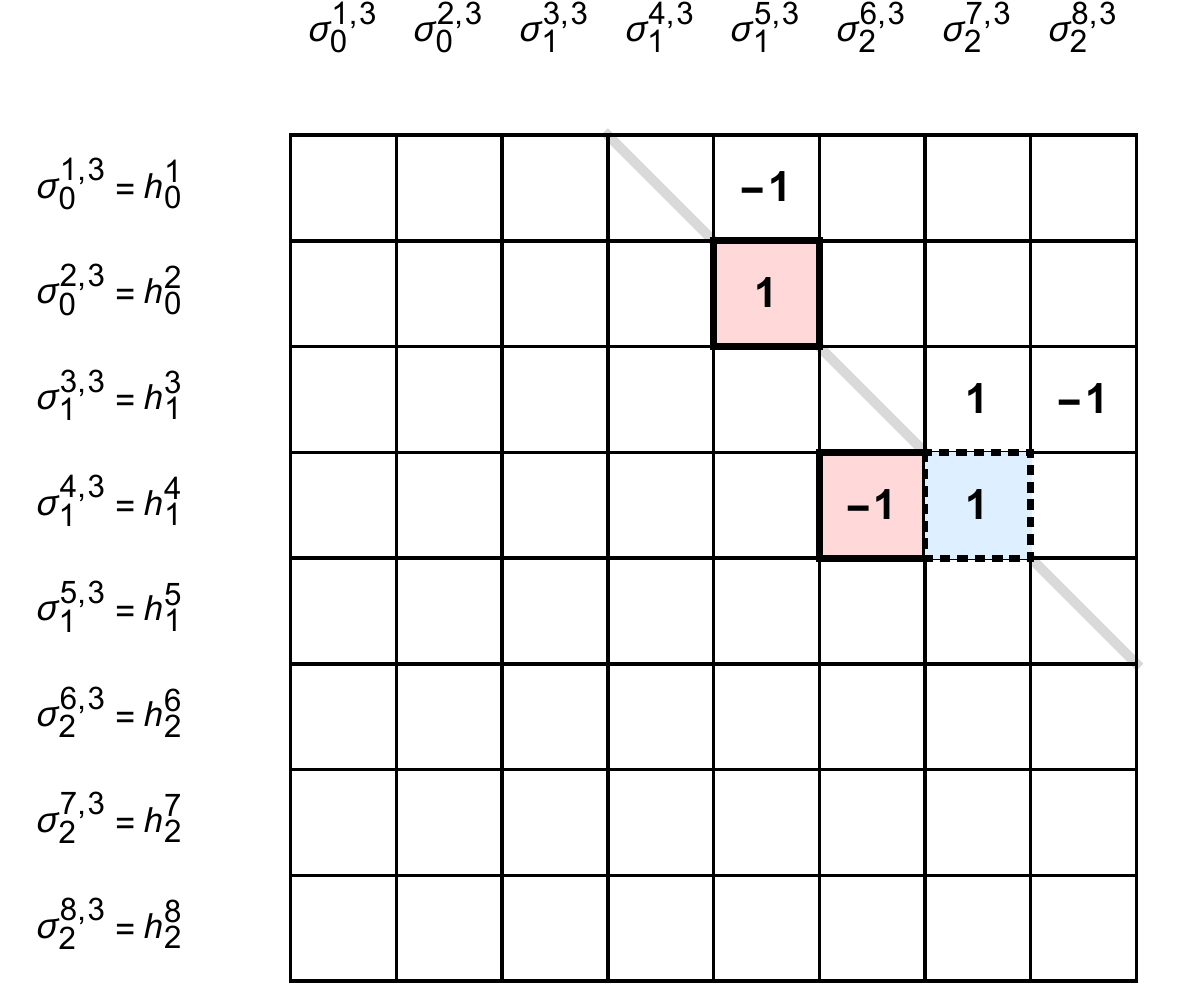}
\label{whitney_matriz2}
}
\quad 
\subfloat[$\Delta^4$, sweeping 4-th diagonal.]{
\includegraphics[height=6cm]{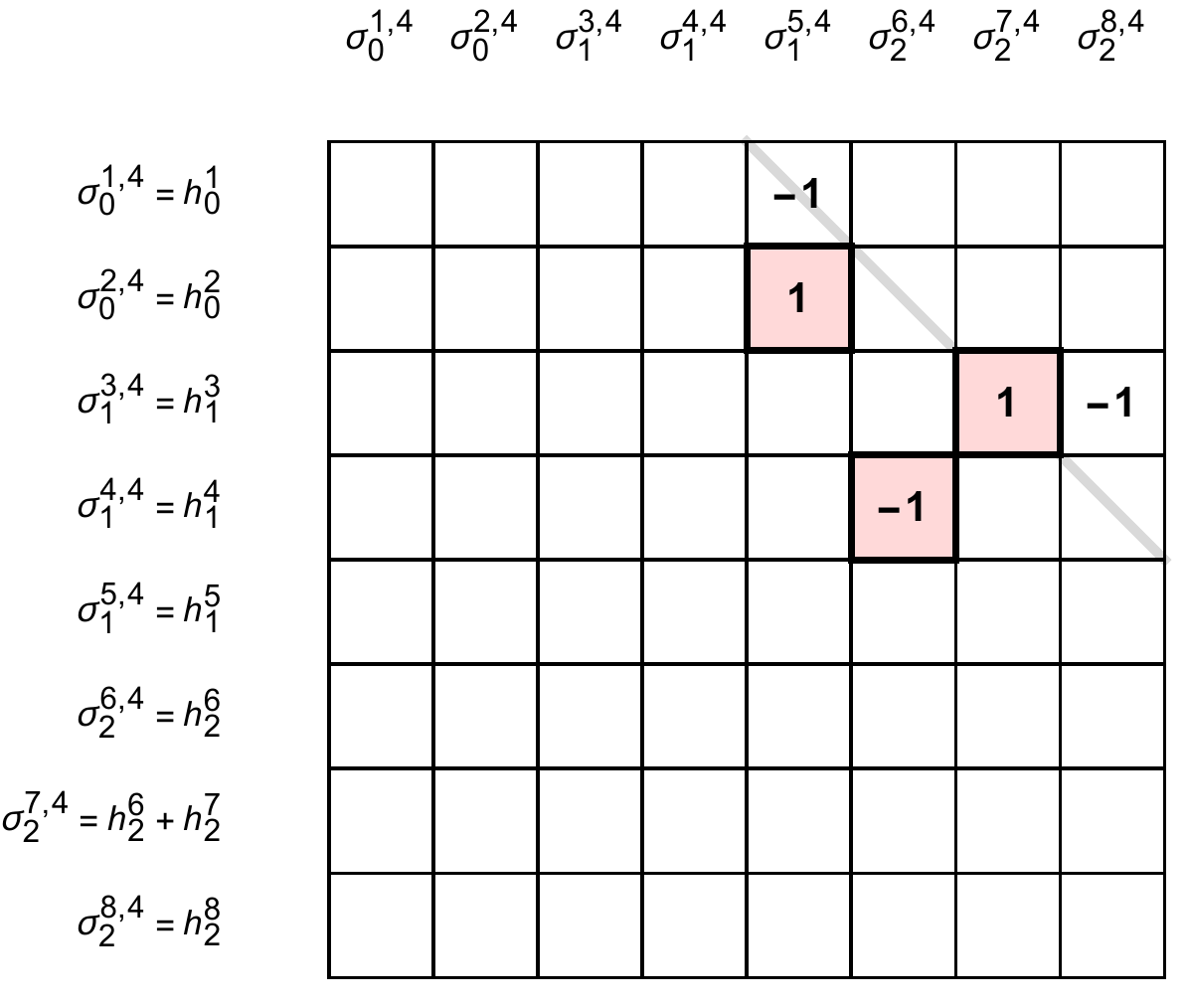}
\label{whitney_matriz3}
}\\
\subfloat[$\Delta^5$, sweeping 5-th diagonal.]{
\includegraphics[height=6cm]{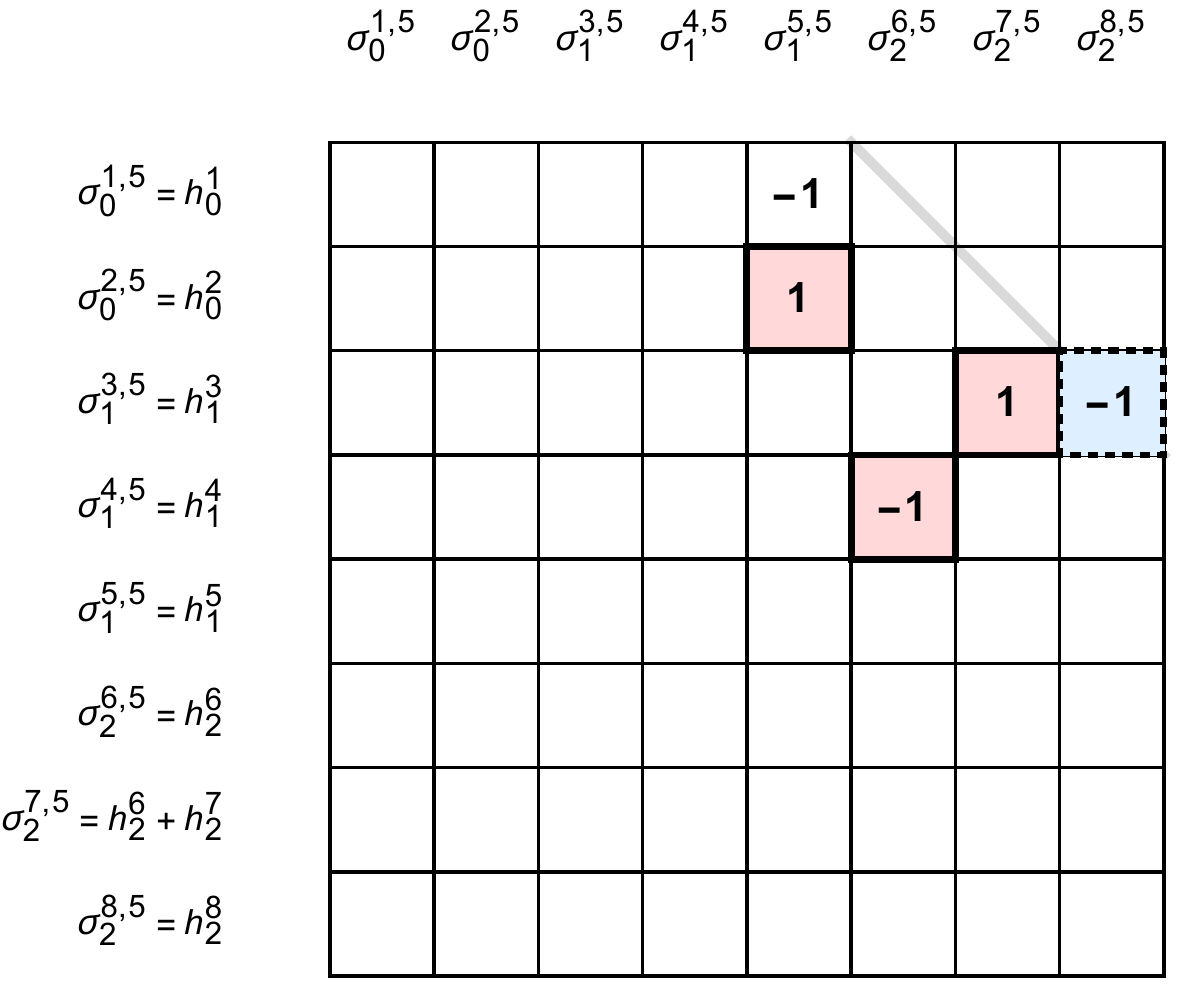}
\label{whitney_matriz4}
}
\quad 
\subfloat[$\Delta^6$, sweeping 6-th diagonal.]{
\includegraphics[height=6cm]{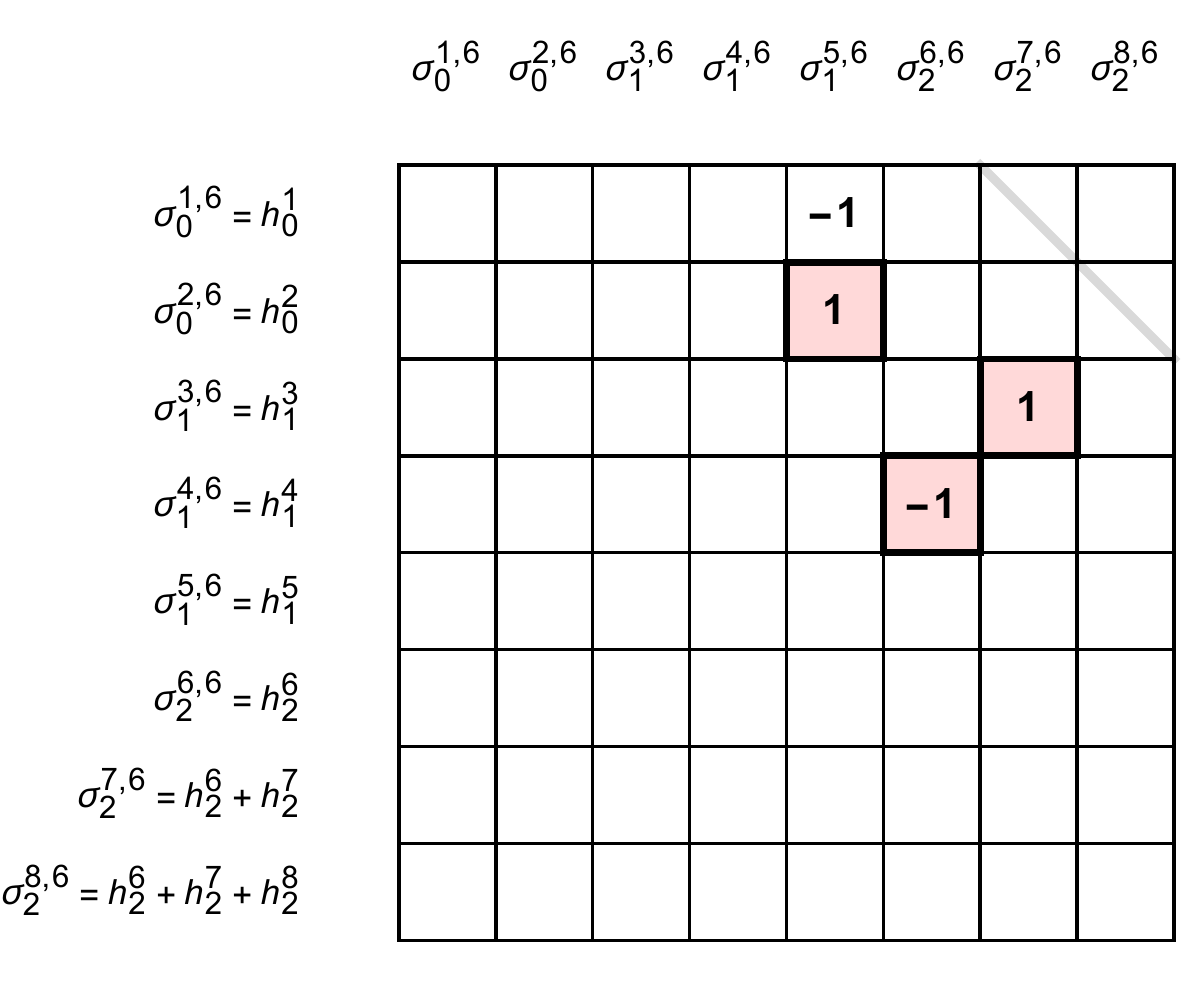}
\label{whitney_matriz5}
}
\caption{Sequence of matrices produced by the SSSA.}
\label{fig01}
\end{figure}

%
%

As proven in Theorem~\ref{cancel1}, the primary pivots detect  algebraic cancellations of the modules of the spectral sequence. More specifically, 
\begin{itemize}
\item the primary pivot $\Delta^{2}_{4,6}$ detects the algebraic cancellation of the modules $E^{2}_{5}$ and $E^{2}_{3}$;
\item the primary pivot $\Delta^{3}_{2,5}$ detects the algebraic cancellation of the modules $E^{3}_{4}$ and $E^{3}_{1}$;
\item the primary pivot $\Delta^{4}_{3,7}$ detects the algebraic cancellation of the modules $E^{4}_{6}$ and $E^{4}_{2}$.
\end{itemize}
On the other hand, these algebraic cancellations are associated to dynamical homotopical cancellations by Theorem~\ref{cancel1}, namely:
\begin{itemize}
\item  the algebraic cancellation of  $E^{1}_{5}$ and $E^{1}_{3}$ determines the dynamical homotopical cancellation of the singularities $(x_1,y_2)$.
\item  the algebraic cancellation of  $E^{3}_{4}$ and $E^{3}_{1}$ determines the dynamical homotopical cancellation of the singularities $(y_3,z_2)$.
\item the algebraic cancellation of $E^{4}_{6}$ and $E^{4}_{2}$ determines the dynamical homotopical cancellation of the singularities $(x_2,y_1)$.
\end{itemize}

Figure \ref{ex_whitney1-new2N000}  shows the dynamical cancellations of the pair of  singularities  $(x_1,y_2)$, $(y_3,z_2)$ and  $(x_2,y_1)$, respectively. 

\begin{figure}[H]
    \centering
        \includegraphics[width=0.85\textwidth]{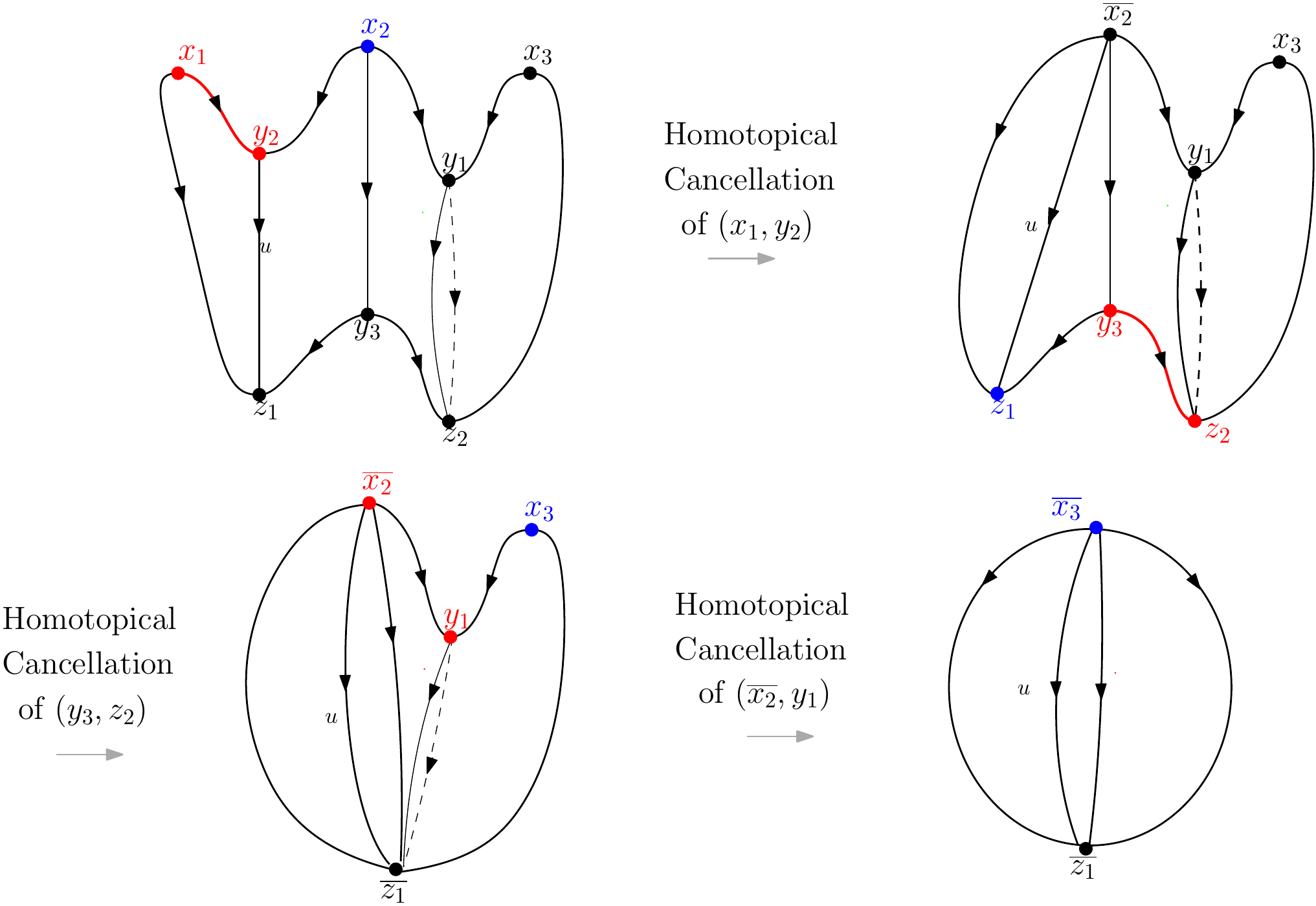}
    \caption{Homotopical cancellation  the pair of singularities $(x_1,y_2)$, $(y_3,z_2)$ and  $(x_2,y_1)$, sucessively.}\label{ex_whitney1-new2N000}
\end{figure}

%
%

\end{example}

\begin{example}
Consider the singular manifold $M\in\mathfrak{M}(\mathcal{GD})$ and the GS-flow $\varphi_{X}$ associated to a vector field   $X\in\mathfrak{X}_{\mathcal{GD}}(M)$ as in  Figure~\ref{fig_caracteristico_d1}.  
The GS-chain complex associated to $(M,X)$ is presented in Example~\ref{ex_d}.
The GS-boundary operator $\Delta^{\mathcal{GS}}_{\ast}$ is  given by the matrix in Figure~\ref{fig_flow_d}.  

Consider a finest  filtration on  $(C_{\ast}^{\mathcal{GD}}(M,X),\Delta^{\mathcal{GD}}_{\ast})$, namely, $
F_0C^{\mathcal{GD}}  = \mathbb{Z}[z_1^e]$, 
$F_1C^{\mathcal{GS}}  \setminus  F_0C^{\mathcal{GD}} = \mathbb{Z}[z_1^i]$, 
$F_2C^{\mathcal{GD}}  \setminus F_1C^{\mathcal{GD}} = \mathbb{Z}[z_2^e]$, $F_3C^{\mathcal{GD}}  \setminus F_2C^{\mathcal{GD}} = \mathbb{Z}[z_2^i]$, 
$F_4C^{\mathcal{GD}}  \setminus  F_3C^{\mathcal{GD}}= \mathbb{Z}[y_1^e]$, 
$F_5C^{\mathcal{GD}}  \setminus F_4C^{\mathcal{GD}}= \mathbb{Z}[y_1^i]$, 
$F_6C^{\mathcal{GD}}  \setminus F_5C^{\mathcal{GD}}= \mathbb{Z}[y_2^e]$,
$F_7C^{\mathcal{GD}} \setminus  F_6C^{\mathcal{GD}}= \mathbb{Z}[y_2^i]$,
$F_8C^{\mathcal{GD}}  \setminus F_7C^{\mathcal{GD}}= \mathbb{Z}[y_3]$,
$F_9C^{\mathcal{GD}}  \setminus F_8C^{\mathcal{GD}}= \mathbb{Z}[x_1]$, 
$F_{10}C^{\mathcal{GD}}   \setminus F_9C^{\mathcal{GD}}= \mathbb{Z}[x_2]$, $F_{11}C^{\mathcal{GD}}   \setminus F_10C^{\mathcal{GD}}= \mathbb{Z}[x_3]$, $F_{12}C^{\mathcal{GD}}   \setminus F_11C^{\mathcal{GD}}= \mathbb{Z}[x_4]$ and  $F_{13}C^{\mathcal{GD}}   \setminus F_12C^{\mathcal{GD}}= \mathbb{Z}[x_5]$.
The  spectral sequence associated to  $(C_{\ast}^{\mathcal{GD}}(M,X),\Delta^{\mathcal{GD}}_{\ast})$ enriched with the filtration $F$ is shown in Figure~\ref{spec}.

\begin{figure}[H]
      \centering
        \includegraphics[width=0.95\textwidth]{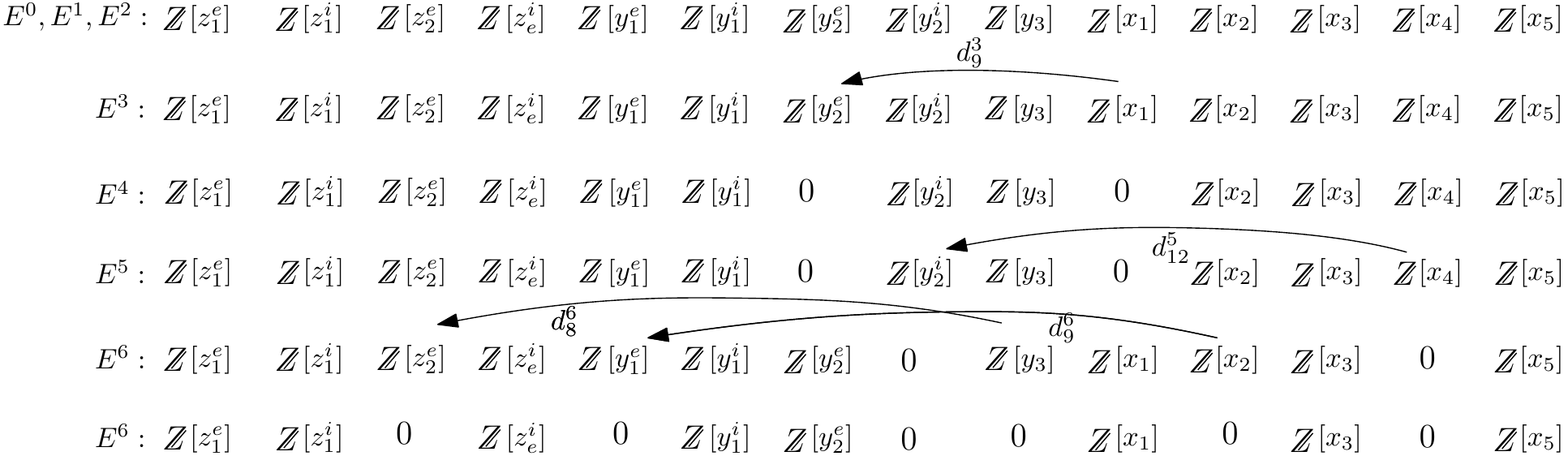}
    \caption{The spectral sequence for $(C_{\ast}^{\mathcal{GD}}(M,X),\Delta^{\mathcal{GD}}_{\ast})$ with filtration $F$.}\label{specD}
\end{figure}

Applying the  SSSA to the GS-boundary differential $\Delta^{\mathcal{GD}}$, one obtains the sequence of matrices  $\Delta^{1},\cdots,\Delta^{8}$ as in  Figures 
~\ref{duplo_matriz3},$\cdots$,~\ref{duplo_matriz8}, respectively, where the singularities are identified by  $h_0^1=z_1^e$, $h_0^2=z_1^i$, $h_0^3=z_2^e$, $h_0^4=z_2^i$, $h_1^5=y_1^e$, $h_1^6=y_1^i$, $h_1^7=y_2^e$, $h_1^8=y_2^i$, $h_1^9=y_3$ and $h_2^{j+9}=x_j$, for $j=1\dots 5$.

%

\begin{figure}[H]
\centering
\subfloat[$\Delta^3$, sweeping 3-rd diagonal.]{
\includegraphics[height=6.7cm]{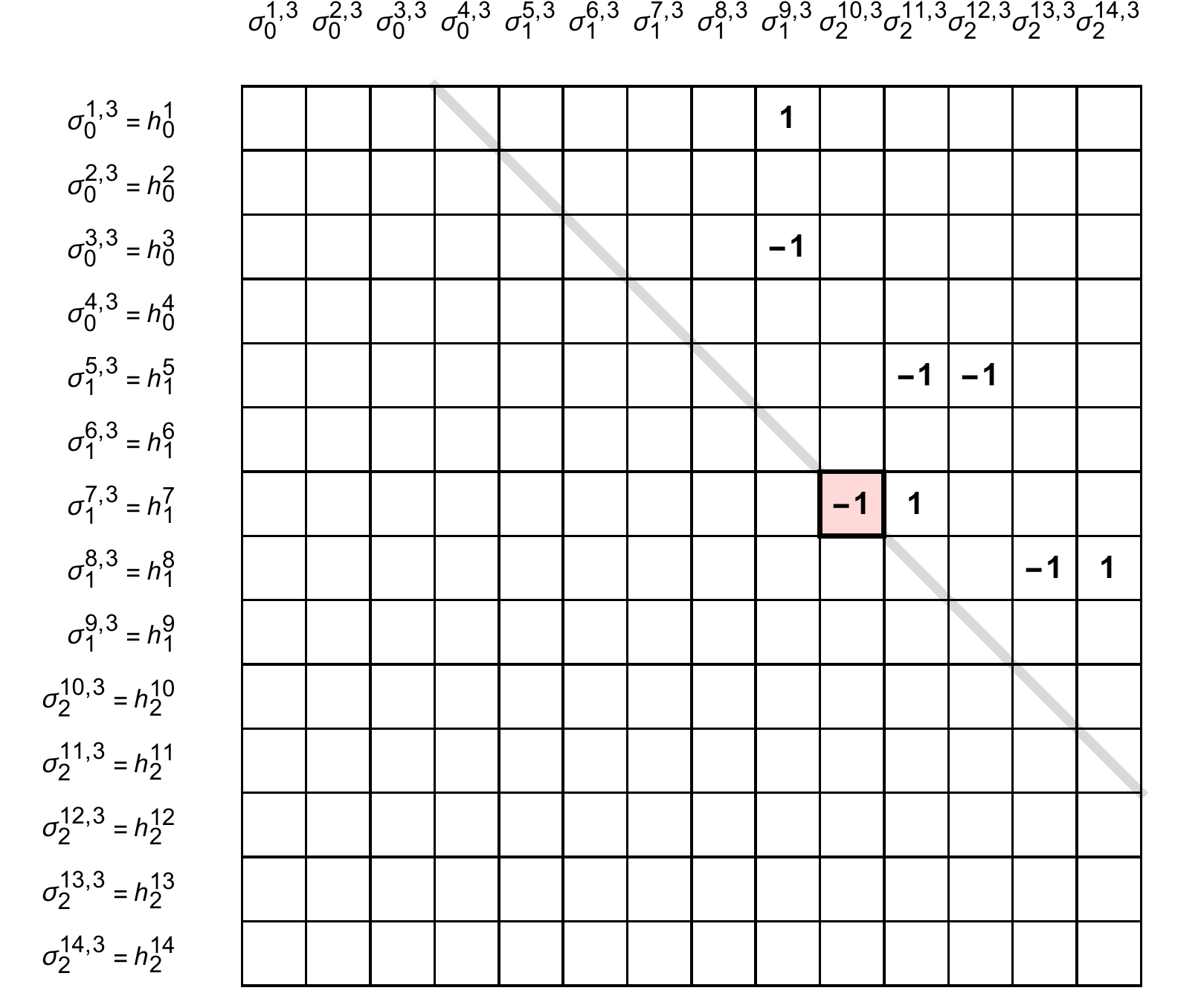}
\label{duplo_matriz3}
}
\quad 
\subfloat[$\Delta^4$, sweeping 4-th diagonal.]{
\includegraphics[height=6.7cm]{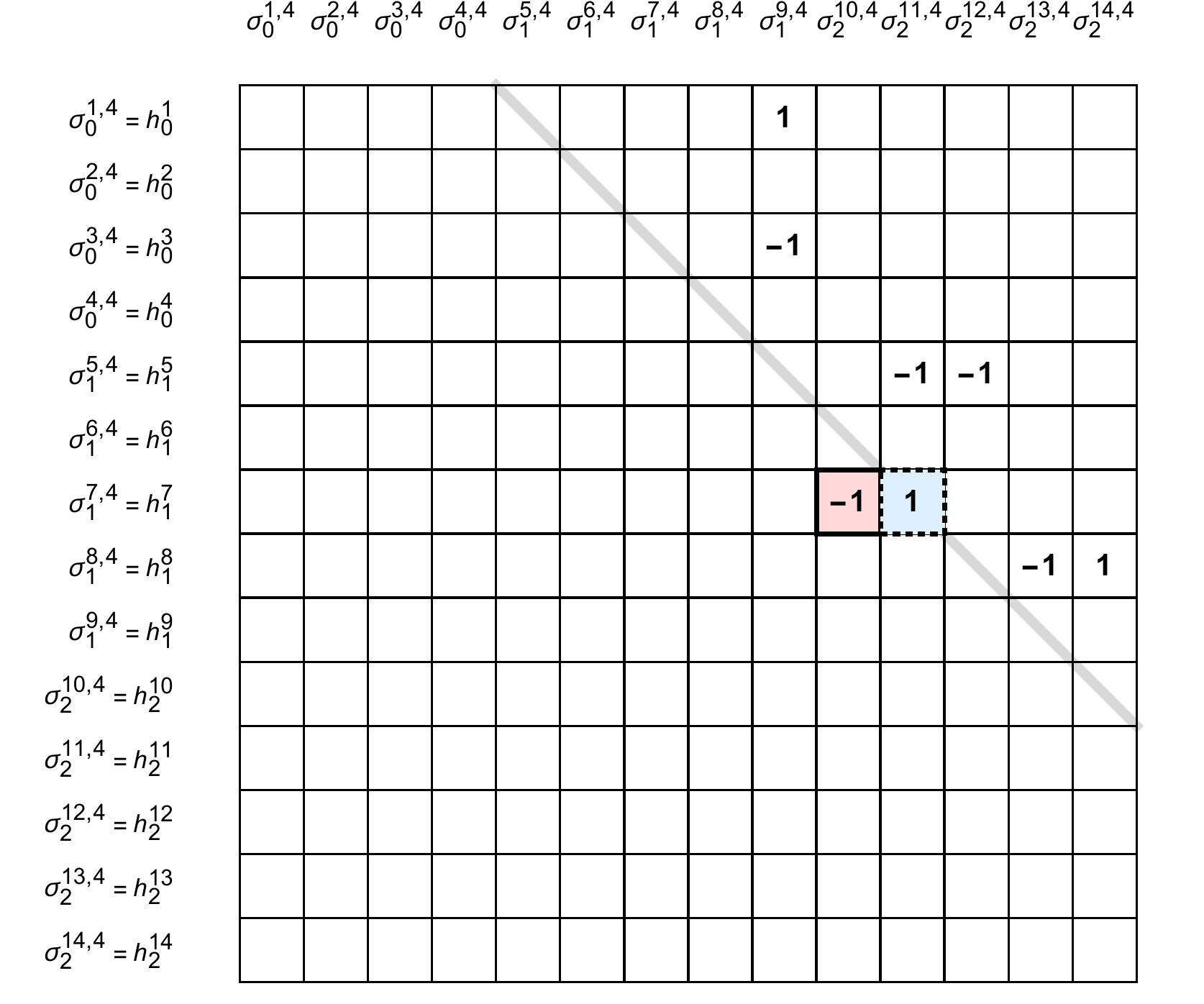}
\label{duplo_matriz4}
}\\
\subfloat[$\Delta^5$, sweeping 5-th diagonal..]{
\includegraphics[height=6.6cm]{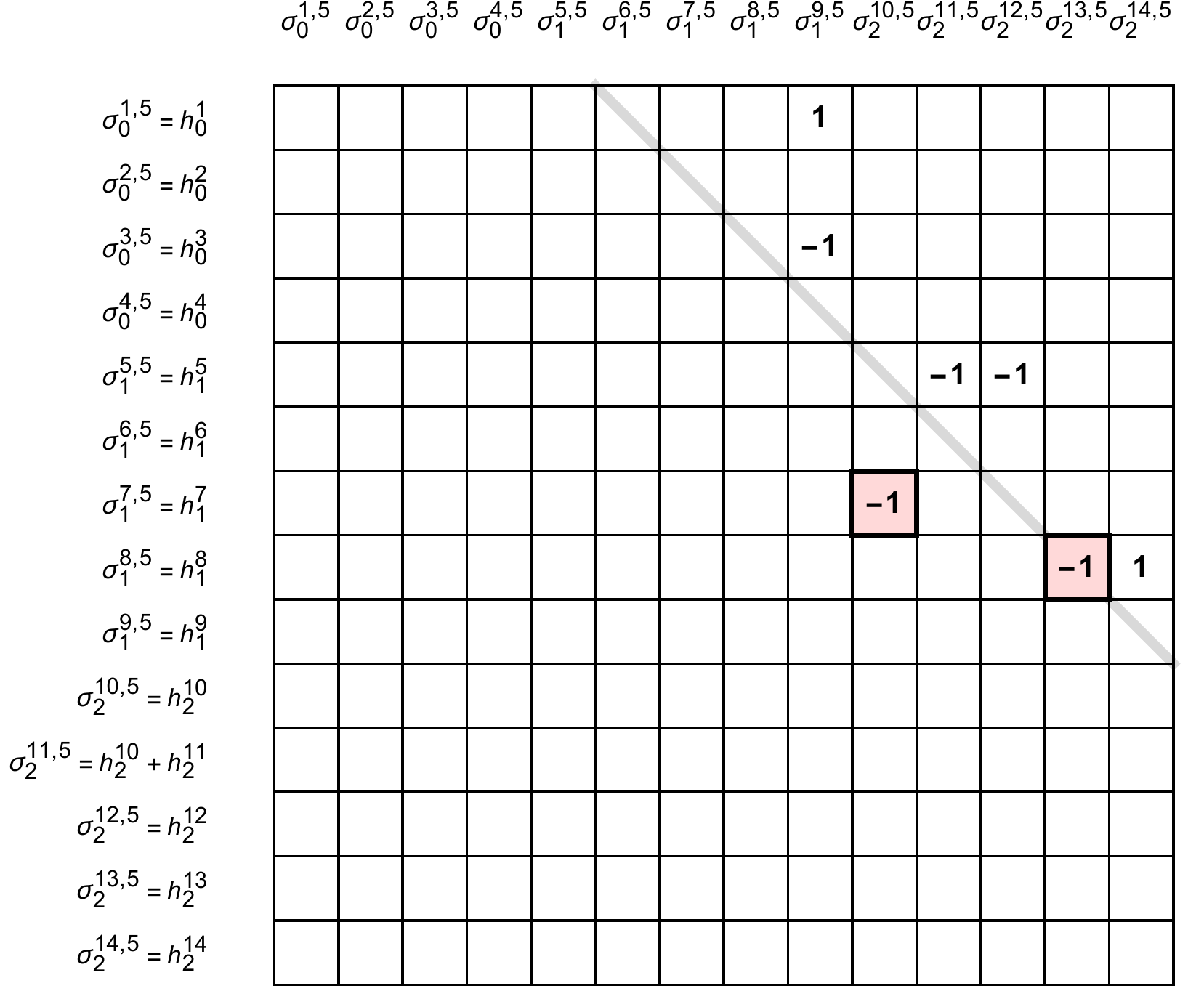}
\label{duplo_matriz5}
}
\quad 
\subfloat[$\Delta^6$, sweeping 6-th diagonal.]{
\includegraphics[height=6.6cm]{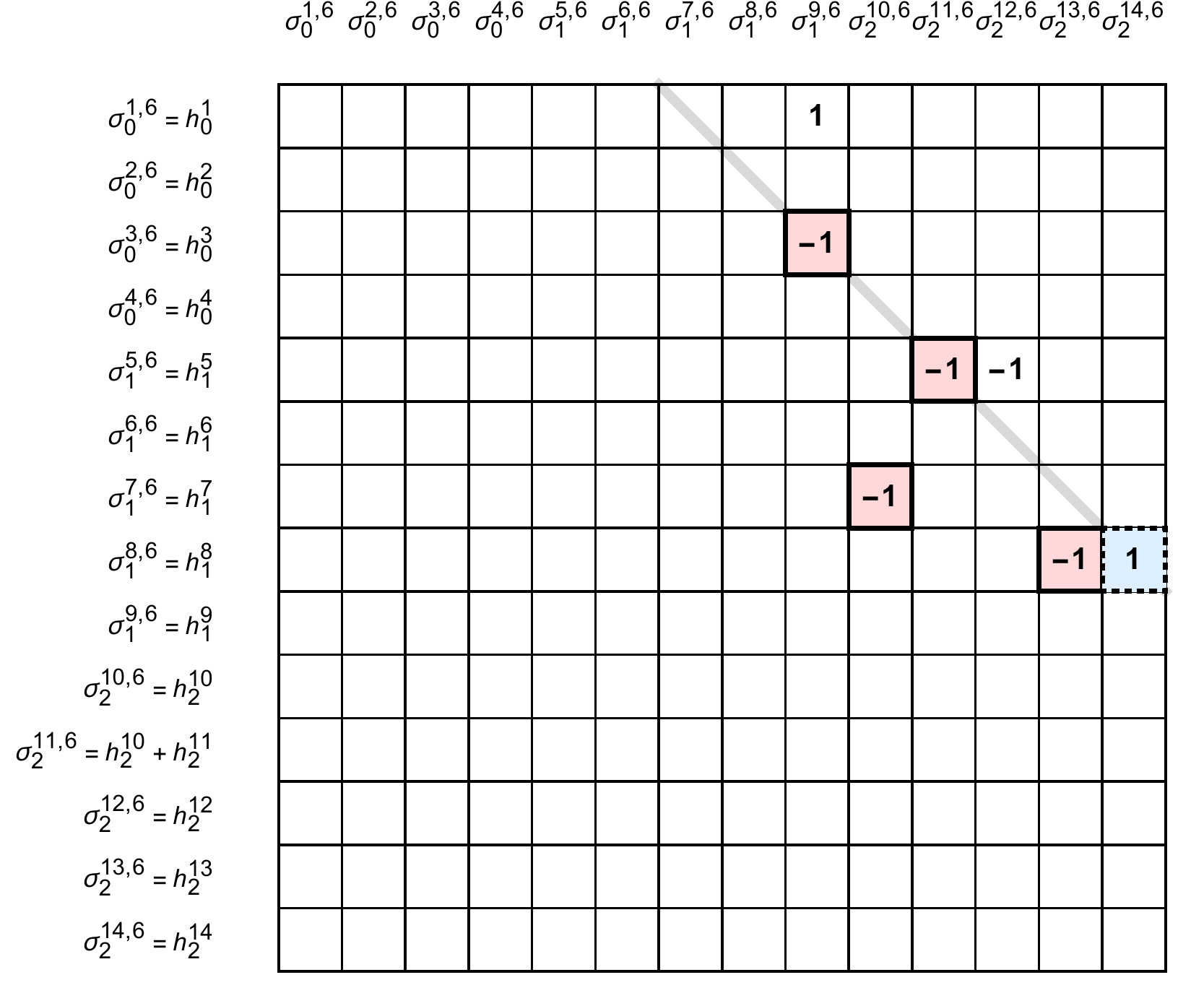}
\label{duplo_matriz6}
}\\
\subfloat[$\Delta^7$, sweeping 7-th diagonal.]{
\includegraphics[height=6.7cm]{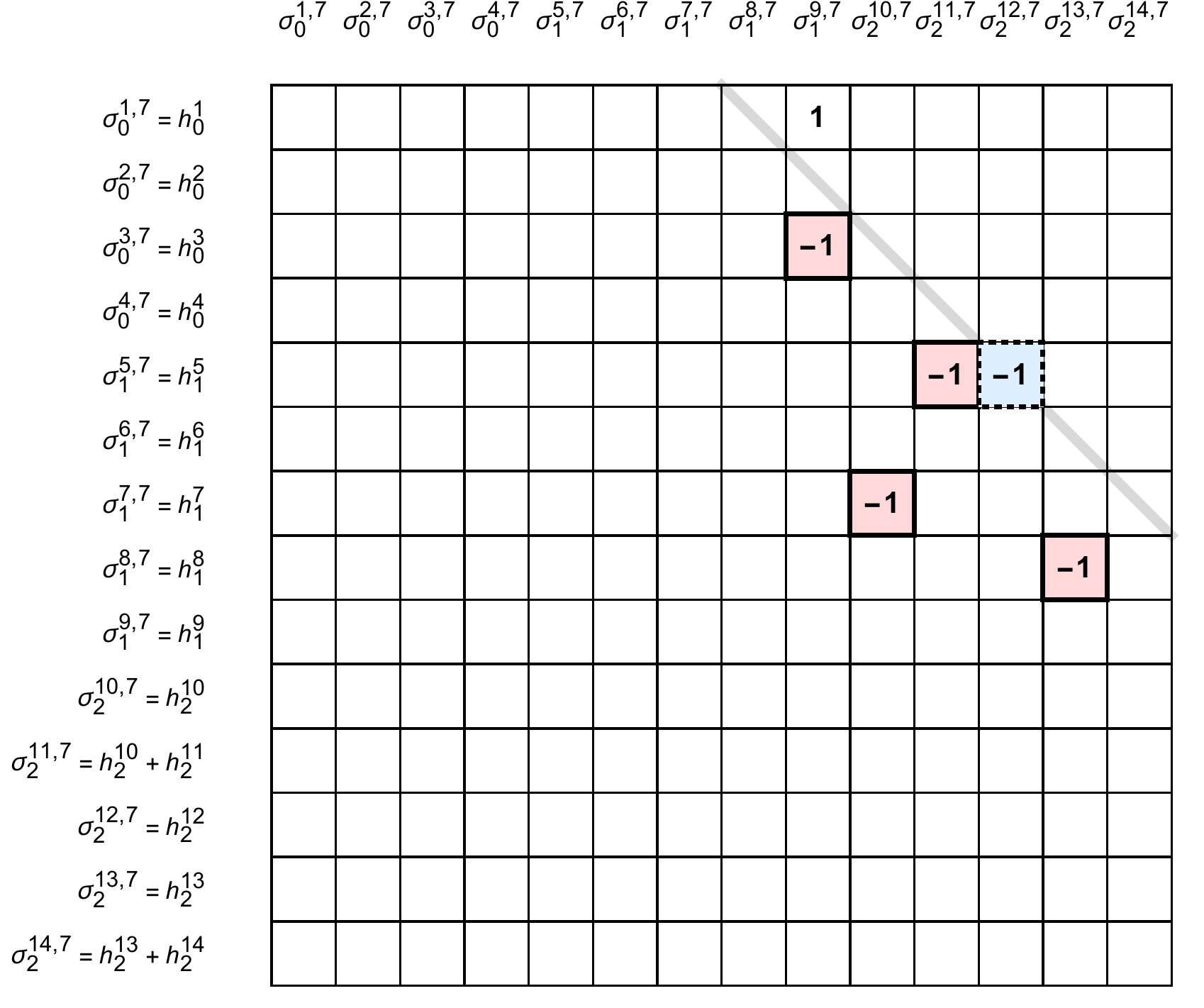}
\label{duplo_matriz7}
}
\quad 
\subfloat[$\Delta^8$, sweeping 8-th diagonal.]{
\includegraphics[height=6.7cm]{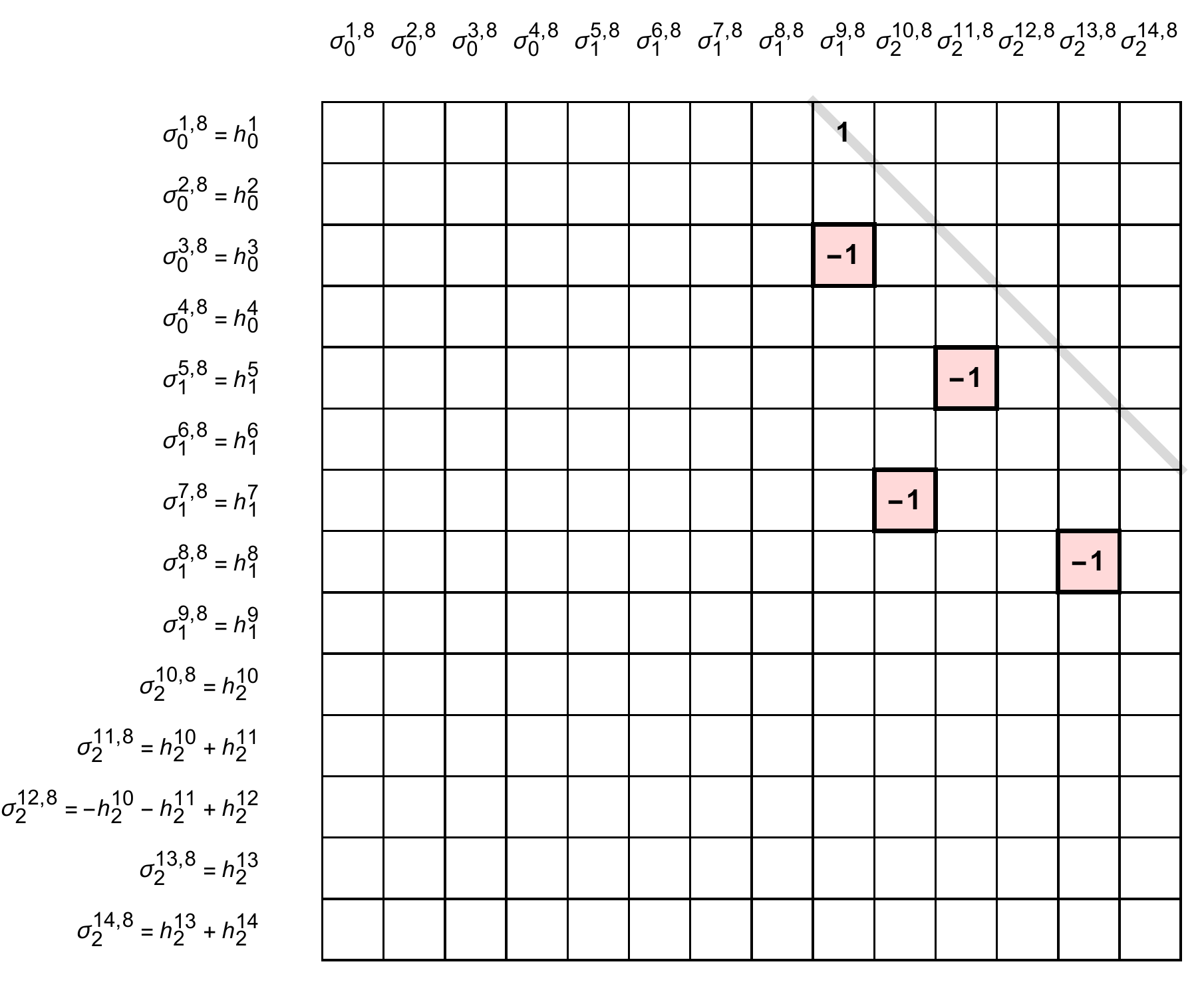}
\label{duplo_matriz8}
}
\caption{Sequence of matrices produced by the SSSA.}
\label{fig01}
\end{figure}

%
%
%

As proven in Theorem~\ref{cancel1}, the primary pivots detect  algebraic cancellations of the modules of the spectral sequence. More specifically, 
\begin{itemize}
\item the primary pivot $\Delta^{3}_{7,10}$ detects the algebraic cancellation of the modules $E^{3}_{9}$ and $E^{3}_{6}$;
\item the primary pivot $\Delta^{5}_{8,13}$ detects the algebraic cancellation of the modules $E^{5}_{12}$ and $E^{5}_{7}$;
\item the primary pivot $\Delta^{6}_{3,9}$ detects the algebraic cancellation of the modules $E^{6}_{8}$ and $E^{6}_{2}$;
\item the primary pivot $\Delta^{6}_{5,11}$ detects the algebraic cancellation of the modules $E^{6}_{10}$ and $E^{6}_{4}$.
\end{itemize}
On the other hand, these algebraic cancellations are associated to dynamical homotopical cancellations, namely:
\begin{itemize}
\item  the algebraic cancellation of  $E^{3}_{9}$ and $E^{3}_{6}$ determines the dynamical cancellation of the singularities $(y_2^e,x_1)$.
\item  the algebraic cancellation of  $E^{5}_{12}$ and $E^{5}_{7}$ determines the dynamical homotopical cancellation of the singularities $(y_2^i,x_4)$.
\item the algebraic cancellation of $E^{6}_{8}$ and $E^{6}_{2}$ determines the dynamical homotopical cancellation of the singularities $(z_2^e,y_3)$.
\item the algebraic cancellation of $E^{6}_{10}$ and $E^{6}_{4}$ determines the dynamical homotopical cancellation of the singularities $(y_1^e,\bar{x}_3)$.
\end{itemize}

Figure \ref{ex_duplo1-newN000}  shows the dynamical cancellation of the pair of  singularities  $(y_2^e,x_1)$, $(y_2^i,x_4)$, $(z_2^e,y_3)$ and $(y_1^e,\bar{x}_3)$, respectively.

\begin{figure}[H]
    \centering
        \includegraphics[width=0.83\textwidth]{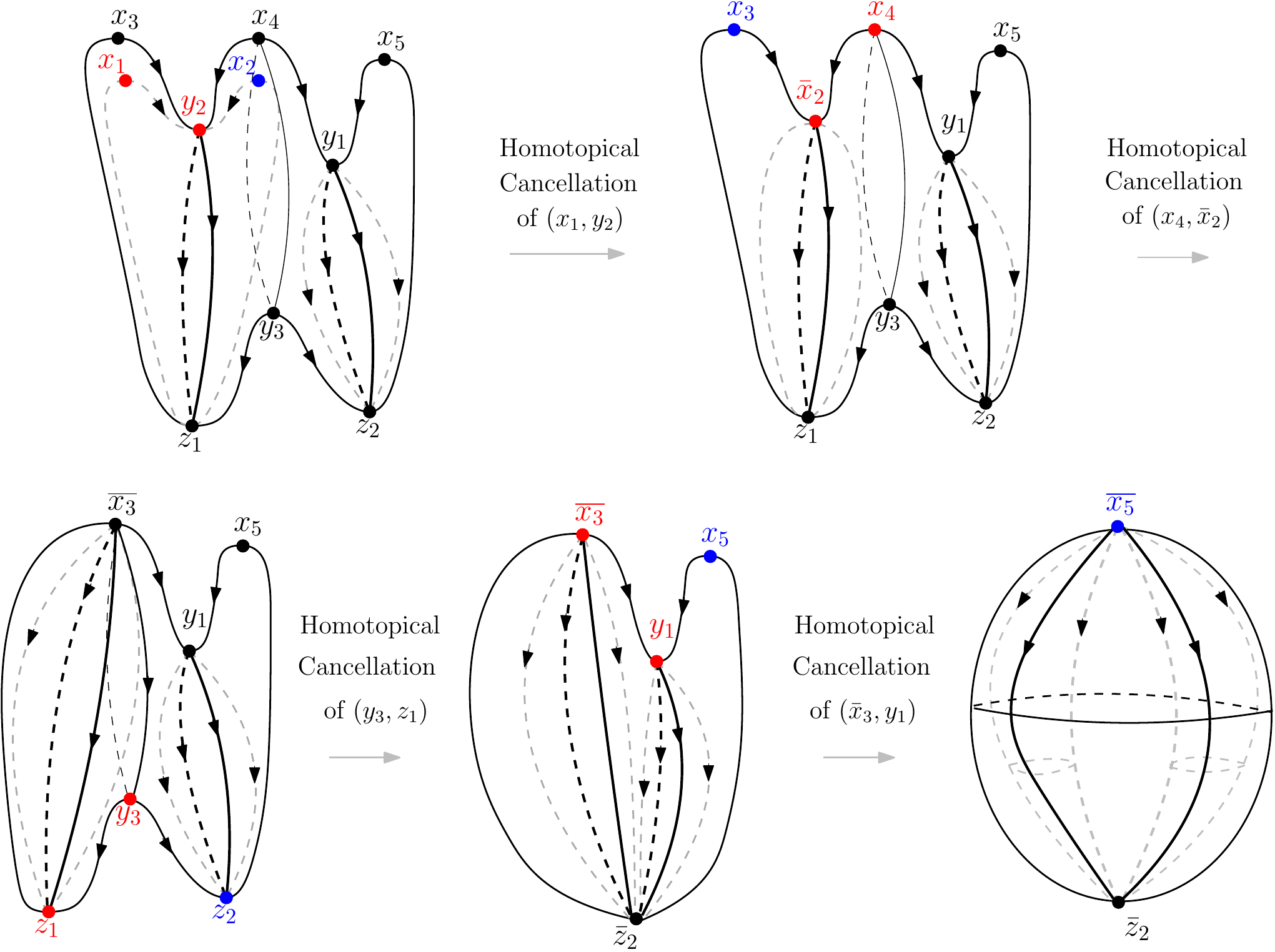}
    \caption{Homotopical cancellation of the pair of generators $(y_2^e,x_1)$, $(y_2^i,x_4)$, $(z_2^e,y_3)$ and $(y_1^e,\bar{x}_3)$, sucessively.}\label{ex_duplo1-newN000}
\end{figure}

%
%
%


\end{example}


\vspace{1cm}

{\sc {\footnotesize
\noindent  D. V. S. Lima - CMCC, Universidade Federal do ABC, Santo André, SP, Brazil.

\noindent e-mail: dahisy.lima@ufabc.edu.br

\vspace{0.2cm}

\noindent  S. A. Raminelli - IMECC, Universidade Estadual de Campinas, Campinas, SP, Brazil. 

\noindent e-mail: stefakemi@outlook.com.br
\vspace{0.2cm}

\noindent  K. A. de Rezende - IMECC, Universidade Estadual de Campinas, Campinas, SP, Brazil.

\noindent e-mail: ketty@ime.unicamp.br

}}

\end{document}